\documentclass[final, reqno, 12pt
]{uofcthesis}
\usepackage{amsmath}
\usepackage{amsfonts}
\usepackage{amssymb}
\usepackage{amsthm}

\newtheorem{theorem}{Theorem}[chapter]
\newtheorem{lemma}[theorem]{Lemma}
\newtheorem{proposition}[theorem]{Proposition}
\newtheorem{corollary}[theorem]{Corollary}
\newtheorem{observation}[theorem]{Remark}
\newtheorem{assumption}{Assumption}

\newtheorem{definition}{Definition}[chapter]
\newcommand{\set}{\mathbb}
\newcommand{\dl}{\nabla}
\renewcommand{\frak}{\mathfrak}

\newcommand{\mc}{\mathcal}
\newcommand{\be}{\begin{equation}}
\newcommand{\ee}{\end{equation}}
\newcommand{\bee}{\begin{align}}
\newcommand{\eee}{\end{align}}

\newcommand{\ba}{\begin{array}}
\newcommand{\ds}{\displaystyle}

\newcommand{\ea}{\end{array}}
\newcommand{\bpm}{\begin{pmatrix}}
\newcommand{\epm}{\end{pmatrix}}
\newcommand{\lb}{\label}
\DeclareMathOperator{\sgn}{sgn}

\DeclareMathOperator{\supp}{supp}
\DeclareMathOperator{\Ran}{Ran}
\DeclareMathOperator{\Ker}{Ker}
\DeclareMathOperator{\Rere}{Re}
\DeclareMathOperator{\Imim}{Im}
\DeclareMathOperator{\Dom}{Dom}
\DeclareMathOperator*{\esssup}{ess\; sup}

\newcommand{\ov}{\overline}
\newcommand{\dd}{{\,}{d}}

\renewcommand{\Re}{\Rere}
\renewcommand{\Im}{\Imim}

\newcommand{\Dil}{\mathrm{Dil}}

\title
{A Critical Centre-Stable Manifold for the Cubic Focusing Schr\"{o}dinger Equation in Three Dimensions}
\author{Marius Beceanu}
\date{June 2009}
\department{Mathematics}
\division{Physical Sciences}
\degree{Doctor of Philosophy}
\begin{document}
\maketitle
\numberwithin{equation}{chapter}
\topmatter{Dedication}
\begin{center}
To my loved ones.
\end{center}
\topmatter{Acknowledgements}
My deepest gratitude goes to Professor Wilhelm Schlag, for pointing me to the rich research topic that turned into my thesis. It has been an honor to be his student. Throughout the years, his teaching has shaped my understanding of Analysis, immeasurably improving my comprehension of the subject. His close reading of my work has resulted in interesting observations and in many changes to its presentation and organization. He patiently listened and guided me in transforming a hardly legible initial draft to a clearer final version. His patience has been invaluable all throughout.\\
I would like to thank Professor Carlos Kenig, whose teaching has had a profound impact on my mathematical education and who graciously agreed to read a draft of my thesis. Professor Peter Constantin, who was a member of the defense committee, has been very helpful and accommodating.\\
I would also like to thank Professor Michael Goldberg for our helpful conversations about the Schr\"{o}dinger equation and for providing a very valuable perspective on the subject.\\
Others deserve to be mentioned --- especially my parents and my girlfriend, Lilian, who helped, supported, and made me very happy over the years.
\topmatter{Abstract}
The current work establishes several new results concerning the linear and nonlinear Schr\"odinger equation. The linear results are helpful in proving the nonlinear ones, but present independent interest also.\\
Firstly, we consider the $\dot H^{1/2}$-critical Schr\"{o}dinger equation with a focusing cubic nonlinearity in $\set R^3$
\be
i \partial_t \psi + \Delta \psi + |\psi|^2 \psi = 0.
\lb{NLS}
\ee
It admits an eight-dimensional manifold of periodic solutions called solitons
\be\lb{0.2}
w(\pi) = w(\alpha, \Gamma, v_k, D_k) = e^{i(\Gamma + v\cdot x - t|v|^2 + \alpha^2 t)} \phi(x-2tv-D, \alpha),
\ee
where $\phi(x, \alpha)$ is a positive ground state solution of the semilinear elliptic equation
\be
-\Delta \phi + \alpha^2\phi = \phi^3.
\lb{phi}
\ee
We prove that in the neighborhood of the soliton manifold there exists a $\dot H^{1/2}$ real-analytic manifold $\mc N$ of asymptotically stable solutions of (\ref{NLS}), meaning they are the sum of a moving soliton and a dispersive term.\\
Furthermore, we show that a solution starting on $\mc N$ remains on $\mc N$ for all positive time and for some finite negative time and that $\mc N$ is a centre-stable manifold for this equation.\\
The proof of the nonlinear results is based on the method of modulation, introduced by Soffer and Weinstein and adapted by Schlag to the $L^2$-supercritical case. Novel elements include a different linearization and new Strichartz-type estimates for the linear Schr\"{o}dinger equation.\\
The main nonlinear result depends on a spectral assumption concerning the absence of embedded eigenvalues in the continuous spectrum of the Hamiltonian.\\
In the second part, we establish several new estimates for solutions of the time-dependent and time-independent linear Schr\"odinger equation, which hold under sharper or more general conditions than previously known. Several of these estimates are based on a new approach that makes use of Wiener's Theorem in the context of function spaces.\\


\tableofcontents

\mainmatter
\chapter{Introduction}
\section{Main result}
Consider equation (\ref{NLS}):
$$
i \partial_t \psi + \Delta \psi + |\psi|^2\psi = 0.
$$
For a parameter path $\pi = (v_k, D_k, \alpha, \Gamma)$ such that $\|\dot \pi\|_{L^1_t \cap L^{\infty}_t} < ~\infty$ and a positive solution $\phi(x, \alpha)$ of (\ref{phi})
$$
-\Delta \phi + \alpha^2 \phi = \phi^3,
$$
define the nonuniformly moving soliton $w_{\pi}(t)$ by
\be\begin{aligned}\lb{1.1}
w_{\pi}(t)(x) &= e^{i \theta(x, t)} \phi(x-y(t), \alpha(t)) \\
\theta(x, t) &\ds= v(t) \cdot x + \int_0^t(\alpha^2(s)-|v(s)|^2) \dd s + \Gamma(t)\\
y(t) &\ds= 2\int_0^t v(s) \dd s + D(t).
\end{aligned}
\ee
\begin{theorem}[Main result]\lb{theorem_1}
There exists a codimension-one real analytic manifold $\mc N \subset \dot H^{1/2}$, in a neighborhood of the soliton manifold, such that for any initial data $\psi(0) \in \mc N$, equation (\ref{NLS}) has a global solution $\psi$. $\mc N$ and $\psi$ have the following properties:
\begin{enumerate}
\item $\psi$ is asymptotically stable, in the sense that it decomposes into a dispersive term $r$ and a moving soliton $w_{\pi}(t)$, given by (\ref{1.1}), which converges to a final state:
\be\begin{aligned}
\psi &= r + w_{\pi}(t) \\
\|\dot\pi\|_1 &\leq C \alpha(0) \|r(0)\|_{\dot H^{1/2}_x} \\
\|r\|_{L^{\infty}_t \dot H^{1/2}_x \cap L^2_t \dot W^{1/2, 6}_x} &\leq C \|r(0)\|_{\dot H^{1/2}_x} \leq C \min_{w \text{ is a soliton}} \|\psi(0)-w\|_{\dot H^{1/2}_x}.
\end{aligned}\ee
\item $\psi$, $r$, and $w_{\pi}$ depend real-analytically on the initial data $\psi(0)$.
\item The dispersive term $r$ scatters like the solution of the free Schr\"{o}dinger equation: there exists $r_0 \in \dot H^{1/2}$ such that
\be
r(t) = e^{it\Delta} r_0 + o_{\dot H^{1/2}}(1).
\ee
\item $\psi$ stays on $\mc N$ for infinite positive time and for finite negative time, meaning that $\mc N$ is invariant under the Hamiltonian flow.
\item $\mc N$ is the centre-stable manifold of the equation.
\end{enumerate}
\end{theorem}
This result depends on the absence of embedded eigenvalues within the continuous spectrum of the linearized Hamiltonian (Section \ref{sec_1.5}). For a definition of centre-stable manifolds, the reader is referred to Section \ref{Chapter_1.3}; for a definition of the norms involved in the statement of Theorem \ref{theorem_1}, one is referred to Appendix \ref{spaces}.

As a final matter concerning notation, we denote by $C$ various constants that appear in the proof, not all equal. 
\section{Background and history of the problem}
From a physical point of view, the nonlinear Schr\"{o}dinger equation in $\set R^3$ with cubic nonlinearity and the focusing sign (\ref{NLS}) describes, to a first approximation, the self-focusing of optical beams due to the nonlinear increase of the refraction index. As such, the equation appeared for the first time in the physical literature in 1965, in \cite{kelley}. Equation (\ref{NLS}) can also serve as a simplified model for the Schr\"odinger map equation and it arises as a limiting case of the Hartree equation, the Gross-Pitaevskii equation, or in other physical contexts.

More generally, consider the semilinear focusing Schr\"odinger equation in $\set R^{n+1}$
\be
i \partial_t \psi + \Delta \psi + |\psi|^p \psi = 0,\ \psi(0) \text{ given}.
\lb{NLS'}
\ee
It admits soliton solutions, $e^{it\alpha^2} \phi(x, \alpha)$, where
\be\lb{phi'}
-\Delta \phi + \alpha^2\phi = \phi^{p+1}.
\ee
Important invariant quantities for this equation include the mass
\be
M[\psi] = \int_{\set R^n} |\psi(x, t)|^2 \dd x
\ee
and the energy
\be
E[\psi] = \int_{\set R^n} \frac 1 2 |\dl \psi(x, t)|^2 - \frac 2 {p+2} |\psi(x, t)|^{p + 2} \dd x.
\ee

Equation (\ref{NLS'}) is invariant under the rescaling
\be
\psi(x, t) \to \alpha^{2/p} \psi(\alpha x, \alpha^2 t).
\ee
We interpret this as meaning that (\ref{NLS'}) is $\dot H^{s_c}$-critical, for $s_c = n/2 - 2/p$. Of particular interest are the $L^2$ or mass-critical ($p = 4/n$) and the $\dot H^1$ or energy-critical ($p = 4/(n-2)$) cases.

Except for this introductory discussion of other authors' results, we always assume that $n=3$ and $p = 2$, in which case (\ref{NLS'}) reduces to (\ref{NLS}) and (\ref{phi'}) to (\ref{phi}). In particular, (\ref{NLS}) is $\dot H^{1/2}$-critical; in general, $p=4/(n-1)$ is $\dot H^{1/2}$-critical.

Given a soliton solution (\ref{phi'}) of the Schr\"odinger equation (\ref{NLS'}), a natural question concerns its stability under small perturbations. This issue has been addressed in the $L^2$-subcritical case by Cazenave and Lions \cite{cazenave} and Weinstein \cite{wein1}, \cite{wein2}. Their work addressed the question of orbital stability and introduced the method of modulation, which also figured in every subsequent result.

A first asymptotic stability result was obtained by Soffer--Weinstein \cite{soffer1}, \cite{soffer2}. Further results belong to Pillet--Wayne \cite{pillet}, Buslaev--Perelman \cite{buslaev1}, \cite{buslaev2}, \cite{buslaev3}, Cuccagna \cite{cuc}, \cite{cuca2}, Rod\-ni\-anski--Schlag--Soffer  \cite{rod2}, \cite{rod3}, Tsai--Yau \cite{tsayau1}, \cite{tsayau2}, \cite{tsayau3}, Gang--Sigal \cite{gansig}, and Cuccagna--Mizumachi \cite{cucmiz}.

Grillakis, Shatah, and Strauss \cite{gril1}, \cite{gril2} developed a general theory of soliton stability for Hamiltonian evolution equations, which, when applied to the Schr\"odinger equation, shows the dichotomy between the $L^2$-subcritical and critical or supercritical cases.

In the $L^2$-supercritical, $H^1$-subcritical case in $\set R^3$, Schlag proved the existence of a codimension-one Lipschitz manifold of $W^{1, 1} \cap H^1$ initial data that generate asymptotically stable solutions to (\ref{NLS}). This was followed by more results in the same vein, such as Buslaev--Perelman \cite{buslaev1}, Krieger--Schlag \cite{krisch1}, Cuccagna \cite{cuca2}, Beceanu \cite{bec}, and Marzuola \cite{mar}.

In the $L^2$-critical or supercritical case, negative energy $\langle x\rangle^{-1} L^2 \cap H^1$ initial data leads to solutions of (\ref{NLS'}) that blow up in finite time, due to the virial identity (see Glassey \cite{glassey}). For a relaxation of this condition and a survey of results see \cite{sulem} and \cite{caz2}. Berestycki--Cazenave \cite{bercaz} showed that blow-up can occur for arbitrarily small perturbations of ground states (\ref{phi'}). Recent results concerning this subject include Merle--Raphael \cite{mer2} and Krieger--Schlag \cite{krisch2}.

Weinstein \cite{wein} showed that all $H^1$ solutions $\phi$ of the mass-critical equation (\ref{NLS'}) of mass strictly less than that of the soliton, $M[\psi] < M[\phi]$, have global in time existence. Merle \cite{merle} classified threshold solutions ($M[\psi] = M[\phi]$), finding necessary and sufficient criteria for blow-up for $H^1$ data (namely, all blow-up solutions arise from transformations of the soliton), and global existence and scattering, for $\langle x \rangle^{-1} L^2 \cap H^1$ data.

A comparable result was obtained in 2006 by Kenig--Merle \cite{kenig} for the energy-critical equation (\ref{NLS'}) in the radial case. Namely, for radial $\dot H^1$ data $\phi$ of energy strictly less than that of the soliton, $E[\psi] < E[\phi]$, the following dichotomy takes place: if $\|\dl \psi(0)\|_2 < \|\dl \phi\|_2$, then $\phi$ exists globally and scatters, while if $\|\dl \psi(0)\|_2 > \|\dl \phi\|_2$ and $\psi(0) \in L^2$ then $\psi$ blows up in finite time. In this regime, the equality $\|\dl \psi(0)\|_2 = \|\dl \phi\|_2$ cannot occur. The behavior of solutions at the energy threshold, $E[\psi] = E[\phi]$, was then classified by Duyckaerts--Merle \cite{duymer}: the same two cases are present, together with three more.

Following this approach, Holmer--Roudenko \cite{holrou}, Duyckaerts--Holmer--Rou\-den\-ko \cite{duhoro}, and Duyckaerts--Roudenko \cite{duyrou} established corresponding results for the $\dot H^{1/2}$-critical equation (\ref{NLS}). Their main findings may be summarized as follows:
\begin{theorem}\lb{holmer}
Assume that $\psi$ is a solution of (\ref{NLS}) with
\be
M[\psi] E[\psi] - 2 P[\psi]^2 \leq M[\phi] E[\phi],
\ee
where $\phi>0$ is a soliton given by (\ref{phi}) and $P$ is the momentum
\be
P[\psi] = \int_{\set R^3} \ov {\psi}(x, t) i \dl \psi(x, t) \dd x.
\ee
Then one of the following holds:
\begin{enumerate}
\item If $M[\psi] \|\dl \psi\|_2^2 - 2 P[\psi]^2 < M[\phi] \|\dl \phi\|_2^2$, then $\psi$ exists globally and scatters or equals a special solution, $\phi_-$, up to symmetries: Galilean coordinate changes, scaling, complex phase change, or conjugation.
\item If $M[\psi] \|\dl \psi\|_2^2 - 2 P[\psi]^2 = M[\phi] \|\dl \phi\|_2^2$, then $\psi$ equals $e^{it} \phi(\cdot, 1)$ up to symmetries.
\item If $M[\psi] \|\dl \psi\|_2^2 - 2 P[\psi]^2 > M[\phi] \|\dl \phi\|_2^2$ and $\psi \in \langle x \rangle^{-1} L^2$ is radial, then $\psi$ blows up in finite time or must equal, up to symmetries, a special solution $\phi_+$.
\end{enumerate}
\end{theorem}
The special solutions $\phi_-$ and $\phi_+$ were defined by the aforementioned authors to have the following properties: $\phi_-(t)$ scatters  as $t \to -\infty$ and converges at an exponential rate to a soliton solution as $t \to + \infty$, while $\phi_+(t)$ blows up in finite time for $t<0$ and converges at an exponential rate to a soliton solution as $t \to + \infty$.

We explore the connection between their result, Theorem \ref{holmer}, and Theorem \ref{theorem_1} in the subsequent remark.
\begin{observation} The boundary of the region
\be
B = \{\psi \mid M[\psi] E[\psi] - 2 P[\psi]^2 \leq M[\phi] E[\phi]\},
\ee
is not a smooth manifold. Theorem \ref{theorem_1} shows that, in the neighborhood of the soliton manifold, $B$ is contained between two transverse hypersurfaces. The boundary of the region $B$ and the two transverse hypersurfaces intersect along curves (one-dimensional manifolds).
\lb{obs1}
\end{observation}
The intersection curves are described by the special solutions $\phi_+(t)$ and $\phi_-(t)$, for sufficiently large $t$.

This remark fits with several other natural observations. Firstly, eliminating the soliton manifold ($\phi$ and its symmetry transformations --- an eight-dimensional set) from $B$, an infinite-dimensional set, divides the latter into two disconnected components; this certainly could not happen if $B$ were smooth.

Secondly, consider the functional that defines $B$,
\be
F[\psi] = M[\psi] E[\psi] - 2 P[\psi]^2.
\ee
Due to the extremizing property of the soliton $\phi$ in regard to the Gagliardo-Ni\-ren\-berg-Sobolev inequality, the first differential of $F$ at $\phi$ is identically zero:
\be
\lim_{\epsilon \to 0} \frac {F[\phi + \epsilon h] - F[\phi]}{\epsilon} = 0.
\ee
Therefore, the tangent cone to $B$ at $\phi$ is actually given by the sign of second differential of $F$, which is an indefinite quadratic form; hence the lack of smoothness.

The most directly relevant results to which Theorem \ref{theorem_1} should be compared are those of Schlag \cite{schlag}, Beceanu \cite{bec}, and Cuccagna \cite{cuca2}.

In \cite{schlag}, Schlag extended the method of modulation to the $L^2$-supercritical case and proved that in the neighborhood of each soliton there exists a codimension-one Lipschitz submanifold of $H^{1}(\set R^3) \cap W^{1,1}(\set R^3)$ such that initial data on the submanifold lead to global $H^1 \cap W^{1, \infty}$ solutions to (\ref{NLS}), which decompose into a moving soliton and a dispersive term.

\cite{bec} showed that for initial data in $\Sigma = \langle x \rangle^{-1} L^2\cap H^{1}$, on a codimension one local Lipschitz manifold, there exists a global solution to (\ref{NLS}) in the same space $\Sigma$. Furthermore, the manifold is identified as the centre-stable manifold for the equation in the space $\Sigma$ (in particular, the solution stays on the manifold for some positive finite time).

Cuccagna \cite{cuca2} constructed asymptotically stable solutions for the one-di\-men\-sio\-nal mass-supercritical Schr\"odinger equation ((\ref{NLS'}) for $n=1$, $5<p<\infty$)
\be
iu_t+u_{xx}+ |u|^pu=0,\ 5<p<\infty,
\ee
starting from even $H^1$ initial data ($p=5$ is the $L^2$-critical exponent in one dimension, while every exponent is $H^1$-subcritical). The set of solutions obtained was not endowed with a manifold structure.


Unlike previous results, the current one (Theorem \ref{theorem_1}) holds in a critical space, $\dot H^{1/2}$, for equation (\ref{NLS}). Previously it was not known whether the asymptotically stable manifold exists in the critical norm or is a phenomenon related to using stronger norms than the critical one in the study of the equation. The current work puts this question to rest.

In addition, working in the critical space permits a series of improvements in the nature of results obtained. Firstly, we identify the asymptotically stable manifold as a centre-stable manifold for the equation in the sense of \cite{bates}. Secondly, using the fact that the critical norm does not grow with time, we can prove that the manifold is globally in time invariant: solutions starting on the manifold exist globally and remain on the manifold, for all positive time.

Finally, showing that the centre-stable manifold is real analytic raises the issue of its analytic continuation beyond the immediate neighborhood of the soliton manifold. It becomes an interesting problem to identify the global object that corresponds to the local centre-stable manifold.

\section{The centre-stable manifold}\lb{Chapter_1.3}
In 1983, Keller \cite{keller} considered the semilinear wave equation in $\set R^N$ with damping
\be
u_{tt} + au_t - \Delta u +f(u) = 0
\ee
and a stationary solution given by
\be
-\Delta u_0 + f(u_0) = 0.
\ee
Under some growth and smoothness conditions on $f$, namely
\begin{enumerate}
\item[F1] $f(s) - ms = \tilde f(s) \in C^1(\set R)$, $m>0$; 
\item[F2] $\tilde f(0) = \tilde f'(0) = 0$; 
\item[F3] $|\tilde f'(s)| \leq C(1 + |s|^{2/(N-2)})$; 
\item[F4] $F(s) = \int_0^s f(t) \dd t < 0$ for some $s > 0$,
\end{enumerate}
Keller proved the existence of an infinite-dimensional invariant local Lipschitz manifold of $H^1$ solutions that approach $u_0$ as $t \to \infty$ and of a finite-dimensional invariant local Lipschitz manifold of solutions that approach $u_0$ as $t$ goes to $-\infty$.

In 1989, Bates--Jones \cite{bates} proved for a large class of semilinear equations that the space of solutions decomposes into an unstable and a centre-stable manifold. Their result is the following: consider a Banach space $X$ and the semilinear equation
\be
u_t = A u + f(u),
\ee
under the assumptions
\begin{enumerate}
\item[H1] $A:X \to X$ is a closed, densely defined linear operator that generates a $C_0$ group.
\item[H2]\lb{H2} The spectrum of $A$ decomposes into $\sigma(A) = \sigma_s(A) \cup \sigma_c(A) \cup \sigma_u(A)$ situated in the left half-plane, on the imaginary axis, and in the right half-plane respectively and $\sigma_s(A)$ and $\sigma_u(A)$ are bounded.
\item[H3]\lb{H3} The nonlinearity $f$ is locally Lipschitz, $f(0) = 0$, and $\forall \epsilon>0$ there exists a neighborhood of zero on which $f$ has Lipschitz constant $\epsilon$.
\end{enumerate}
Moreover, let $X^u$, $X^c$, and $X^s$ be the $A$-invariant subspaces corresponding to $\sigma_u$, $\sigma_c$, and respectively $\sigma_s$ and let $S^c(t)$ be the evolution generated by $A$ on $X^c$. Bates and Jones further assume that
\begin{enumerate}
\item[C1-2] $dim X^u$, $dim X^s < \infty$.
\item[C3]\lb{C3} $\forall \rho>0$ $\exists M>0$ such that $\|S^c(t)\| \leq M e^{\rho|t|}$.
\end{enumerate}

Let $\Upsilon$ be the flow associated to the nonlinear equation. We call $\mc N \subset U$ $t$-invariant if, whenever $\Upsilon(s)v \in U$ for $s \in [0, t]$, $\Upsilon(s)v \in \mc N$ for $s \in [0, t]$.

Let $W^u$ be the set of $u$ for which $\Upsilon(t) u \in U$ for all $t<0$ and decays exponentially as $t \to -\infty$. Also, consider the natural direct sum projection $\pi^{cs}$ on $X^c \oplus X^s$.
\begin{definition}
A \emph{centre-stable} manifold $\mc N \subset U$ is a Lipschitz manifold with the property that $\mc N$ is $t$-invariant relative to $U$, $\pi^{cs}(\mc N)$ contains a neighborhood of $0$ in $X^c \oplus X^s$, and $\mc N \cap W^u = \{0\}$.
\lb{centr}
\end{definition}

The finding of \cite{bates} is then
\begin{theorem}
Under assumptions H1-H3 and C1-C3, there exists an open neighborhood $U$ of zero such that $W^u$ is a Lipschitz manifold which is tangent to $X^u$ at $0$ and there exists a centre-stable manifold $W^{cs} \subset U$ which is tangent to $X^{cs}$.
\lb{t1}
\end{theorem}

Gesztesy, Jones, Latushkin, Stanislavova \cite{ges} proved that Theorem \ref{t1} applies to the semilinear Schr\"odinger equation. 
More precisely, the authors established a spectral mapping theorem for the semigroup generated by linearizing the equation about the standing wave. In particular, this theorem shows that the spectral condition H2 implies the semigroup norm estimate C3. The local Lipschitz property of the nonlinearity is not addressed in \cite{ges}; it  holds on $H^s$ for sufficiently large $s$.

While providing a very general answer to the problem, the global existence of solutions on the centre manifold is not the subject of \cite{bates} and \cite{ges}. The $t$-invariance property of \cite{bates} means that a solution starting on the manifold remains there for as long as it stays small. However, one does not know the global in time behavior of the solutions. Once a solution on the centre-stable manifold leaves the specified neighborhood of zero, not even its existence is guaranteed any longer.

By contrast, the estimates used by Keller \cite{keller} required a damping term, but his result held globally in time. Schlag \cite{schlag} dispensed with the need for a damping term and proved a global asymptotic stability result, but the manifold that he constructed was not time-invariant.

The current work identifies a centre-stable manifold for (\ref{NLS}) in the critical space for the equation $\dot H^{1/2}$ and shows that solutions starting on the manifold exist globally and remain on the manifold for all time.

\section{Setting and notations}\lb{Chapter_1.4}

Equation (\ref{NLS}) admits periodic solutions $e^{it\alpha^2} \phi(x, \alpha)$, where
\be\lb{derrick-pohozaev}
\phi = \phi(x, \alpha) = \alpha \phi(\alpha x, 1)
\ee
is a solution of the semilinear elliptic equation (\ref{phi})
$$
-\Delta \phi + \alpha^2\phi = \phi^3.
$$
In particular, we restrict our attention to positive $L^2$ solutions, called \emph{ground states}. They are unique up to translation, radially symmetric, smooth, and exponentially decreasing. Their existence was proved by Berestycki--Lions in \cite{bere}, who further showed that solutions are infinitely differentiable and exponentially decaying. Uniqueness was established by Coffman \cite{coffman} for the cubic and Kwong \cite{kwong} and McLeod--Serrin \cite{mcl} for more general nonlinearities.

Equation (\ref{NLS}) is invariant under Galilean coordinate transformations, rescaling, and changes of complex phase, which we collectively call symmetry transformations:
\be
\frak g(t)(f(x, t)) = e^{i(\Gamma+v \cdot x-t|v|^2)} \alpha f(\alpha (x-2tv-D), \alpha^2 t).
\lb{coord}
\ee
If $\psi(t)$ is a solution to the equation then so is $\frak g(t) \psi(t)$, with initial data given by $\frak g(0) \psi(0)$.

Applying these transformations to $e^{it} \phi(\cdot, 1)$, the result is a wider eight-parameter family of solutions to (\ref{NLS})
\be
\frak g(t)(e^{it} \phi(x, 1)) = e^{i(\Gamma + v \cdot x - t|v|^2 + \alpha^2 t)} \phi(x-2t v- D, \alpha),
\ee
which we call solitons or standing waves.


In the sequel we denote by a capital letter the column vector consisting of a complex-valued function (written in lowercase) and its conjugate, e.g.
\be
\Psi = \bpm \psi \\ \ov \psi \epm,\ R = \bpm r \\ \ov r \epm, \text{ etc.}
\ee
On this space we consider the real-valued dot product
\be
\langle F, G \rangle = \langle f, g \rangle + \langle \ov f, \ov g \rangle = \int_{\set R^3} \big(f(x) \ov g(x) + \ov f(x) g(x)\big) \dd x.
\ee
We leave it to the reader to check that our computations preserve this symmetry, in the sense that all column vectors that appear in the sequel comprise a function and its conjugate, meaning they are of the form $F = \bpm f \\ \ov f \epm$.

We look for solutions to (\ref{NLS}) that get asymptotically close to the manifold of solitons. More precisely, we seek solutions of the form
\be
\Psi = W_{\pi}(x, t) + R(x, t) = \bpm w_{\pi}(x, t) \\ \ov w_{\pi}(x, t) \epm + \bpm r(x, t) \\ \ov r(x, t) \epm,
\ee
where $W_{\pi} = \bpm w_{\pi} \\ \ov w_{\pi} \epm$ is a moving soliton and $R=\bpm r \\ \ov r \epm$ is a small correction term that disperses like the solution of the free Schr\"{o}dinger equation as $t \to +\infty$.



We parametrize the moving soliton $w_{\pi}$ by setting, as in (\ref{1.1}),
\be\begin{aligned}
w_{\pi} & = e^{i \theta(x, t)} \phi(x - y(t), \alpha(t)) \\
& \textstyle = e^{i(\Gamma(t) + \int_0^t (\alpha^2(s) - |v(s)|^2) \dd s + v(t) x)} \phi(x - 2\int_0^t v(s) \dd s - D(t), \alpha(t)).
\lb{8}
\end{aligned}\ee
The quantities that appear in this formula,
\be
\alpha(t),\ \Gamma(t),\ v(t)=(v_1(t), v_2(t), v_3(t)), \text{ and } D(t) = (D_1(t), D_2(t), D_3(t)),
\ee
are called \emph{modulation parameters} and $\pi = (\alpha, \Gamma, v, D)$ is called the \emph{parameter} or \emph{modulation path}. There are eight parameters in total, since $v$ and $D$ have three components each.

Due to the nonlinear interaction between $w_{\pi}$ and $r$, the modulation parameters are not constant in general; they are time-dependent. However, in the course of the proof they are not allowed to vary too much. A minimal condition, which we impose henceforth, is that
\be
\dot \alpha, \dot \Gamma, \dot v_k, \dot D_k \in L^1_t
\ee
and are small in this norm. We do not assume any stronger rate of decay. This implies that the modulation parameters converge as $t \to \infty$, at no particular rate, and that their range is contained within arbitrarily small intervals.

\section{Spectrum of the Hamiltonian and the spectral assumption}\lb{sec_1.5}
The proof of Theorem \ref{theorem_1} is based on a fixed point argument. Linearizing the equation around a moving soliton produces a time-dependent Hamiltonian of the following form, for a soliton $w_{\pi^0}$ determined by a modulation path $\pi^0$ as in (\ref{8}):
\be\begin{aligned}
\mc H_{\pi^0}(t) &= \bpm \Delta + 2|w_{\pi^0}(t)|^2 & (w_{\pi^0}(t))^2 \\ -(\ov w_{\pi^0}(t))^2 & -\Delta - 2|w_{\pi^0}(t)|^2 \epm.
\end{aligned}\ee
We can always reduce the study of $i\partial_t + \mc H_{\pi^0}(t)$ to that of $i \partial_t + \mc H$, where
\be\lb{1.29}
\mc H = \bpm \Delta - 1 + 2\phi^2(\cdot, 1) & \phi^2(\cdot, 1) \\ -\phi^2(\cdot, 1) & -\Delta + 1 - 2\phi^2(\cdot, 1) \epm.
\ee
The spectrum of this Hamiltonian has been extensively studied. In Section \ref{spectru} we give a summary of the known facts. In addition, though, we must make the following \emph{spectral assumption}:
\begin{assumption}
The Hamiltonian $\mc H$ has no embedded eigenvalues within its continuous spectrum.
\lb{assum}
\end{assumption}
Such assumptions are routinely made in the proof of asymptotic stability results, as, for example, in \cite{buslaev1}, \cite{cuc}, \cite{rod3}, or \cite{schlag}.

\section{Linear results}
In the process of establishing the main nonlinear result, Theorem \ref{theorem_1}, a crucial ingredient turns out to be a Strichartz-type estimate for the solution of the time-dependent linear Schr\"{o}dinger equation
\be\lb{1.30}
i \partial_t Z - i v(t) \dl Z + A(t) \sigma_3 Z + \mc H Z = F,\ Z(0) \text{ given}.
\ee
Here
\be
\mc H = \mc H_0 + V = \bpm \Delta-\mu & 0 \\ 0 & -\Delta+\mu \epm + \bpm |W_1| & W_2 \\ -\ov W_2 & -|W_1| \epm,\ \mu \geq 0
\lb{1.17}
\ee
in the matrix nonselfadjoint case, which matters for the nonlinear applications. $|W_1|$ we always assume to be real-valued and the same is true for $W_2$ in the main case of interest. We also consider the scalar (real or complex-valued) case
\be\lb{1.32}
\mc H = \mc H_0 + V = -\Delta + V,
\ee
for which we prove analogous results.

For the purpose of establishing the nonlinear result Theorem \ref{theorem_1}, we may as well consider only potentials $V$ that are smooth and exponentially decaying. However, working in spaces of critical regularity simplifies the proof and enables one to better distinguish between essential and unnecessary features. Therefore, for the most part, we prove dispersive estimates under sharp or almost sharp assumptions.

For the purpose of solving the nonlinear equation, we are interested in two main kinds of estimates: $L^2$ boundedness of the evolution,
\be
\|Z(t)\|_2 \leq C \bigg(\|Z(0)\|_2 + \int_0^t \|F(s)\|_2 \dd s\bigg),
\ee
and Strichartz estimates
\be
\|Z\|_{L^2_t L^6_x \cap L^{\infty}_t L^2_x} \leq C (\|Z(0)\|_2 + \|F\|_{L^2_t L^{6/5}_x + L^1_t L^2_x}).
\ee

Other estimates, such as $t^{-3/2}$ kernel decay estimates, smoothing estimates, or wave operator estimates, while interesting in themselves, are not directly useful and will constitute the subject of separate papers.

Note that the Hamiltonian (\ref{1.29}) is nonselfadjoint and time-dependent, leading to specific problems that are absent in the selfadjoint case.

In the scalar selfadjoint setting, Keel--Tao \cite{tao} proved endpoint Strichartz estimates for the free Schr\"{o}dinger and wave equations and introduced a general method for obtaining endpoint estimates (based on real interpolation) that has been useful in all similar contexts.

The issue of selfadjointness matters because one needs to reprove the usual dispersive estimates concerning the Schr\"{o}dinger equation. They do not follow in the same manner as in the selfadjoint case, where, for example, the unitarity of the time evolution immediately implies the $L^2$ boundedness.

In \cite{schlag}, Schlag proved $L^1 \to L^{\infty}$ dispersive estimates for the Schr\"{o}dinger equation with a nonselfadjoint Hamiltonian, as well as non-endpoint Strichartz estimates. Erdo\^{g}an and Schlag \cite{erdsch2} proved $L^2$ bounds for the evolution as well. In \cite{bec}, endpoint Strichartz estimates in the nonselfadjoint case were obtained following the method of Keel and Tao. Finally, Cuccagna and Mizumatchi \cite{cucmiz} obtained the boundedness of the wave operators, from which all of the above can be inferred as a simple consequence.

In the linear setting, the most difficult to handle are terms of the form
\be
(\alpha(t) - \alpha(\infty)) \sigma_3 Z \text{ and } (v(t) - v(\infty)) \dl Z,
\lb{1.12}
\ee
where $Z$ is the solution. Instead of using Strichartz estimates to handle them, we make them part of the time-dependent Hamiltonian and thus avoid the issue altogether.

In the proof of linear estimates, we use Wiener's theorem in the following manner. In an abstract setting, let $H$ be a Hilbert space and $K = \mc L(H, M_t H)$ be the algebra of bounded operators from $H$ to $M_t H$, where $M_t H$ is the space of $H$-valued Borel measures on $\set R$ of finite mass, see (\ref{3.11}).
\begin{theorem}[see Theorem \ref{thm7}]\lb{thm_1.6}
If $A \in K$ is invertible then $\widehat A(\lambda)$ is invertible for every $\lambda$. Conversely, assume $\widehat A(\lambda)$ is invertible for each $\lambda$, $A = I + L$, and
\be
\lim_{\epsilon \to 0} \|L(\cdot + \epsilon) - L\|_K = 0,\ \lim_{R \to \infty}\|(1-\chi(t/R)) L(t)\|_K =0
\ee
Then $A$ is invertible in $K$.
\end{theorem}

To apply this abstract result, consider a decomposition of the potential $V$ into
\be
V = V_1 V_2,\ V_1 = |V|^{1/2} \sgn V,\ V_2 = |V|^{1/2}.
\lb{1.35}
\ee
In the matrix nonselfadjoint case (\ref{3.2}), an analogous decomposition is
\be
V = V_1 V_2,\ V_1 = \sigma_3 \bpm |W_1| & W_2 \\ \ov W_2 & |W_1| \epm^{1/2},\ V_2 = \bpm |W_1| & W_2 \\ \ov W_2 & |W_1| \epm^{1/2},
\lb{1.36}
\ee
where $\sigma_3$ is the Pauli matrix
\be
\sigma_3 = \bpm 1 & 0 \\ 0 & -1 \epm.
\ee
We are concerned with the operator
\be
T_{V_2, V_1} = V_2 e^{it \mc H_0} V_1
\ee
and its Fourier transform in regard to time
\be
\widehat T_{V_2, V_1} = i V_2 R_0(\lambda) V_1.
\ee
To apply Theorem \ref{thm_1.6}, we need to establish that the kernel $T_{V_2, V_1}$ is time integrable. This is indeed the case, in the following scale-invariant setting.
\begin{proposition}[see Proposition \ref{prop_21}] Let $\mc H_0$ be as in (\ref{1.17}) or (\ref{1.32}). Then
\be
\int_0^{\infty} \|e^{it\mc H_0} f\|_{L^{6, \infty}} \dd t \leq C \|f\|_{L^{6/5, 1}}.
\ee
\end{proposition}
We made use of Lorentz spaces, see Appendix \ref{spaces}.

Then, for the Hilbert space $H=L^2$ and $V_1$, $V_2 \in L^{3, 2}$, it follows that $T_{V_2, V_1} \in K = \mc L(H, M_t H)$. Its invertibility within this space is directly related to the existence of Strichartz estimates.

We also introduce the projection $P_c$ on the continuous spectrum and projections on the point spectrum. In Section \ref{sect_3.5} we show that, under the spectral assumption that there are no exceptional values of $\mc H$ within $\sigma(\mc H_0)$, these projections are well-defined and well-behaved.

Furthermore, we obtain the following bound on the $L^2$ evolution outside of the exceptional set:
\begin{lemma}[see Lemma \ref{lemma_28}] Consider $V \in L^{3/2, \infty}$ as in (\ref{1.17}), and $\chi \in L^{\infty}(\set R)$, such that, for some $\epsilon>0$, $\chi(\lambda) = 0$ on $\{\lambda \mid d(\lambda, \mc E) \leq \epsilon\}$. Then
\be
\chi(\mc H) =  \frac i {2\pi} \int_{\set R} \chi(\lambda) (R_V(\lambda+i0) - R_V(\lambda-i0)) \dd \lambda
\ee
is a bounded operator from $L^2$ to itself, of norm at most $C_{\epsilon} \|\chi\|_{\infty}$. Such operators commute with the evolution and with one another:
\be
e^{it \mc H} \chi(\mc H) = \chi(\mc H) e^{it \mc H} = (e^{it\lambda} \chi)(\mc H),\ \chi_1(\mc H) \chi_2(\mc H) = (\chi_1 \chi_2)(\mc H).
\ee
In particular, $e^{it \mc H} \chi(\mc H)$ is a uniformly $L^2$-bounded family of operators.
\end{lemma}

This preparation leads to the following result in the time-independent case:
\begin{theorem}[see Theorem \ref{theorem_26}] Let $Z$ be a solution of the linear Schr\"{o}dinger equation
\be
i \partial_t Z + \mc H Z = F,\ Z(0) \text{ given}.
\ee
Assume that $\mc H = \mc H_0 + V$, $V \in L^{3/2, 1}$ is as in (\ref{1.17}), and that no exceptional values of $\mc H$ are contained in $\sigma(\mc H_0)$. Then Strichartz estimates hold: for $t \geq 0$
\be\lb{1.45}
\|P_c Z\|_{L^{\infty}_t L^2_x \cap L^2_t L^{6, 2}_x} \leq C \Big(\|Z(0)\|_2 + \|F\|_{L^1_t L^2_x + L^2_t L^{6/5, 2}_x}\Big)
\ee
and
\be
\|P_c Z\|_{L^1_t L^{6, \infty}_x} \leq C \Big(\|Z(0)\|_{L^{6/5, 1}} + \|F\|_{L^1_t L^{6/5, 1}_x}\Big).
\ee
\end{theorem}

Such statements are mostly known, though the proof provided here is more general and easily translates to the nonselfadjoint case (\ref{1.17}). Goldberg \cite{gol} obtained Strichartz estimates for $L^{3/2}$ potentials. \cite{bpst} obtained decay results for critical potentials under stronger regularity assumptions. An early result of this type also belongs to Rodnianski--Schlag \cite{rodsch}, who proved nonendpoint Strichartz estimates for large potentials with $\langle x \rangle^{-2-\epsilon}$ decay.

We refer the reader to these works for a more extensive survey of known results.

For the Strichartz estimates, one can go beyond $L^{3/2, 1}$ potentials, but we keep this space for convenience. The proof actually shows that (\ref{1.45}) is true for the $L^{3/2, \infty}$ closure of bounded functions of compact support.

We treat the time-dependent case as a small perturbation of the time-independent one, but our method works for perturbations of high order, of the form (\ref{1.12}). Our main result in this context, which improves upon and is distinct from previous estimates such as \cite{rodsch}, is the following:
\begin{theorem}[see Theorem \ref{theorem_13}]
Consider equation (\ref{1.30}), for $\mc H = \mc H_0 + V$ as in (\ref{1.17}) and $V \in L^{3/2, 1}$ not necessarily real-valued:
\be\nonumber
i \partial_t Z - i v(t) \dl Z + A(t) \sigma_3 Z + \mc H Z = F,\ Z(0) \text{ given},
\ee
\be\nonumber
\mc H = \bpm \Delta - \mu & 0 \\ 0 & -\Delta + \mu \epm + \bpm |W_1| & W_2 \\ -\ov W_2 & -|W_1| \epm,
\ee
and assume that $\|A\|_{\infty}$ and $\|v\|_{\infty}$ are sufficiently small (in a manner that depends on $V$) and $\sigma(\mc H_0)$ contains no exceptional values of $\mc H$. Then
\be
\|P_c Z\|_{L^{\infty}_t L^2_x \cap L^2_t L^{6, 2}_x} \leq C \Big(\|Z(0)\|_2 + \|F\|_{L^1_t L^2_x + L^2_t L^{6/5, 2}_x}\Big).
\ee
\end{theorem}
Under the same conditions, assume in addition that, when factoring $V = V_1 V_2$ as in (\ref{1.36}), $V_1$, $V_2 \in \dot W^{1/2, 2-\epsilon} \cap \dot W^{1/2, 2+\epsilon}$. Then
\be
\|P_c Z\|_{L^{\infty}_t \dot H^{1/2}_x \cap L^2_t \dot W^{1/2, 6}_x} \leq C \Big(\|R(0)\|_{\dot H^{1/2}} + \|F\|_{L^1_t \dot H^{1/2}_x + L^2_t \dot W^{1/2, 6/5}_x}\Big).
\ee
Here $P_c$ is the projection on the continuous spectrum of $\mc H$. Only the $\dot H^{1/2}$ version of this result is important for the nonlinear application.

More detailed statements of these results and remarks concerning them follow  within the paper.
\chapter{The Nonlinear Result}
\section{Deriving the linearized equation}

Substituting $\psi = w_{\pi} + r$ in the original equation (\ref{NLS}),
\be
i\partial_t (w_{\pi} + r) + \Delta (w_{\pi} + r) + (\ov w_{\pi} + \ov r) (w_{\pi}+r)^2 = 0.
\ee
In keeping with (\ref{8}), the soliton $w_{\pi}$ is described by
\be\textstyle\nonumber
w_{\pi}(t)=e^{i(\Gamma(t) + \int_0^t (\alpha^2(s) - |v(s)|^2) \dd s + v(t) \cdot x)} \phi\big(x - 2\int_0^t v(s) \dd s - D(t), \alpha(t)\big).
\ee
Note that $w_{\pi}(t)$ depends on the values $\pi$ takes on $[0, t]$. Indeed, if we define
\be\lb{wpi}
w(\pi)=w(\alpha, \Gamma, v, D) = e^{i(\Gamma + v \cdot x)} \phi\big(x - D, \alpha\big),
\ee
then $w_{\pi}(t) \ne w(\pi(t))$; in fact,
\be
w_\pi(t) = w\bigg(\alpha(t), \Gamma(t) + \int_0^t (\alpha^2(s) - |v(s)|^2) \dd s, v(t), D(t) + 2\int_0^t v(s) \dd s\bigg).
\ee
Expanding the equation accordingly, note that
\begin{align}
\partial_t w_{\pi} &= (\dot \Gamma + \alpha^2 - v^2) \partial_{\Gamma} w_{\pi} + \dot \alpha \partial_{\alpha} w_{\pi} + \dot v \partial_v w_{\pi} - (2v + \dot D) \partial_D w_{\pi}\\
\intertext{and}
\Delta w_{\pi} &= \Delta e^{i \theta(x, t)} \phi(x-y(t), \alpha(t)) + 2 \dl e^{i \theta(x, t)} \dl \phi(x-y(t), \alpha(t)) + \nonumber\\
&+ e^{i \theta(x, t)} \Delta \phi(x-y(t), \alpha(t)) \\
&= (\alpha^2 -v^2) w_{\pi} + 2i v \dl w_{\pi} - |w_{\pi}|^2 w_{\pi}. \nonumber
\end{align}
Here we used the following notation: given a soliton $w=w(\pi)$ as in (\ref{wpi}), $\partial_{\Gamma} w$, $\partial_{\alpha} w$, $\partial_{D_k} w$, and $\partial_{v_k} w$ 
are the partial derivatives of $w$ with respect to these parameters and span the tangent space to the soliton manifold at the point $w$:
\be
\partial_{\Gamma} w = iw,\ \partial_{\alpha} w = \partial_{\alpha} w,\ \partial_{D_k}w = \partial_{x_k} w,\ \partial_{v_k} w = i x_k w.
\ee
We use these partial derivatives to define the differential $d_{\pi} w$ of the map $\pi \mapsto w=w(\pi)$. The differential $d_{\pi} w$ maps a vector $\delta \pi = (\delta \alpha, \delta \Gamma, \delta v_k, \delta D_k)$ in the tangent space at $\pi$ to a vector in the tangent space at $w$:
\be
(d_{\pi} w) \delta \pi = (\delta \Gamma) \partial_{\Gamma} w + (\delta \alpha) \partial_{\alpha} w + (\delta D_k) \partial_{D_k}w + (\delta v_k) \partial_{v_k} w.
\ee

This expansion results in the cancellation of the main term involving $w_{\pi}$. The equation concerning $r$ becomes
\begin{multline}\lb{2.7}
i\partial_t r + \Delta r + i(\partial_\pi w_{\pi})\dot \pi + (|r|^2 r + r^2 \ov w_{\pi} + 2|r|^2 w_{\pi} + 2r|w_{\pi}|^2 + \ov r w_{\pi}^2) = 0.
\end{multline}
Here $d_{\pi} w_{\pi}$ is the differential evaluated at the point $w_{\pi}$. Since as noted previously $w_{\pi} \ne w(\pi)$, $(d_{\pi} w_{\pi}) \dot \pi$ is not the same as the time derivative of $w_{\pi}$.

We regard equation (\ref{2.7}) as a Schr\"{o}dinger equation for $r$. By conjugating, we obtain an equivalent equation for $\ov r$:
\begin{multline}
i\partial_t \ov r - \Delta \ov r + i \ov{d_{\pi} w_{\pi}} \dot \pi - (|r|^2 \ov r + (\ov r)^2 w_{\pi} + 2|r|^2 \ov w_{\pi} + 2\ov r|w_{\pi}|^2 + r (\ov w_{\pi})^2) = 0.
\end{multline}
Since both equations involve both $r$ and $\ov r$, it is most convenient to solve them together as a system or, rather, to see the pair of equations as just one equation concerning the column vector $R = \bpm r \\ \ov r \epm$.

Adopting this point of view, the main terms assemble into a time-dependent, nonselfadjoint two-by-two \emph{matrix} Hamiltonian
\be\lb{2.10}
\mc H_{\pi}(t) = \bpm \Delta + 2|w_{\pi}(t)|^2 & w_{\pi}(t)^2 \\ -\ov w_{\pi}(t)^2 & -\Delta - 2|w_{\pi}(t)|^2 \epm = \Delta \sigma_3 + V_{\pi}(t),
\ee
whereas the other terms are better treated as the homogenous right-hand side of the equation:
\be
F = \bpm i(d_{\pi} w_{\pi})\dot \pi + |z|^2 z + z^2 \ov{w}_{\pi} + 2|z|^2 w_{\pi} \\ i \, \ov{d_{\pi} w_{\pi}} \dot \pi - |z|^2 \ov z - (\ov z)^2 w_{\pi} - 2|z|^2 \ov w_{\pi} \epm.
\ee
Given a $\set C^2$-valued soliton $W= \bpm w \\ \ov w \epm$, we also introduce the partial derivatives $\partial_f W = \bpm \partial_f w \\ \ov {\partial_f w} \epm$, for $f \in \{\alpha, \Gamma, v_k, D_k\}$:
\be
\partial_{\Gamma}W = \bpm i w \\ -i \ov w \epm = i \sigma_3 W,\ \partial_{\alpha} W = \partial_{\alpha} W,\ \partial_{D_k} W = \partial_{x_k} W,\ \partial_{v_k} W = i \sigma_3 x_k W.
\ee

As a reminder, $\sigma_3$ is one of the Pauli matrices:
\be
\sigma_3 = \bpm 1 & 0 \\ 0 & -1 \epm.
\ee

We use these to define the differential $d_{\pi} W(\pi)$ of the map
\be\begin{aligned}
\pi \mapsto W &= W(\pi) = \bpm w(\pi) \\ \ov w(\pi) \epm, \\
(d_{\pi} W(\pi)) \delta \pi &= (\delta \Gamma) \partial_{\Gamma} W + (\delta \alpha) \partial_{\alpha} W + (\delta D_k) \partial_{D_k} W + (\delta v_k) \partial_{v_k} W.
\end{aligned}\ee
Thus, the first half of $F$ is simply given by the differential
\be
\bpm i(d_\pi w_{\pi})\dot \pi \\ i\, \ov{d_\pi w_{\pi}} \dot \pi \epm = i (d_{\pi} W_{\pi}) \dot \pi.
\ee
We denote the remaining nonlinear term by
\be
N(R, W_{\pi}) = \bpm -|r|^2 r - r^2 \ov{w}_{\pi} - 2|r|^2 w_{\pi} \\ |r|^2 \ov {r} + \ov {r}^2 w_{\pi} + 2|r|^2 \ov w_{\pi} \epm.
\ee
In vector form, the equation fulfilled by $R$ becomes
\be\lb{rr}
i \partial_t R - \mc H_{\pi}(t) R = F(t),\ F = -i (d_{\pi} W_{\pi}) \dot \pi + N(R, W_{\pi}).
\ee

To this equation concerning $R$ we join the \emph{modulation equations} that determine the path $\pi$. For future convenience, define the cotangent vectors
\be\begin{aligned}\lb{2.19}
\Xi_{\alpha}(W) &= i\sigma_3 \partial_{\Gamma}W, &\Xi_{\Gamma}(W) &= i\sigma_3 \partial_{\alpha}W, \\
\Xi_{v_k}(W) &= i\sigma_3 \partial_{D_k}W, &\Xi_{D_k}(W) &= i\sigma_3 \partial_{v_k}W.
\end{aligned}\ee
Also consider the (second) differential $d_{\pi} \Xi_f(W)$, naturally defined as the differential of the map $\pi \mapsto \Xi_f(W(\pi))$, for $f \in \{\alpha, \Gamma, v_k, D_k\}$.

At each time $t$ and for all $f \in \{\alpha, \Gamma, v_k, D_k\}$ we impose the \emph{orthogonality condition}
\be\lb{2.20}
\langle R(t), \Xi_{f}(W_{\pi}(t)) \rangle = 0.
\ee
This condition arises as follows: given parameters $\pi=(\alpha, \Gamma, v_k, D_k)$, let
\be\begin{aligned}
w&=w(\pi)=e^{i(\Gamma + v \cdot x)} \phi(x-D, \alpha), \\
W&=W(\pi)= \bpm w \\ \ov{w} \epm, \\
\mc H(W) &= \bpm \Delta + 2|w|^2 & w^2 \\ -w^2 & -\Delta -2|w|^2 \epm + 2iv \dl - (\alpha^2-|v|^2) \sigma_3.
\end{aligned}\ee
The Hamiltonian $\mc H(W)$, naturally associated to each soliton $W$ (see Section \ref{spectru}), will play an important role in the sequel.

Note that $\mc H(W_{\pi}(t)) = \mc H_{\pi}(t) + 2iv(t) \dl - (\alpha(t)^2-|v(t)|^2) \sigma_3$.

Condition (\ref{2.20}) is then equivalent to asking that $R$ should not live in the zero eigenspace of the time-dependent Hamiltonian $\mc H(W_{\pi}(t))$.

Taking the derivative in (\ref{2.20}), it translates into the following modulation equations.
\begin{lemma}[The modulation equations]\lb{lemma_6}
\be\begin{aligned}
\dot \alpha &= 4\alpha\|W_{\pi}\|_2^{-2} (\langle R, (d_{\pi} \Xi_{\alpha}(W_{\pi}))\dot \pi \rangle - i\langle N(R, W_{\pi}), \Xi_{\alpha}(W_{\pi})\rangle)\\
\dot \Gamma &= 4\alpha^{}\|W_{\pi}\|_2^{-2} (\langle R, (d_{\pi} \Xi_{\Gamma}(W_{\pi}))\dot \pi\rangle - i\langle N(R, W_{\pi}), \Xi_{\Gamma}(W_{\pi})\rangle)\\
\dot v_k &= 2\|W_{\pi}\|_2^{-2} (\langle R, (d_{\pi} \Xi_{v_k}(W_{\pi}))\dot \pi\rangle - i\langle N(R, W_{\pi}), \Xi_{v_k}(W_{\pi})\rangle)\\
\dot D_k &= 2\|W_{\pi}\|_2^{-2} (\langle R, (d_{\pi} \Xi_{D_k}(W_{\pi}))\dot \pi\rangle - i\langle N(R, W_{\pi}), \Xi_{D_k}(W_{\pi})\rangle).
\lb{mod}\end{aligned}\ee
\end{lemma}
\begin{proof} Begin by observing that for every soliton $W$
\be\begin{aligned}
\langle \partial_{\alpha}W, \Xi_{f}(W) \rangle &= \frac 1 {4 \alpha^{}} \|W\|_2^2 \text{ if } f=\alpha \text{ and zero otherwise}\\
\langle \partial_{\Gamma}W, \Xi_{f}(W) \rangle &= \frac 1 {4 \alpha^{}} \|W^{}\|_2^2 \text{ if } f=\Gamma \text{ and zero otherwise}\\
\langle \partial_{D_k}W, \Xi_{f}(W) \rangle &= \frac 1 2 \|W^{}\|_2^2 \text{ if } f=D_k \text{ and zero otherwise}\\
\langle \partial_{v_k}W, \Xi_{f}(W) \rangle &= \frac 1 2 \|W^{}\|_2^2 \text{ if } f=v_k \text{ and zero otherwise}.
\end{aligned}\ee
Furthermore,
\be\begin{aligned}\lb{2.22}
\mc H^*(W) \Xi_{\alpha}(W) &= 0, &\mc H^*(W) \Xi_{\Gamma}(W) &= -2i \Xi^{}_{\alpha}, \\
\mc H^*(W) \Xi_{v_k}(W) &= 0, &\mc H^*(W) \Xi_{D_k}(W) &= -2i \Xi^{}_{v_k}.
\end{aligned}\ee
Finally,
\be
\partial_t \Xi_f(W_{\pi}) = (d_{\pi} \Xi_{f}(W_{\pi})) \dot \pi + (\mc H^*(W_{\pi}) - \mc H_{\pi}^*) \Xi_{f}(W_{\pi}).
\ee
Then, in the equality
\be
\langle R, \partial_t \Xi_{f}(W_{\pi}) \rangle = - \langle \partial_t R, \Xi^{}_{f} \rangle,
\ee
we replace $\partial_t R$ by its expression (\ref{rr}) and arrive at (\ref{mod}).
\end{proof}

Let
\be\begin{aligned}
L_{\pi^{}} R = &4\alpha^{} \sum_{f \in \alpha, \Gamma} \|W_{\pi}\|_2^{-2} \langle R, (d_{\pi} \Xi_{f}(W_{\pi})) \dot \pi \rangle \partial_f W_{\pi} \\
+ & 2\sum_{f \in \{v_k, D_k\}} \|W_{\pi}\|_2^{-2} \langle R, (d_{\pi} \Xi_{f}(W_{\pi})) \dot \pi \rangle \partial_f W_{\pi}
\end{aligned}\ee
and
\be\begin{aligned}
N_{\pi}(R, W_{\pi}) = & 4\alpha^{} \sum_{f \in \{\alpha, \Gamma\}} \|W_{\pi}\|_2^{-2} i\langle N(R, W_{\pi}), \Xi_{f}(W_{\pi}) \rangle \partial_{f}W_{\pi} \\
+ & 2\sum_{f \in \{v_k, D_k\}} \|W_{\pi}\|_2^{-2} i\langle N(R, W_{\pi}), \Xi_{f}(W_{\pi})\rangle \partial_{f} W_{\pi}.
\end{aligned}\ee
The modulation equations can then be rewritten as
\be
(d_{\pi} W_{\pi}) \dot \pi = L_{\pi^{}} R - i N_{\pi}(R, W_{\pi}).
\lb{modula}
\ee
$L_{\pi^{}} R$ represents the part that is linear in $R$ and $N_{\pi}(R, W)$ represents the nonlinear component $\langle N(R, W), \Xi_{f}(W_{\pi})\rangle$.

As an aside, using the notation of Section \ref{spectru}, let $P_0(W_{\pi})$ be the zero spectrum projection of the Hamiltonian corresponding to $W_{\pi}$. We recognize
\be
N_{\pi}(R, W_{\pi}) = P_0(W_{\pi}(t)) N(R, W_{\pi})
\ee
and a similar, but more complicated, expression for $L_{\pi} R$.

At this point we linearize the equation. We use an auxiliary function for all quadratic and cubic terms and do the same in regard to the soliton. We only keep linear, first-order terms in the unknowns $R$ and $\pi$ for which we solve the equation, so that the equation becomes linear in $R$ and $\pi$ and quadratic and cubic in the terms involving $R^0$ and $\pi^0$.

We introduce an auxiliary function $R^0$ and an auxiliary modulation path $\pi^0$. Starting with $\pi^0$, we construct the moving soliton $w_{\pi^0}$, its partial derivatives $\partial_f w_{\pi^0}$, cotangent vectors $\Xi_f(W_{\pi^0})$, the differentials $d_{\pi} W_{\pi^0}$ and $d_{\pi} \Xi_f(W_{\pi^0})$, etc., all following previous definitions:
\be\begin{aligned}\lb{w0}
\pi^0 &= (\alpha^0, \Gamma^0, v^0, D^0), \\
w_{\pi^0} &\textstyle = e^{i(\Gamma^0(t) + \int_0^t ((\alpha^0)^2(s) -|v^0|^2(s)) \dd s + v^0(t) \cdot x)}
\textstyle\phi(x - 2\int_0^t v^0(s) \dd s - D^0(t), \alpha^0(t)), \\
W_{\pi^0} &= \bpm w_{\pi^0} \\ \ov {w}_{\pi^0} \epm,
\end{aligned}\ee
as well as
\be\begin{aligned}
\partial_{\Gamma} w_{\pi^0} &= iw_{\pi^0},\ \partial_{\alpha} w_{\pi^0} = \partial_{\alpha} w_{\pi^0},\ \partial_{D_k} w_{\pi^0} = \partial_{x_k} w_{\pi^0},\ \partial_{v_k} w_{\pi^0} = i x_k w_{\pi^0}, \\
\partial_{f} W_{\pi^0} &= \bpm \partial_f w_{\pi^0} \\ \partial_f \ov w_{\pi^0} \epm, \\
\Xi_{\alpha}(W_{\pi^0}) &= i\sigma_3 \partial_{\Gamma} W_{\pi^0}, \Xi_{\Gamma}(W_{\pi^0}) = i\sigma_3 \partial_{\alpha} W_{\pi^0}, \\
\Xi_{v_k}(W_{\pi^0}) &= i\sigma_3 \partial_{D_k} W_{\pi^0}, \Xi_{D_k}(W_{\pi^0}) = i\sigma_3 \partial_{v_k} W_{\pi^0}.
\end{aligned}\ee

\begin{lemma}
$\psi$ is a solution of (\ref{NLS}) if and only if
\be
\Psi = \bpm \psi \\ \ov{\psi} \epm = W_{\pi^0} + R,
\ee
$W_{\pi^0}$ is given by (\ref{w0}), and $(R, \pi)$ is a fixed point of the map $(R^0, \pi^0) \mapsto (R, \pi)$, where $R$ and $\pi$ solve the following system of linear equations:
\be\begin{aligned}
&i \partial_t R + \mc H_{\pi^0}(t) R = F,\ F= -i L_{\pi^0} R + N(R^0, W_{\pi^0}) - N_{\pi^0}(R^0, W_{\pi^0})\\
&\dot f = 4\alpha^0 \|W_{\pi^0}\|_2^{-2} (\langle R, (d_{\pi} \Xi_{f}(W_{\pi^0})) \dot \pi^0 \rangle - i\langle N(R^0, W_{\pi^0}), \Xi_{f}(W_{\pi^0})\rangle), f \in \{\alpha, \Gamma\}\\
&\dot f = 2\|W_{\pi^0}\|_2^{-2} (\langle R, (d_{\pi} \Xi_{f}(W_{\pi^0})) \dot \pi^0 \rangle - i\langle N(R^0, W_{\pi^0}), \Xi_{f}(W_{\pi^0})\rangle), f \in \{v_k, D_k\}.
\end{aligned}\lb{ec_liniara}\ee
\lb{lemma1.1}
\end{lemma}
Here
\be\begin{aligned}
\mc H_{\pi^0}(t) &= \bpm \Delta + 2|w_{\pi^0}(t)|^2 & (w_{\pi^0}(t))^2 \\ -(\ov w_{\pi^0}(t))^2 & -\Delta - 2|w_{\pi^0}(t)|^2 \epm, \\
N(R^0, W_{\pi^0}) &= \bpm -|r^0|^2 r^0 - (r^0)^2 \ov w_{\pi^0} - 2|r^0|^2 w_{\pi^0} \\ |r^0|^2 \ov {r^0} + (\ov {r^0})^2 w_{\pi^0} + 2|r^0|^2 \ov w_{\pi^0} \epm, \\
L_{\pi^0} R &= 4\alpha^0 \sum_{f \in \alpha, \Gamma} \|W_{\pi^0}\|_2^{-2} \langle R, (d_{\pi} \Xi_{f}(W_{\pi^0})) \dot \pi^0 \rangle \partial_f W_{\pi^0} \\
&+ 2\sum_{f \in \{v_k, D_k\}} \|W_{\pi^0}\|_2^{-2} \langle R, (d_{\pi} \Xi_{f}(W_{\pi^0})) \dot \pi^0 \rangle \partial_f W_{\pi^0},
\end{aligned}\ee
\be\nonumber\begin{aligned}
N_{\pi^0}(R^0, W_{\pi^0}) &= 4\alpha^0 \sum_{f \in \{\alpha, \Gamma\}} \|W_{\pi^0}\|_2^{-2} i\langle N(R^0, W_{\pi^0}), \Xi_{f}(W_{\pi^0}) \rangle \partial_{f}W_{\pi^0} \\
&+ 2\sum_{f \in \{v_k, D_k\}} \|W_{\pi^0}\|_2^{-2} i\langle N(R^0, W_{\pi^0}), \Xi_{f}(W_{\pi^0})\rangle \partial_{f} W_{\pi^0}.
\end{aligned}
\ee
Solving the system (\ref{ec_liniara}) in a suitable space, with a view toward applying a fixed point theorem to conclude that $(R^0, \pi^0) = (R, \pi)$, is our main objective.

\begin{proof}
We substitute $R^0$ for $R$ and $\pi^0$ for $\pi$ into all higher order terms of (\ref{rr}) and (\ref{mod}). Following this substitution, the nonlinear Schr\"{o}dinger equation (\ref{rr}) transforms into the linear equation, in both $R$ and $\dot \pi$,
\be
i \partial_t R - \mc H_{\pi^0}(t) R = F,
\lb{nlsu}
\ee
where
\be\begin{aligned}
F &= - i (d_{\pi} W_{\pi^0}) \dot \pi + N(R^0, W_{\pi^0}), \\
(d_{\pi} W_{\pi^0}) \dot \pi &= \dot \Gamma \partial_{\Gamma}W_{\pi^0} + \dot \alpha \partial_{\alpha}W_{\pi^0} + \dot D_k \partial_{D_k}W_{\pi^0} + \dot v_k \partial_{v_k}W_{\pi^0}.
\end{aligned}\ee
The orthogonality condition we impose in the linear setting is
\be
\langle R(t), \Xi_f(W_{\pi^0}(t)) \rangle = 0,\ f \in \{\alpha, \Gamma, v_k, D_k\}.
\ee

By taking the derivative, in a manner entirely similar to that of Lemma \ref{lemma_6} we obtain the linearized modulation equations: (\ref{mod}) become
\be\begin{aligned}
\dot \alpha &= 4\alpha^{0}\|W_{\pi^0}\|_2^{-2} (\langle R, (d_{\pi} \Xi_{\alpha}(W_{\pi^0})) \dot \pi^0 \rangle - i\langle N(R^{0}, \pi^{0}), \Xi_{\alpha}^{0}\rangle)\\
\dot \Gamma &= 4\alpha^{0}\|W^{0}\|_2^{-2} (\langle R, (d_{\pi} \Xi_{\Gamma}(W_{\pi^0})) \dot \pi^0 \rangle - i\langle N(R^{0}, \pi^{0}), \Xi_{\Gamma}^{0}\rangle)\\
\dot v_k &= 2\|W^{0}\|_2^{-2} (\langle R, (d_{\pi} \Xi_{v_k}(W_{\pi^0})) \dot \pi ^0\rangle - i\langle N(R^{0}, \pi^{0}), \Xi_{v_k}^{0}\rangle)\\
\dot D_k &= 2\|W^{0}\|_2^{-2} (\langle R, (d_{\pi} \Xi_{D_k}(W_{\pi^0})) \dot \pi^0 \rangle - i\langle N(R^{0}, \pi^{0}), \Xi_{D_k}^{0}\rangle).
\lb{mod0}\end{aligned}\ee
The identity (\ref{modula}) translates into
\be
(d_{\pi} W_{\pi^0}) \dot \pi = L_{\pi^0} R - i N_{\pi^0}(R^0, \pi^0).
\lb{modula0}
\ee

Finally, we collect together (\ref{nlsu}) and (\ref{mod0}) and replace $(d_{\pi} W_{\pi^0}) \dot\pi$ on the right-hand side of (\ref{nlsu}) by its expression (\ref{modula}) in order to arrive at (\ref{ec_liniara}).
\end{proof}

\section{Spectral considerations}\lb{spectru}

Given parameters $\pi=(\alpha, \Gamma, v_k, D_k)$, define the soliton
\be\begin{aligned}
w&=w(\pi)=e^{i(\Gamma + v \cdot x)} \phi(x-D, \alpha), \\
W&=W(\pi)= \bpm w \\ \ov{w} \epm
\end{aligned}\ee
and consider the associated Hamiltonian
\be\lb{ham}
\mc H(W) = \mc H(\alpha, \Gamma, v, D) = \bpm \Delta + 2|w|^2 & w^2 \\ -\ov w^2 & -\Delta -2|w|^2 \epm + 2iv \dl - (\alpha^2-|v|^2) \sigma_3.
\ee

By rescaling and conjugating by $e^{i(xv+\Gamma)\sigma_3}$ as well as by a translation, one sees that all these operators are in fact conjugate, up to a constant factor of $\alpha^2$:
\be
\Dil_{1/\alpha} e^{-D\dl} T_D e^{-i(xv+\Gamma) \sigma_3} \mc H(\alpha, \Gamma, v, D) e^{i(xv +\Gamma)\sigma_3} e^{D\dl} \Dil_{\alpha} = \alpha^2 \mc H(1, 0, 0, 0).
\ee
Therefore all have the same spectrum up to dilation and have similar spectral properties; thus, it suffices to study $\mc H = \mc H(1, 0, 0, 0)$:
\be
\mc H = \bpm \Delta - 1 + 2\phi^2 & \phi^2 \\ -\phi^2 & -\Delta + 1 - 2\phi^2 \epm.
\ee

We restate the known facts about the spectrum of $\mc H$. As proved by Buslaev, Perelman \cite{buslaev1} and also Rodnianski, Schlag, Soffer in \cite{rod3}, under fairly general assumptions, $\sigma(\mc H) \subset \set R \cup i\set R$ and is symmetric with respect to the coordinate axes and all eigenvalues are simple with the possible exception of $0$. Furthermore, by Weyl's criterion $\sigma_{ess}(\mc H) = (-\infty, -1] \cup [1, +\infty)$.

Grillakis, Shatah, Strauss \cite{gril1} and Schlag \cite{schlag} showed that there is only one pair of conjugate imaginary eigenvalues  $\pm i\sigma$ and that the corresponding eigenvectors decay exponentially. For the decay see Hundertmark, Lee \cite{hund}. The pair of conjugate imaginary eigenvalues $\pm i\sigma$ reflects the $L^2$-supercritical nature of the problem.

The generalized eigenspace at $0$ arises due to the symmetries of the equation, which is invariant under Galilean coordinate changes, phase changes, and scaling. It is relatively easy to see that each of these symmetries gives rise to a generalized eigenvalue of the Hamiltonian $\mc H$ at $0$, but proving the converse is much harder and was done by Weinstein in \cite{wein1}, \cite{wein2}.

Schlag \cite{schlag} showed, using ideas of Perelman \cite{perel}, that if the operators
\be
L_{\pm} = -\Delta + 1 - 2\phi^2(\cdot, 1) \mp \phi^2(\cdot, 1)
\ee
that arise by conjugating $\mc H$ with $\bpm 1 & i \\ 1 & -i \epm$ have no eigenvalue in $(0, 1]$ and no resonance at $1$, then the real discrete spectrum of $\mc H$ is $\{0\}$ and the edges $\pm 1$ are neither eigenvalues nor resonances. A work of Demanet, Schlag \cite{demanet} verified numerically that the scalar operators meet these conditions. Therefore, there are no eigenvalues in $[-1, 1]$ and $\pm 1$ are neither eigenvalues nor resonances for $\mc H$.

Furthermore, the method of Agmon \cite{agmon}, adapted to the matrix case, enabled Erdogan--Schlag \cite{erdsch2} and independently \cite{cuc2} to prove that any resonances embedded in the interior of the essential spectrum (that is, in $(-\infty, -1) \cup (1, \infty)$) have to be eigenvalues, under very general assumptions.

Under the spectral Assumption \ref{assum} we now have a complete description of the spectrum of $\mc H$. It consists of a pair of conjugate purely imaginary eigenvalues $\pm i\sigma$, a generalized eigenspace at $0$, and the essential spectrum $(-\infty, -1] \cup [1, \infty)$.

Following \cite{schlag}, let $F^{\pm}$ be the eigenfunctions of $\mc H$ corresponding to $\pm i \sigma$; then there exists $f^+ \in L^2$ such that
\be
F^+ = \bpm f^+ \\ \ov f^+ \epm,\ F^- = \ov F^+ = \bpm \ov f^+ \\ f^+ \epm.
\ee
Due to the symmetry
\be
\sigma_3 \mc H \sigma_3 = \mc H^*
\ee
of the operator, the respective eigenfunctions of $\mc H^*$ are $\sigma_3 F^{\pm}$. Then, the imaginary spectrum projection is given by
\be
P_{im} = P_+ + P_-,\ P_{\pm}= \langle \cdot, i\sigma_3 F^{\mp} \rangle F^{\pm}
\ee
up to a constant; the constant becomes $1$ after the normalization
\be
\int_{\set R^3} \Re f^+(x) \Im f^+(x) \dd x = -1/2.
\ee

It helps in the proof to exhibit the discrete eigenspaces of $\mc H(W)$. Denote by $F^{\pm}(W)$ the eigenfunctions of $\mc H(W)$ corresponding to the $\pm i \sigma(W)$ eigenvalues. Even though there is no explicit form of the imaginary eigenvectors, Schlag \cite{schlag} proved in a more general setting that $F^{\pm}(W)$, $L^2$ normalized, and $\sigma$ are locally Lipschitz continuous as a function of $W$ and that $F^{\pm}(W)$ are exponentially decaying.

More is true in the case under consideration. Since the operators $\mc H(W)$ are conjugate up to a constant, the dependence of $F^{\pm}(W)$ and $\sigma$ on the parameters can be made explicit:
\be\begin{aligned}\lb{2.50}
F^{\pm}(W) &= e^{i(xv+\Gamma)\sigma_3} e^{D\dl} \Dil_{\alpha} F^{\pm},\\
\sigma(W) &= \alpha^2 \sigma.
\end{aligned}\ee

Also observe that $\partial_f W$, where $f \in \{\alpha, \Gamma, v_k, D_k\}$, are the generalized eigenfunctions of $\mc H(W)$ at zero and $\Xi_f(W)$, defined as in (\ref{2.19}), fulfill the same role for~$\mc H(W)^*$.

We then express the Riesz projections corresponding to the three components of the spectrum of $\mc H(W)$ as
\begin{align}
P_{im}(W) &= P_+(W) + P_-(W),\ P_{\pm}(W)= \alpha^{-3} \langle \cdot, i\sigma_3 F^{\mp}(W) \rangle F^{\pm}(W), \\
P_{0}(W) &= 4\alpha \|W\|_2^{-2} (\langle \cdot, \Xi_{\alpha} \rangle \partial_{\alpha}W + \langle \cdot, \Xi_{\Gamma} \rangle \partial_{\Gamma}W) + \\
&\nonumber+ 2 \|W\|_2^{-2} \sum_k (\langle \cdot, \Xi_{v_k} \rangle \partial_{v_k}W + \langle \cdot, \Xi_{D_k} \rangle \partial_{D_k}W),
\intertext{and}
P_c(W) &= 1-P_{im}(W)-P_{0}(W).
\end{align}


\section{The fixed point argument: stability} Consider a small neighborhood of a fixed soliton $w_0=w(\pi_0)$, determined by parameters $\pi_0 = (\alpha_0, \Gamma_0, v_{k0}, D_{k0})$:
\be\begin{aligned}
W_0 &= \bpm w_0 \\ \ov w_0 \epm, \\
w_0 &= e^{i(\Gamma_0 + v_0 \cdot x)} \phi(x-D_0, \alpha_0).
\end{aligned}\ee
As in Section \ref{spectru}, $W_0$ has an associated Hamiltonian $\mc H(W_0)$ of the form (\ref{ham})
\be
\mc H(W_0) = \bpm \Delta + 2 |w_0|^2 & w_0^2 \\ - \ov w_0^2 & -\Delta - 2|w_0|^2 \epm + 2iv_0\dl - (\alpha_0^2-|v_0|^2) \sigma_3.
\ee
In turn, to $\mc H(W_0)$ we associate the zero spectrum projection $P_0(W_0)$, the imaginary spectrum projection $P_{im}(W_0) = P_+(W_0) + P_-(W_0)$, and the continuous spectrum projection $P_c(W_0) = I - P_0(W_0) - P_{im}(W_0)$.

Up to quadratic corrections, the centre-stable submanifold is given by the affine subspace
\be
W_0 + (P_c(W_0) + P_-(W_0)) \dot H^{1/2} = \{W_0 + R_0 \mid R_0 \in \dot H^{1/2},\ (P_0(W_0) + P_+(W_0)) R_0 = 0\}.
\ee
This submanifold will have codimension nine, so we need a supplementary argument, presented at the end, to recover eight codimensions.

Take initial data of the form
\be\begin{aligned}
R(0) &= R_0 + h F^+(W_0), \pi(0)=\pi_0=(\alpha_0, \Gamma_0, (v_k)_0, (D_k)_0),
\lb{Z_initial}
\end{aligned}\ee
where $R_0 \in  (P_c + P_-) \dot H^{1/2}$. Thus, $R_0$ belongs to the appropriate linear subspace, while $h F^+$ is the quadratic correction, made in the direction of $F^+$; in particular,
\be
P_0(0) R(0) = 0,\ P_+ R(0) = h.
\ee
(\ref{Z_initial}) also implies that $\|R(0)\|_{\dot H^{1/2}} \leq C (\|R_0\|_{\dot H^{1/2}} + |h|)$.

We define a map $\Upsilon$ as follows:
\begin{definition}
$\Upsilon$ is the map that, given a pair $(R^0, \pi^0)$, associates to it the unique bounded solution $(R, \pi)$ of the linearized equation system (\ref{ec_liniara}) from Lemma \ref{lemma1.1}, with initial data as in (\ref{Z_initial}), where $R_0$ is given and $h$ is allowed to take any value:
\be
\Upsilon((R^0, \pi^0)) = (R, \pi).
\lb{2.32}
\ee
\end{definition}
In the sequel we show that, for given $R_0$, $R^0$, and $\pi^0$, the \emph{bounded} solution $(R, \pi)$ exists and is unique and that the parameter $h = h(R_0, R^0, \pi^0)$ is determined by the condition that the solution should have finite $X$ seminorm. There exist a unique value of $h$ for which the solution is bounded and a unique solution $(R, \pi)$ corresponding to that value of $h$. Thus the map $\Upsilon$ is well-defined.

Consider the space
\be\begin{aligned}
X &= \{(R, \pi) \mid R \in L^{\infty}_t \dot H^{1/2}_x \cap L^2_t \dot W^{1/2, 6}_x,\ \dot \pi \in L^1\},
\lb{X}
\end{aligned}\ee
with the seminorm
\be
\|(R, \pi)\|_X = \|R\|_{L^{\infty}_t \dot H^{1/2}_x \cap L^2_t \dot W^{1/2, 6}_x} + \alpha_0^{-1} \|\dot \pi\|_{L^1_t},
\ee
where $\alpha_0$ is the scaling parameter, see (\ref{Z_initial}). $X$ is the natural space for the study of equation (\ref{NLS}) and of its linearized version (\ref{ec_liniara}). Indeed, since the Schr\"{o}dinger equations (\ref{NLS}) and (\ref{rr}) are $\dot H^{1/2}$-critical, we need to study them in the critical Strichartz space $L^{\infty}_t \dot H^{1/2}_x \cap L^2_t \dot W^{1/2, 6}_x$. In regard to the modulation path, $\dot \pi \in L^1$ is a minimal assumption to ensure we are dealing with small perturbations --- and there is no room for a stronger condition, due to working in a critical space.

We prove that, for $\|(R^0, \pi^0)\|_X < \delta$, it follows that $\|(R, \pi)\|_X = \|\Upsilon(R^0, \pi^0)\|_X < \delta$ as well, when $\delta$ is small. For clarity, we state the stability result formally:
\begin{proposition}\lb{prop9}
Let $(R, \pi)$ be a solution of (\ref{ec_liniara}) with initial data given by (\ref{Z_initial}). Assume that the auxiliary functions $(R^0, \pi^0)$ of (\ref{ec_liniara}) also satisfy (\ref{Z_initial}) and, for some sufficiently small $\delta \leq \delta_0$, $\|(R^0, \pi^0)\|_X \leq \delta$. Then $(R, \pi) \in X$ for a unique value of $h$, which we denote $h(R_0, R^0, \pi^0)$, and this bounded solution satisfies
\be\begin{aligned}
\|R\|_{L^{\infty}_t \dot H^{1/2}_x \cap L^2_t \dot W^{1/2, 6}_x} &\leq C(\|R_0\|_{\dot H^{1/2}} + \delta^2), \\
\|\dot \pi\|_{L^1_t} + |h(R_0, R^0, \pi^0)| &\leq C \alpha_0 (\delta \|R_0\|_{\dot H^{1/2}} + \delta^2).
\end{aligned}\ee
\end{proposition}

Supposing that $\|R_0\|_{\dot H^{1/2}} \leq c \delta$ and $\delta$ is small, it follows that the set
\be
\{(R, \pi) \mid R(0) = R_0 + h F^+(W_0),\ \pi(0) = \pi_0,\ \|R, \pi\|_X \leq \delta\}
\ee
is stable under the action of the map $\Upsilon$ and that
\be
|h(R_0, R^0, \pi^0)| \leq C \delta^2.
\ee
To a first order, $R$ itself is then given by the time evolution of $R_0$ under $\mc H_{\pi^0}(t)$ and all other terms in its composition are of higher order (of size $\delta^2$).

We claim the previous result for each $W_0$ and sufficiently small $\delta$, but are also interested in the dependence of $\delta$ on $W_0$. Since the Strichartz norms are scaling-invariant, they are left unchanged by symmetry transformations; $\pi$ and $h$ scale like $\alpha_0$. After accounting for this fact, we are left with an invariant statement. 

\begin{proof}
The modulation path $\pi^0:\set [0, \infty) \to \set R^8$
determines the moving soliton $w_{\pi^0}$, following (\ref{w0}), and the time-dependent Hamiltonian $\mc H_{\pi^0}(t)$, as per Lemma \ref{lemma1.1}.

Having in view Theorem \ref{theorem_13} and the preceding discussion in Section \ref{sec3.7}, recall the notations
\be\begin{aligned}
w(\pi^0(t)) &= e^{i(\Gamma^0(t) + v^0(t) \cdot x)} \phi(x-D^0(t), \alpha^0(t)), \\
W(\pi^0(t)) &= \bpm w(\pi^0(t)) \\ \ov w(\pi^0(t)) \epm,
\end{aligned}\ee
and consider the family of isometries
\be\lb{2.63}
U(t) = e^{\textstyle\int_0^t(2 v^0(s) \dl + i ((\alpha^0(s))^2-|v^0(s)|^2) \sigma_3) \dd s}.
\ee
Observe that
\be\begin{aligned}
W(\pi^0(t)) &= U(t) W_{\pi^0}(t), \\
\mc H(W(\pi^0(t))) &= U(t)^{-1} \mc H(W_{\pi^0}(t)) U(t),
\end{aligned}\ee
and let
\be
Z(t) = U(t) R(t).
\ee
For brevity we denote
\be\begin{aligned}\lb{2.67}
\mc H(t) &= \mc H(W(\pi^0(t))), &P_0(t) &= P_0(W(\pi^0(t))),\\
P_c(t) &= P_c(W(\pi^0(t))), &P_{im}(t) &= P_{im}(W(\pi^0(t))),\\
P_+(t) &= P_+(W(\pi^0(t))), &P_-(t) &= P_-(W(\pi^0(t))),\\
F^+(t) &= F^+(W(\pi^0(t))), &F^-(t) &= F^-(W(\pi^0(t))),\\
\Xi_f(t) &= \Xi_f(W(\pi^0(t))), &\sigma(t) &= \sigma(W(\pi^0(t))).
\end{aligned}\ee
Rewriting the equation (\ref{ec_liniara}) from Lemma \ref{lemma1.1}
\be
i \partial_t R - \mc H_{\pi^0}(t) R = F(t)
\ee
with $Z$ as the unknown, we obtain
\be\lb{2.69}
i \partial_t Z - \mc H(t) Z = U(t) F(t).
\ee

Instead of fixing a time-independent Hamiltonian for this equation, we consider the time-dependent Hamiltonian $\mc H(t)$ and divide the equation for $Z$ into three parts, according to the three components of the spectrum --- continuous, null, and imaginary. Let
\be\begin{aligned}
I &= P_c(t) + P_0(t) + P_{im}(t),\ P_{im}(t) &= P_+(t) + P_-(t).
\end{aligned}\ee
Then, we separately prove, for each of the three components, estimates that enable us to carry out the contraction scheme.

The $P_0$ component is the most straightforward. By (\ref{Z_initial}), the orthogonality condition
\be
\langle R(t), \Xi_{f}(W_{\pi^0}(t))\rangle = 0
\ee
holds at time $t=0$ and the modulation equations (\ref{ec_liniara}) then imply that the orthogonality condition still holds at any other time $t$. Applying the isometry $U(t)$, the orthogonality condition turns into
\be
\langle Z(t), \Xi_{f}(t)\rangle = 0,
\ee
which directly implies that
\be\lb{2.68}
P_0(t) Z(t) = 0.
\ee

The continuous spectrum projection of $Z$ fulfills the equation, derived from (\ref{2.69}),
\begin{align}
i\partial_t(P_c(t) Z) + \mc H(t) P_c(t) Z &= P_c(t) U(t) F(t) + i (\partial_t P_c(t)) Z.
\lb{disp}
\end{align}
The right-hand side term
\be
F=-i L_{\pi^0} R + N(R^0, \pi^0) - N_{\pi^0}(R^0, \pi^0)
\ee
is bounded in the dual Strichartz norm by means of the fractional Leibniz rule:
\be\begin{aligned}
\|L_{\pi^0} R\|_{L^2_t \dot W^{1/2, 6/5}_x} &\leq C \delta \|R\|_{L^2_t \dot W^{1/2, 6}_x} = C \delta \|Z\|_{L^2_t \dot W^{1/2, 6}_x},\\
\|N(R^0, \pi^0) - N_{\pi^0}(R^0, \pi^0)\|_{L^2_t \dot W^{1/2, 6/5}_x} &\leq C \delta^2.
\lb{2.49}
\end{aligned}\ee
When using the fractional Leibniz rule, there is an endpoint Sobolev embedding issue, namely that $\dot W^{1/2, 6}$ does not embed into $L^{\infty}$. However, when estimating the quadratic and cubic terms, such as $|r^0|^2 r^0$, present in $N(R^0, \pi^0)$, we can avoid the problem by putting $r^0$ in non-endpoint Strichartz spaces:
\be\begin{aligned}
\|(r^0)^2 w(\pi^0)\|_{L^2_t \dot W^{1/2, 6/5}_x} &\leq C\|(r^0)^2\|_{L^2_t \dot H^{1/2}_x} \|w(\pi^0)\|_{L^{\infty}_t \dot H^{1/2}_x} \\
&\leq C \|R^0\|_{L^4_t W^{1/2, 3}_x}^2 \|W(\pi^0)\|_{L^{\infty}_t \dot H^{1/2}_x}.
\end{aligned}\ee
The same applies to the cubic term, except that there we replace $w(\pi^0)$ by $r^0$.

The endpoint Strichartz space intervenes only when evaluating linear terms such as $L_{\pi^0} R$. However, there we avoid the attendant endpoint Sobolev embedding issue by means of a less sharp fractional Leibniz rule, which we can use because solitons are of Schwartz class.

Furthermore,
\be\begin{aligned}
(\partial_t P_+(t)) Z &= (\partial_t \alpha^{-3}) \langle Z, i \sigma_3 F^-(t) \rangle F^+(t) + \alpha^{-3} \langle Z, i \sigma_3 \partial_t F^-(t) \rangle F^+(t) + \\
&+ \alpha^{-3} \langle Z, i \sigma_3 F^-(t) \rangle \partial_t F^+(t) \\
&= -3 \alpha^{-4} \dot \alpha \langle Z, i \sigma_3 F^-(t) \rangle F^+(t) + \alpha^{-3} \langle Z, i \sigma_3 d_{\pi} F^-(t) \dot \pi^0(t) \rangle F^+(t) + \\
&+ \alpha^{-3} \langle Z, i \sigma_3 F^-(t)\rangle d_{\pi} F^+(t) \dot \pi^0(t)
\end{aligned}\ee
and likewise for $P_-$ and $P_0$. Since $P_c = I - P_0 - P_{im}$, by direct examination it follows that
\be\begin{aligned}\lb{2.72}
\|(\partial_t P_{\pm}(t)) Z\|_{L^2_t \dot W^{1/2, 6/5}_x} &\leq C \delta \|Z\|_{L^2_t \dot W^{1/2, 6}_x}.
\end{aligned}\ee

Provided $\|\dot \pi^0\|_1 < \delta$ is sufficiently small, Theorem \ref{theorem_13} leads to Strichartz estimates for $P_c(t) Z$, by means of the following construction. $\mc H(t)$ is given by
\be
\mc H(t) = \bpm \Delta + 2|w(\pi^0(t))|^2 & w^2(\pi^0(t)) \\ -\ov w^2(\pi^0(t)) & -\Delta - 2|w(\pi^0(t))|^2\epm + 2iv^0(t) \dl - ((\alpha^0(t))^2 - |v^0(t)|^2) \sigma_3.
\ee
Denote, for $\pi^0(0) = \pi_0$ following (\ref{Z_initial}),
\be
\tilde {\mc H}(t) = \bpm \Delta + 2|w(\pi_0)|^2 & w^2(\pi_0) \\ -\ov w^2(\pi_0) & -\Delta - 2|w(\pi_0)|^2\epm  + 2iv^0(t) \dl - ((\alpha^0(t))^2 - |v^0(t)|^2) \sigma_3.
\ee
The difference $\tilde {\mc H} - \mc H$ is small in the appropriate $\dot W^{1/2, 6/5-\epsilon} \cap \dot W^{1/2, 6/5+\epsilon}$ norm (in much stronger norms too), so the corresponding term can be bounded by means of endpoint Strichartz estimates:
\be\begin{aligned}\lb{2.73}
\|(\tilde {\mc H} - \mc H) Z\|_{L^2_t \dot W^{1/2, 6/5}_x} &\leq \|\tilde {\mc H}(t) - \mc H(t)\|_{L^{\infty}_t \dot W^{1/2, 6/5-\epsilon} \cap \dot W^{1/2, 6/5+\epsilon}} \|Z\|_{L^2_t \dot W^{1/2, 6}_x} \\
&\leq C \|\pi^0(t)-\pi_0\|_{L^{\infty}_t} \|Z\|_{L^2_t \dot W^{1/2, 6}_x} \\
&\leq C \|\dot \pi^0\|_1 \|Z\|_{L^2_t \dot W^{1/2, 6}_x} \\
&\leq C \delta \|Z\|_{L^2_t \dot W^{1/2, 6}_x}.
\end{aligned}\ee
Then, by (\ref{2.49}), (\ref{2.72}), and (\ref{2.73}), for sufficiently small $\delta$ and under the assumption that there are no embedded eigenvalues we can apply Theorem \ref{theorem_13} with the Hamiltonian $\tilde{\mc H}$ and obtain
\be\begin{aligned}\lb{2.76}
\|P_c(t) Z\|_{L^{\infty}_t \dot H^{1/2}_x \cap L^2_t \dot W^{1/2, 6}_x} 
&\leq C (\|Z(0)\|_{\dot H^{1/2}} + \delta^2 + \delta \|Z\|_{L^2_t \dot W^{1/2, 6}_x}).
\end{aligned}\ee

For the imaginary spectrum projection, write, following Section \ref{spectru},
\be
P_{im}(t) Z(t) = b^+(t) F^+(t) + b^-(t) F^-(t).
\lb{partea_finita}
\ee
Then, by taking the time derivative of
\be
b^{\pm}(t) = (\alpha^0(t))^{-3} \langle Z(t), i \sigma_3 F^{\mp}(t) \rangle,
\ee
we obtain
\be\begin{aligned}
\partial_t b^{\pm} &= (\alpha^0)^{-3} \langle \partial_t Z, i \sigma_3 F^{\mp} \rangle -3 \dot \alpha^0 (\alpha^0)^{-4} \langle Z, i \sigma_3 F^{\mp}(t) \rangle + \\
&+ (\alpha^0)^{-3} \langle Z, i \sigma_3 (d_{\pi} F^{\mp}) \dot \pi^0 \rangle.
\end{aligned}\ee
Substituting $\partial_t Z$ by its expression given by equation (\ref{ec_liniara}), we arrive at
\be\begin{aligned}\lb{hiperb}
\partial_t b^{\pm} &= \pm \sigma(t) b^{\pm} - \langle U F, \sigma_3 F^{\mp} \rangle - \\
& -3 \dot \alpha^0 (\alpha^0)^{-4} \langle Z, i \sigma_3 F^{\mp} \rangle +  (\alpha^0)^{-3} \langle Z, i \sigma_3 (d_{\pi} F^{\mp}) \dot \pi^0 \rangle.
\end{aligned}\ee
Thus $b^+$ and $b^-$ satisfy the equation
\be
\partial_t  \bpm b_-\\ b_+\epm + \bpm \sigma(t) & 0 \\ 0 & -\sigma(t)\epm \bpm b_-\\ b_+\epm = \bpm N_-(t) \\ N_+(t) \epm,
\lb{hiper}
\ee
where
\be
N_{\pm} = -\langle U F, \sigma_3 F^{\mp} \rangle -3 \dot \alpha^0 (\alpha^0)^{-4} \langle Z, i \sigma_3 F^{\mp} \rangle +  (\alpha^0)^{-3} \langle Z, i \sigma_3 (d_{\pi} F^{\mp}) \dot \pi^0 \rangle.
\ee
Here $\pm i\sigma(t)$ are the imaginary eigenvalues of $\mc H(t)$, as in our discussion of its spectrum in Section \ref{spectru}.

Concerning the right-hand side, (\ref{hiperb}) and (\ref{2.49}) imply
\be
\|N_{\pm}(t)\|_{L^2_t} \leq C (\delta \|Z(t)\|_{L^2_t \dot W^{1/2, 6}_x} + \delta^2).
\ee

To control the solution we use the following simple fact, see \cite{schlag}. It characterizes the bounded solution of the ordinary differential equation system (\ref{hiper}).

\begin{lemma}\lb{hyp}
Consider the equation
\be
\dot x - \bpm \sigma(t) & 0 \\ 0 & -\sigma(t)\epm x = f(t),
\ee
where $x = \bpm x_1 \\ x_2 \epm$ is the unknown, $f = \bpm f_1 \\ f_2 \epm \in L^1_t \cap L^{\infty}_t$, and $\sigma(t) \geq \sigma_0 > 0$ is bounded from below. Then $x$ is bounded on $[0, \infty)$ if and only if
\be
0 = x_1(0) + \int_0^{\infty} e^{-\int_0^t \sigma(\tau) \dd \tau} f_1(t) \dd t.
\lb{condi}
\ee
In this case, for all $t \geq 0$
\be\begin{aligned}\lb{2.83}
x_1(t) &= -\int_t^{\infty} e^{\int_s^t \sigma(\tau) \dd \tau} f_1(s) \dd s, \\
x_2(t) &= e^{-\int_0^t \sigma(\tau) \dd \tau} x_2(0) + \int_0^t e^{-\int_s^t \sigma(\tau) \dd \tau} f_2(s) \dd s.
\end{aligned}\ee
\end{lemma}
\begin{proof}
Any solution will be a linear combination of the exponentially increasing and the exponentially decaying ones and we want to make sure that the exponentially increasing one is absent. It is always true that
\be\begin{aligned}
x_1(t) &= e^{\int_0^t \sigma(\tau) \dd \tau} \Big(x_1(0) + \int_0^t e^{-\int_0^s \sigma(\tau) \dd \tau} f_1(s) \dd s\Big), \\
x_2(t) &= e^{-\int_0^t \sigma(\tau) \dd \tau} x_2(0) + \int_0^t e^{-\int_s^t \sigma(\tau) \dd \tau} f_2(s) \dd s.
\end{aligned}\ee
Thus, if $x_1$ is to remain bounded, the expression between parantheses must converge to $0$, hence (\ref{condi}). Conversely, if (\ref{condi}) holds, then
\be
x_1(t) = -\int_t^{\infty} e^{\int_s^t \sigma(\tau) \dd \tau} f_1(s) \dd s
\ee
tends to $0$.
\end{proof}

Note that $\sigma$ depends Lipschitz continuously on the scaling parameter $\alpha$, in an explicit manner, following (\ref{2.50}). Then $\sigma(t)$ belongs to a compact subset $[a_1, a_2]$ of $(0, \infty)$, because $\alpha^0(t)$ belongs to a compact subset of $(0, \infty)$. Consequently, equation (\ref{hiper}) has a bounded solution if and only if
\be
0 = b_+(0) + \int_0^{\infty} e^{-\int_0^t \sigma(\tau) \dd \tau} N_+(t) \dd t.
\lb{conditie}
\ee

One sees that $b_+(0) = h$, where $b_+(0)$ is given by (\ref{conditie}) and $h$ by (\ref{Z_initial}). $R$ is globally bounded in time if and only if $Z$ is bounded. $Z$ is bounded if and only if each of its components is bounded, $P_{im} Z$ in particular. Thus $R$ is bounded only if
\be\lb{hh}
h = -\int_0^{\infty} e^{-\int_0^t \sigma(\tau) \dd \tau} N_+(t) \rangle \dd t.
\ee
We are interested in a more direct formula for $h$, one that involves $R$ instead of $Z$. Note that $\sigma(\tau)$ is also the imaginary eigenvalue of $\mc H(W_{\pi^0}(\tau))$. Expanding $N_+$ and reverting the isometry $U$ within (\ref{hh}) leads to the explicit formula
\be\begin{aligned}\lb{2.100}
h &= -\int_0^{\infty} e^{-\int_0^t \sigma(W_{\pi^0}(\tau)) \dd \tau} \big(\langle F,  \sigma_3 F^-(W_{\pi^0}(t)) \rangle - \\
&- 3 \dot \alpha^0(t) (\alpha^0(t))^{-4} \langle R, i \sigma_3 F^-(W_{\pi^0}(t)) \rangle + \\
&+(\alpha^0(t))^{-3} \langle R, i \sigma_3 (d_{\pi} F^-(W_{\pi^0}(t))) \dot \pi^0(t) \rangle\big) \dd t.
\end{aligned}\ee
It remains to show that, for this unique value of $h$, $R$ is indeed bounded.


We obtain
\be\begin{aligned}
|h| &\leq C \int_0^{\infty} e^{-\int_0^t \sigma(\tau) \dd \tau} |N_+(t)| \dd t \\
&\leq C \int_0^{\infty} e^{-t a_1} |N_+(t)| \dd t \\
&\leq C\|N_+\|_{L^1_t + L^{\infty}_t} \\
&\leq C (\delta \|Z\|_{L^2_t \dot W^{1/2, 6}_x} + \delta^2).
\end{aligned}\ee
Following (\ref{2.83}), both $b_{+}$ and $b_-$ are given by convolutions with exponentially decaying kernels in $t$, whose rate of decay is bounded from below:
\be\begin{aligned}
|b_+(t)| &\leq \int_t^{\infty} e^{(t-s)a_1} |N_+(s)| \dd s,\\
|b_-(t)| &\leq \int_{-\infty}^t e^{-(t-s)a_1} |N_-(s)| \dd s + e^{-ta_1} \|R_0\|_{\dot H^{1/2}},
\end{aligned}\ee
with the convention that $N(s) = 0$ for $s<0$; the second term for $b_{-}(t)$ stems from $e^{-t\sigma} b_-(0)$. One has
\be\begin{aligned}\lb{2.90}
\|P_{im}(t) Z\|_{L^2_t \dot W^{1/2, 6}_x} &\leq \|b_{+}\|_{L^2_t} + \|b_{-}\|_{L^2_t} \\
&\leq C (\|N_+\|_{L^2_t} + \|N_-\|_{L^2_t} + \|R_0\|_{\dot H^{1/2}}) \\
&\leq C (\|R_0\|_{\dot H^{1/2}} + \delta \|R\|_{L^2_t \dot W^{1/2, 6}_x} + \delta^2).
\end{aligned}\ee
Putting all three estimates (\ref{2.68}), (\ref{2.76}), and (\ref{2.90}) together and taking into account the fact that
\be
\|Z\|_{L^{\infty}_t \dot H^{1/2}_x \cap L^2_t \dot W^{1/2, 6}_x} = \|R\|_{L^{\infty}_t \dot H^{1/2}_x \cap L^2_t \dot W^{1/2, 6}_x},
\ee
we obtain
\be
\|R\|_{L^2_t \dot W^{1/2, 6}_x} \leq C (\|R_0\|_{\dot H^{1/2}} + \delta \|R\|_{L^2_t \dot W^{1/2, 6}_x} + \delta^2).
\ee
Concerning the modulation path $\pi$, from the modulation equations (\ref{ec_liniara}) we get that
\be
\|\dot \pi\|_1 \leq C (\delta \|Z\|_{L^2_t \dot W^{1/2, 6}_x} + \delta^2).
\ee
Overall,
\be
\|(R, \pi)\|_X = \|(Z, \pi)\|_X \leq C (\|R_0\|_{\dot H^{1/2}_x} + \delta \|(R, \pi)\|_X + \delta^2)
\ee
and this proves stability for the suitable choice of $h$ and small initial data $R_0$.
\end{proof}

\section{The fixed point argument: contraction}\lb{Chapter_2.4}
The parameter $\delta$ was chosen such that the scaling parameter $\alpha(t)$ belongs to a fixed compact subset of $(0, \infty)$ and therefore the imaginary eigenvalue $i\sigma$ fulfills
\be
\sigma \in [a_1, a_2] \subset (0, \infty),
\ee
for all the admissible paths that we consider. Fix, then, a constant $\rho \in (0, a_1)$.

For any two bounded solutions of the linearized equation (\ref{ec_liniara}), $(R_j, \pi_j) = \Upsilon(R_j^0, \pi_j^0)$, $j=1, 2$, located in $X$, such that
\be
\|(R_j^0, \pi_j^0)\|_X \leq \delta,
\ee
we prove that $\Upsilon$ acts as a contraction in the following space $Y$:
\be\begin{aligned}
Y= \{(R, \pi) \mid &\|e^{-t\rho} R(t)\|_{L^{\infty}_t \dot H^{1/2}_x \cap L^2_t \dot W^{1/2, 6}_x} + \|e^{-t\rho} \dot \pi(t)\|_{L^1_t} < \infty \}.
\lb{Y}
\end{aligned}\ee

This is only a seminorm, but defines a metric space for fixed initial data.

Furthermore, for fixed $R_0$ we prove that the unique value of $h$, $h(R_0, R^0, \pi^0)$, for which the solution with initial data (\ref{Z_initial}) is bounded, satisfies
\be
|h(R_0, R^0_1, \pi^0_1) - h(R_0, R^0_2, \pi^0_2)| \leq C \delta \|(R_1^0, \pi_1^0) - (R_2^0, \pi_2^0)\|_Y.
\ee

Observe that this enough to complete the proof. For initial data given by (\ref{Z_initial}) and fixed initial data $R_0$ and $\pi_0$, consider a sequence
\be
(R_n, \pi_n) = \Upsilon((R_{n-1}, \pi_{n-1})),\ \|(R_n, \pi_n)\|_X \leq \delta
\ee
which converges in the $Y$ sense to $(R, \pi)$; the parameters $h_n = h(R_0, R^0_n, \pi^0_n)$ form a Cauchy sequence as well. Then the pair $(R, \pi)$ is a fixed point of $\Upsilon$ and, by virtue of Lemma \ref{lemma1.1}, a solution to the nonlinear equation (locally in time in a weak sense and therefore globally as well) with the specified initial data and, furthermore,
\be
\|(R, \pi)\|_X \leq \lim \sup \|(R_n, \pi_n)\|_X \leq \delta.
\ee
This follows first on any finite time interval $[0, T]$ and then in the limit on $[0, \infty)$.

We seek to prove that for any sufficiently small choice of $\delta$
\be
\|(R_1, \pi_1) - (R_2, \pi_2)\|_Y \leq 1/2 \|(R_1^0, \pi_1^0) - (R_2^0, \pi_2^0)\|_Y.
\ee

Letting $R_1$ and $R_2$ start from distinct initial data proves useful, leading to the formulation of the following perturbation lemma, which we employ repeatedly:
\begin{lemma}\lb{lemma_5}
Consider two solutions, $(R_1, \pi_1)$ and $(R_2, \pi_2)$, of two distinct linearized equations in the form (\ref{ec_liniara}):
\be\begin{aligned}
&i \partial_t R_j + \mc H_{\pi^0_j}(t) R_j = F_j,\ F_j= -i L_{\pi^0_j} R_j + N(R^0_j, W_{\pi^0_j}) - N_{\pi^0_j}(R^0_j, W_{\pi^0_j})\\
&\dot f_j = 4\alpha^0_j \|W_{\pi^0_j}\|_2^{-2} (\langle R_j, (d_{\pi} \Xi_{f}(W_{\pi^0_j})) \dot \pi^0_j \rangle - i\langle N(R^0_j, W_{\pi^0_j}), \Xi_{f}(W_{\pi^0_j})\rangle), f \in \{\alpha, \Gamma\}\\
&\dot f_j = 2\|W_{\pi^0_j}\|_2^{-2} (\langle R_j, (d_{\pi} \Xi_{f}(W_{\pi^0_j})) \dot \pi^0_j \rangle - i\langle N(R^0_j, W_{\pi^0_j}), \Xi_{f}(W_{\pi^0_j})\rangle), f \in \{v_k, D_k\}.
\end{aligned}\lb{ec_liniara12}\ee
for $j = \ov{1, 2}$, with initial data
\be
R_j(0) = R_{0j} + h_j F^+(W(\pi_{0})),\ \pi_j(0) = \pi^0_j(0)=\pi_{0} \text{ given}.
\ee
Assume in addition that $\|(R^0_j, \pi^0_j)\|_X \leq \delta$ and
\be
(R_j, \pi_j) = \Upsilon((R^0_j, \pi^0_j)),\ h_j = h(R_{0j}, R^0_j, \pi^0_j),
\ee
meaning that $h_j$ hold the unique values that make $(R_j, \pi_j)$ the bounded solutions of (\ref{ec_liniara12}). Then, assuming $\delta>0$ is sufficiently small,
\be\begin{aligned}\lb{2.117}
\|(R_1, \pi_1) - (R_2, \pi_2)\|_Y &\leq C\delta \|(R_1^0, \pi_1^0) - (R_2^0, \pi_2^0)\|_Y + C \|R_{01}-R_{02}\|_{\dot H^{1/2}}, \\
|h_1-h_2| &\leq C\delta (\|(R_1^0, \pi_1^0) - (R_2^0, \pi_2^0)\|_Y + \|R_{01}-R_{02}\|_{\dot H^{1/2}}).
\end{aligned}\ee
\end{lemma}
In this formulation, $\delta$ and various constants denoted by $C$ may depend on the value of $\rho$ in the definition of $Y$, (\ref{Y}), which in turn depends on the scaling parameter $\alpha$.

\begin{proof}
Let $R = R_1-R_2$, $\pi = \pi_1-\pi_2$.
$R_j$, $j = 1, 2$, satisfy the equations
\be
i \partial_t R_j + \mc H_{\pi_j^0}(t) R_j = F_j,
\ee
with initial data
\be
R_j(0) = R_{0j} + h_j F^+(W(\pi_0)).
\ee
Note that, by the preceding stability result, Proposition \ref{prop9}, $\|(R_j, \pi_j)\|_X \leq \delta$ for $j=1, 2$. Furthermore, $R_j$ satisfy the orthogonality conditions
\be
\langle R_j(t), \Xi_f(W_{\pi^0_j}(t)) \rangle = 0
\ee
at time $t=0$ and thus, due to equation (\ref{ec_liniara12}), at every time $t$.

Subtracting the linearized equations from one another, we obtain a similar one for the difference $R=R_1-R_2$:
\be
i \partial_t R + \mc H_{\pi^0_1}(t) R = \tilde F,\ \tilde F = F_1 - F_2 - (\mc H_{\pi^0_1}(t) - \mc H_{\pi^0_2}(t)) R_2.
\ee
We choose the Hamiltonian $\mc H_{\pi_1^0}(t)$ (the choice of one or two is arbitrary) and apply the same isometry $U(t)$ as described by (\ref{2.63}-\ref{2.67}):
\be
U(t) = e^{\textstyle\int_0^t(2 v_1^0(s) \dl + i ((\alpha_1^0)^2(s)-|v_1^0(s)|^2) \sigma_3) \dd s}.
\ee
Let
\be
Z(t) = U(t) R(t).
\ee
We reintroduce notations similar to (\ref{2.68}), namely
\be\begin{aligned}
\mc H(t) &= \mc H(W(\pi^0_1(t))), &P_0(t) &= P_0(W(\pi^0_1(t))),\\
P_c(t) &= P_c(W(\pi^0_1(t))), &P_{im}(t) &= P_{im}(W(\pi^0_1(t))),\\
P_+(t) &= P_+(W(\pi^0_1(t))), &P_-(t) &= P_-(W(\pi^0_1(t))),\\
F^+(t) &= F^+(W(\pi^0_1(t))), &F^-(t) &= F^-(W(\pi^0_1(t))),\\
\Xi_f(t) &= \Xi_f(W(\pi^0_1(t))), &\sigma(t) &= \sigma(W(\pi^0_1(t))).
\end{aligned}\ee
Rewritten for $Z$, the equation becomes
\be
i \partial_t Z + \mc H(t) Z = U(t) \tilde F.
\ee
Then we split $Z$ into three parts, according to the Hamiltonian's spectrum:
\be
Z = P_c(t) Z + P_0(t) Z + P_{im}(t) Z.
\ee
We bound each component of $Z$ in the same manner as in the previous section. The main difference is related to $P_0(t)$, which need no longer be zero because the orthogonality condition does not hold.

For the right-hand side $\tilde F = F_1 - F_2 - (\mc H_{\pi^0_1}(t) - \mc H_{\pi^0_2}(t)) R_2$ we estimate the difference term by term, arriving at the following estimate:
\begin{multline}
\|\tilde F\|_{e^{t\rho} L^2_t \dot W^{1/2, 6/5}_x} \leq C \delta (1+ \rho^{-2}) (\|(R, \pi)\|_Y +  \|(R_1^0, \pi_1^0) - (R_2^0, \pi_2^0)\|_Y).
\lb{2.79}
\end{multline}
Here we have taken advantage of the exponential weight, as integrating in time preserves the space $e^{t\rho} L^p_t$, at the cost of a factor of $\rho^{-1}$ :
\be\textstyle
\big\|\int_0^t f(s) \dd s \big\|_{e^{t\rho} L^p_t} \leq C \rho^{-1} \|f\|_{e^{t\rho} L^p_t}.
\ee
In particular, one has in any Schwartz seminorm $S_x$, in the space variables, that
\be
\|W_{\pi^0_1}(t) - W_{\pi^0_2}(t)\|_{L^{\infty}_t S_x} \leq C (1+\rho^{-1}) (\|(R_1^0, \pi_1^0) - (R_2^0, \pi_2^0)\|_Y
\ee

The initial data are given by
\be
Z(0) = (h_1-h_2) F^+(W(\pi_{0})) + R_{01}-R_{02}
\ee
and satisfy the bound
\be
\|Z(0)\|_{\dot H^{1/2}} \leq C(|h_1-h_2| + \|R_{01}-R_{02}\|_{\dot H^{1/2}}).
\ee

It is straightforward to bound the continuous spectrum projection, $P_c(t) Z$, which satisfies the equation
\be\begin{aligned}
i\partial_t(P_c(t) Z) + \mc H(t) P_c(t) Z &= P_c(t) U(t) \tilde F + i (\partial_t P_c(t)) Z.
\end{aligned}\ee
Starting from the usual Strichartz estimates of Theorem \ref{theorem_13} and integrating against the exponential weight $e^{t\rho}$, we obtain Strichartz estimates with this weight:
%
\be\begin{aligned}
&\|P_c(t) Z\|_{e^{t\rho} (L^{\infty}_t \dot H^{1/2}_x \cap L^2_t \dot W^{1/2, 6/5}_x)} \\
\leq &C \Big(\|P_c(0) Z(0)\|_{\dot H^{1/2}} + \|\tilde F\|_{e^{t\rho} L^2_t \dot W^{1/2, 6}_x} +  \|\partial_t P_c(t)) Z\|_{e^{t\rho} L^2_t \dot W^{1/2, 6}_x}\Big) \\
\leq &C \|R_{01}-R_{02}\|_{\dot H^{1/2}} +C \delta (1+\rho^{-1}) (\|(R_1^0, \pi_1^0) - (R_2^0, \pi_2^0)\|_Y + \|(R, \pi)\|_Y).
\lb{2.71}
\end{aligned}\ee
There is no contribution due to $h_1-h_2$, since $P_c(0) Z(0)$ does not depend on $h_1-h_2$.

$R$ satisfies no orthogonality condition, but $R_1$ and $R_2$ do. Taking the difference of these two relations leads to
\be
\langle R_1(t) - R_2(t), \Xi_{f}(W_{\pi^0_1}(t)) \rangle = \langle R_2(t), \Xi_{f}(W_{\pi^0_2}(t)) - \Xi_{f}(W_{\pi^0_1}(t)) \rangle.
\lb{96}
\ee
Applying the isometry $U$, it follows that
\be
\langle Z(t), \Xi_{f}(t) \rangle = \langle R_1(t) - R_2(t), \Xi_{f}(W_{\pi^0_1}(t)) \rangle
\ee
and therefore
\be\lb{2.75}
\|P_0(t) Z(t)\|_{e^{t\rho} (L^{\infty}_t \dot H^{1/2}_x \cap L^2_t \dot W^{1/2, 6/5}_x)} \leq C \delta (1+\rho^{-2}) \|(R_1^0, \pi_1^0) - (R_2^0, \pi_2^0)\|_Y.
\ee

Concerning the imaginary component, Lemma \ref{hyp} applies again, since $Z$ is globally bounded forward in time, as seen from
\be
\|P_{im}(t) Z(t)\|_{\dot H^{1/2}} \leq \|R_1(t)\|_{\dot H^{1/2}} + \|R_2(t)\|_{\dot H^{1/2}} < C < \infty.
\ee
Write, following Section \ref{spectru},
\be
P_{im}(t) Z(t) = b^+(t) F^+(t) + b^-(t) F^-(t).
\ee
Then, as in the proof of the stability result, we infer
\be\begin{aligned}
\partial_t b^{\pm} &= \pm \sigma(t) b^{\pm} - \langle U \tilde F, \sigma_3 F^{\mp} \rangle - \\
& -3 \dot \alpha^0 (\alpha^0)^{-4} \langle Z, i \sigma_3 F^{\mp} \rangle +  (\alpha^0)^{-3} \langle Z, i \sigma_3 (d_{\pi} F^{\mp}) \dot \pi^0 \rangle.
\end{aligned}\ee
Thus $b^+$ and $b^-$ satisfy the equation
\be
\partial_t  \bpm b_-\\ b_+\epm + \bpm \sigma(t) & 0 \\ 0 & -\sigma(t)\epm \bpm b_-\\ b_+\epm = \bpm N_-(t) \\ N_+(t) \epm,
\ee
where
\be
N_{\pm} = -\langle U \tilde F, \sigma_3 F^{\mp} \rangle -3 \dot \alpha^0_1 (\alpha^0_1)^{-4} \langle Z, i \sigma_3 F^{\mp} \rangle +  (\alpha^0_1)^{-3} \langle Z, i \sigma_3 (d_{\pi} F^{\mp}) \dot \pi^0_1 \rangle.
\ee
Then,
\be
\|N_{\pm}\|_{e^{t\rho} L^2_t} \leq C \delta \|(R, \pi)\|_Y.
\ee
As a consequence of the lemma, we obtain that $P_{im}(t) Z$ is in the same space as $N_{\pm}$, as a function of time. Indeed, here we emphasize that since $\rho<a_1<\sigma(t)$, convolution with $e^{-a_1|t|}$ preserves the space $e^{t\rho} L^2_t$. One gets that
\be\begin{aligned}
|h_1-h_2| = |b_+(0)| &\leq C \int_0^{\infty} e^{-t a_1} |N_+(t)| \dd t \\
&\leq C\delta (\|(Z_1^0, \pi_1^0) - (Z_2^0, \pi_2^0)\|_Y+ \|Z\|_{e^{t\rho} \dot H^{1/2}_x}).
\end{aligned}\ee
Also through Lemma \ref{hyp}, it follows that
\be
\|P_{im} Z\|_{Y} \leq C\delta (\|(Z_1^0, \pi_1^0) - (Z_2^0, \pi_2^0)\|_Y+ \|(R, \pi)\|_Y).
\lb{2.77}
\ee

Finally, subtracting the modulation equations (\ref{ec_liniara12}), the difference between the paths, $\pi = \pi_1 - \pi_2$, also fulfills
\be\begin{aligned}
\|\dot \pi\|_{e^{t\rho} L^1_t} &\leq C \delta (1+\rho^{-2})(\|(Z_1^0, \pi_1^0) - (Z_2^0, \pi_2^0)\|_Y + \|(Z, \pi)\|_Y).
\lb{2.78}
\end{aligned}\ee
Up to this point, there was no actual need for exponential weights and we could have used polynomial weights in their stead. However, (\ref{2.78})  makes exponential weights necessary. The modulation equations involve terms of the form $\dot f$ on the left-hand side and $\int_0^t f(s) \dd s$ on the right-hand side. The only way to contain both within the same space involves exponential weights.

Putting (\ref{2.71}), (\ref{2.75}), (\ref{2.77}), and (\ref{2.78}) together we retrieve estimate (\ref{2.117}).
\end{proof}

\begin{observation}
A difference estimate also exists in the case of different starting solitons, $\pi^0_j(0)=\pi_j(0)=\pi_{0j}$, $j=1,2$. 
In order to perform a comparison, we first make the starting solitons coincide by means of a symmetry transformation and then compare the solutions in a unified setting.

More precisely, let
\be
\frak g(t) f(x, t) = e^{i(v \cdot x - t |v|^2 + \Gamma)} \alpha f(\alpha x - 2tv - D, \alpha^2 t)
\ee
be the unique symmetry transformation such that $\frak g(0)$ takes $w(\pi_{01})$ to $w(\pi_{02})$. We define accordingly the action of $\frak g$ on the modulation parameters:
\be
\frak g \pi = \tilde \pi \Longleftrightarrow \frak g w(\pi) = w(\tilde \pi).
\ee
Then, using the notation $\frak g (R, \pi) = (\frak g(t) R(t), \frak g(t) \pi(t))$, it immediately follows from the previous lemma that
\be\begin{aligned}
\|\frak g (R_1, \pi_1) - (R_2, \pi_2)\|_Y &\leq C\delta \|\frak g (R_1^0, \pi_1^0) - (R_2^0, \pi_2^0)\|_Y + C \|\frak g(0) R_{01}-R_{02}\|_{\dot H^{1/2}}, \\
|h_1-h_2| &\leq C\delta (\|\frak g (R_1^0, \pi_1^0) - (R_2^0, \pi_2^0)\|_Y + \|\frak g(0) R_{01}-R_{02}\|_{\dot H^{1/2}}).
\end{aligned}\ee
\end{observation}

\section{Analyticity of the invariant manifold}
Consider a solution of the nonlinear equation (\ref{NLS}) stemming from the contraction argument presented. For a soliton $W_0=W(\pi_0)$, it has the form
\be
R = W_0 + R_0 + h(R_0, W_0) F^+(W_0).
\ee
Here $F^+(W_0)$ is the eigenvector corresponding to the upper half-plane eigenvalue of $\mc H(W_0)$ (see Section \ref{spectru}), $R_0$ belongs to the codimension-nine vector space
\be
\mc N_0(W_0) = (P_c(W_0) + P_-(W_0)) \dot H^{1/2},
\ee
and $h(R_0, W_0)$ is the unique value determined by the contraction argument that leads to an asymptotically stable solution to (\ref{NLS}), for these initial data.

We recall again that $R$ and other vector-valued functions we are considering have the form $\bpm r \\ \ov r \epm$ and that, as a consequence, the dot product is real-valued.

At this point we give the following formal definition:
\begin{definition}\lb{def3} Let $\mc M$ be the eight-dimensional soliton manifold and
\be\begin{aligned}
\mc N_0(W) &= \{R_0 \in (P_c(W) + P_-(W)) \dot H^{1/2} \mid \|R_0\|_{\dot H^{1/2}} < \delta_0 \} \\
\mc N(W) &= \{W + R_0 + h(R_0, W) F^+(W) \mid R_0 \in \mc N_0(W)\} \\
\mc N_0 &= \{(R_0, W) \mid R_0 \in \mc N_0(W), \ W \in \mc M\} \\
\mc N &= \bigcup_{W \in \mc M} \mc N(W).
\end{aligned}\ee
\end{definition}
$\delta_0$ will be chosen independently of $W$.

The fiber bundle $\mc N_0$ is trivial over the soliton manifold. Indeed, for each soliton $W$ there exists a unique symmetry transformation $\frak g_W$ that takes $W_0$ (a fixed soliton) into $W$. Then
\be
(R_0, W) \mapsto \frak g_{W}(W_0 + R_0)
\ee
is an isomorphism between a tubular neighborhood of the base in the (trivial) product bundle $(P_c(W_0) + P_-(W_0)) \dot H^{1/2} \times \mc M$ (where $\mc M$ is the soliton manifold) and $\mc N_0$. This endows $\mc N_0$ with a real analytic manifold structure.

$\mc N$ is the image of $\mc N_0$ under the map
\be
\mc F(R_0, W) = W + R_0 + h(R_0, W) F^+(W).
\lb{2.105}
\ee

Following the contraction argument from beginning to end and giving appropriate values to $\delta$, commensurate with the size of the initial data $R_0$, we summarize the conclusion as follows (this is the gist of our main theorem):
\begin{proposition} There exists $\delta_0>0$ such that for each soliton $W_0=W(\pi_0)$ there is a map $h(\cdot, W_0): \mc N_0(W_0) \to \set R$ such that
\begin{enumerate}
\item $h$ is locally Lipschitz continuous in both variables,
\item $|h(R_0, W_0)| \leq C \alpha_0 \|R_0\|_{\dot H^{1/2}}^2$ ($\alpha_0$ is the scaling parameter of $W_0$),
\end{enumerate}
and
\be
\mc F(R_0, W_0) = W_0 + R_0 + h(R_0, W_0) F^+(W_0)
\ee
gives rise to an asymptotically stable solution $\Psi$ to (\ref{NLS}) with $\Psi(0) = \mc F(R_0, W_0)$ such that
\be
\Psi(t) = W_{\pi}(t) + R(t).
\ee
Here $W_{\pi}$ is a moving soliton with $W_{\pi}(0) = W_0$, governed in accordance to (\ref{8}) by a path $\pi$ such that
\be
\|\dot \pi\|_1 \leq C \alpha_0 \|R_0\|_{\dot H^{1/2}}^2.
\ee
$R$ is in the Strichartz space, with initial data $R(0) = R_0 + h(R_0, W_0) F^+(W_0)$, and
\be
\|R\|_{L^{\infty}_t \dot H^{1/2}_x \cap L^2_t \dot W^{1/2, 6}_x} \leq C \delta.
\ee
\end{proposition}
Adopting a different point of view, for sufficiently small $\delta_0$ $\|R_0\|_{\dot H^{1/2}}$ is comparable to the distance from $\Psi(0)$ to the soliton manifold,
\be
\min_{W} \|\Psi(0)-W\|_{\dot H^{1/2}}.
\ee
Indeed, one inequality is obvious and the other follows by Lemma \ref{lemma_10}, see below.

Applying the perturbation Lemma \ref{lemma_5} to the solution of the nonlinear equation leads to the following:
\begin{proposition}\lb{proposition_10} The solution depends continuously on the initial data: given solutions
\be
\Psi_j(t) = W_{\pi_j}(t) + R_j(t),
\ee
$j=1,2$, with initial data in the same fiber, $\Psi_j(0) \in \mc N(W_0)$, meaning $\Psi_j(0)=\mc F(R_{0j}, W_{0})$, one has
\be\begin{aligned}
\|\Psi_1-\Psi_2\|_{e^{t\rho} (L^{\infty}_t \dot H^{1/2}_x \cap L^2_t \dot W^{1/2, 6/5}_x)} &\leq C \|(R_1, \pi_1) - (R_2, \pi_2)\|_{Y} \\
&\leq C \|R_{01}-R_{02}\|_{\dot H^{1/2}}
\end{aligned}\ee
and
\be\begin{aligned}
&|h(R_{01}, W_0) - h(R_{02}, W_0)| \leq C \alpha_0 (\|R_{01}\|_{\dot H^{1/2}} + \|R_{02}\|_{\dot H^{1/2}}) \|R_{01}-R_{02}\|_{\dot H^{1/2}}.
\end{aligned}\ee
\end{proposition}
%
The continuous correspondence between different fibers is given by symmetry transformations.
\begin{proof}
Firstly, applying the comparison Lemma \ref{lemma_5} yields
\be
\|(R_1, \pi_1) - (R_2, \pi_2)\|_Y \leq C \|R_{01}-R_{02}\|_{\dot H^{1/2}}.
\ee
The comparison Lemma \ref{lemma_5} also implies
\be\begin{aligned}
|h(R_{01}, W_1) - h(R_{02}, W_2)| \leq C \delta \|R_{01}-R_{02}\|_{\dot H^{1/2}}.
\end{aligned}\ee
Taking $\delta$ proportional to $\|R_{01}\|_{\dot H^{1/2}} + \|R_{02}\|_{\dot H^{1/2}}$ leads to the second conclusion.
\end{proof}

In particular, this shows that the map $\mc F$ given by (\ref{2.105}) is locally Lipschitz continuous. We explore its properties further, beginning with a definition and a preliminary lemma.
\begin{definition}
Given two Banach spaces $A$ and $B$, a map $f:A \to B$ is analytic if it admits a Taylor series expansion:
\be
f(a)=f_0 + f_1(a) + f_2(a, a) + \ldots,
\ee
such that for each $n$ $f_n$ is $n$-linear and there exist constants $C_1$, $C_2$ such that $\|f_n\|_{A^{\otimes n} \to B} \leq C_1 C_2^n$.
\end{definition}
A similar definition can be given for differentiable, $C^n$ class, and smooth maps. Furthermore, the definition naturally extends to manifolds.

\begin{lemma} The map $\tilde F: \mc N_0 \times \set R \to \dot H^{1/2}$,
\be
\tilde F(W, R, h) = W + R + h F^+(W),
\ee
is locally a real analytic diffeomorphism in the neighborhood of each point $(W, R, 0)$.
\lb{lemma_11}\end{lemma}
%

\begin{proof} Let $W_0 = W(\pi_0)$ and consider the differential of $\tilde F$, given by the linear map $d\tilde F \mid_{(W_0, R_0, h_0)}: \set R^8 \times \mc N_0(W_0) \times \set R \to \dot H^{1/2}$,
\be\begin{aligned}
d\tilde F \mid_{(W_0, R_0, h_0)}(\delta \pi, \delta R, \delta h) &= (d_{\pi} W_0) \delta \pi + \delta R + (\delta h) F^+(W_0) + h_0 (d_{\pi} F^+(W_0) )\delta \pi.
\end{aligned}\ee
That $d\tilde F$ is bijective at points where $h_0=0$ follows from the following identity:
\be\begin{aligned}
\Psi &= P_0(W_0) \Psi + P_c(W_0) \Psi + P_-(W_0) \Psi + P_+(W_0) \Psi \\
&= \sum_{f \in \{\alpha, \Gamma, v_k, D_k\}} \langle \Psi, \Xi_f(W_0) \rangle \partial_f W_0 + (P_c(W_0) + P_-(W_0)) \Psi + \\
&+ \alpha_0^{-3} \langle \Psi, i \sigma_3 F^-(W_0) \rangle F^+(W_0).
\end{aligned}\ee
This leads to an explicitly constructed inverse for the linear map: if
\be
\Psi = d\tilde F \mid_{(W_0, R_0, h_0)}(\delta \pi, \delta R, \delta h),
\ee
then
\be\begin{aligned}
\delta \pi &= (\langle \Psi, \Xi_{\alpha}(W_0) \rangle, \langle \Psi, \Xi_{\Gamma}(W_0) \rangle, \langle \Psi, \Xi_{v_k}(W_0) \rangle, \langle \Psi, \Xi_{D_k}(W_0) \rangle), \\
\delta R &= (P_c(W_0) + P_-(W_0)) \Psi, \\
\delta h &= \alpha_0^{-3} \langle \Psi, i \sigma_3 F^-(W_0) \rangle.
\end{aligned}\ee
The local invertibility of the nonlinear map $\tilde F$ follows by the inverse function theorem. Smoothness follows by inspection of the explicit forms of $\partial_f W_0$ and $\partial_f F^+(W_0)$.

Next, we consider the analyticity of $\tilde F$, which is closely tied to that of the soliton $W(\pi)$, of its derivatives, and of the eigenvectors $F^+(W(\pi))$, considered as functions of the parameters $\pi$.

The analyticity of the soliton was shown by Li--Bona \cite{bona}, but we need a stronger statement than the one proved in their paper.
\begin{lemma}\lb{lemma_2.10}
Given an exponentially decaying solution $\phi$ of equation (\ref{phi})
\be\nonumber
-\Delta \phi(\cdot, \alpha) + \alpha^2 \phi(\cdot, \alpha) = \phi^3(\cdot, \alpha),
\ee
the soliton
\be\lb{2.170}
w(\pi)(x) = e^{i(x \cdot v+\Gamma)} \phi(x - D, \alpha)
\ee
is a real analytic function of $\pi$ in any Schwartz class seminorm $S$. Furthermore,
\be\lb{analitic}
\|\partial^{\beta_{\alpha}}_{\alpha} \partial^{\beta_{\Gamma}}_{\Gamma} \partial^{\beta_{v_k}}_{v_k} \partial^{\beta_{D_k}}_{D_k} w\|_S \leq C_1 C_2^{|\beta|} \prod_{f \in \{\alpha, v_k\}} \beta_f!.
\ee
\end{lemma}

We distinguish between two kinds of parameters in the definition of $\pi$ (\ref{2.170}). In the parameters $\alpha$ and $v_k$, the proof makes it clear that the domain of analyticity is not the whole complex plane, but only a strip along the real axis. In the parameters $\Gamma$ and $D_k$, on the other hand, the soliton extends to an analytic function of exponential type in the whole complex plane.

\begin{proof}
To begin with, we show that $w$ is an analytic $e^{-(1-\epsilon)|x|} L^{\infty}_x$-valued map. Since the derivatives of an analytic map are also analytic and $\partial_{D_k} w(\pi) = -\partial_{x_k} w(\pi)$, by iterating it follows that arbitrarily many derivatives of $w(\pi)$ are analytic maps into this space.
 
The real analyticity of $w$ is equivalent to its complex analyticity and the joint complex analyticity is equivalent to separate analyticity in each variable. Given that $\phi(\cdot, \alpha)$ decays exponentially at the rate $e^{-(\alpha^2-\epsilon) |x|}$, by the Agmon bound, we can extend $w$ to an analytic function for all $\Gamma$ and on the strip $\{v \mid |\Im v| < \alpha^2\}$. This proves analyticity in regard to $\Gamma$ and $v$.

Analyticity in the other two variables, $D$ and $\alpha$, requires that we show that $\phi(\cdot, 1)$ is an analytic function in the spatial variables, in the sense that there exist constants $C_1$ and $C_2$ such that for every multiindex $\beta = (\beta_{v_1}, \beta_{v_2}, \beta_{v_3})$
\be
\|\partial^{\beta} \phi\|_2 \leq C_1 |\beta|! C_2^{|\beta|}.
\ee
$\phi$ is exponentially decaying, $|\phi(x)| \leq C e^{-(1-\epsilon)|x|}$, and satisfies the equation
\be
-\Delta \phi + \phi = g(\phi),
\ee
where $g(z) = z^3$ is an analytic function. Differentiating $\beta$ times, we obtain
\be\begin{aligned}
(-\Delta+1)(\partial^{\beta} \phi) &= \partial^{\beta} (g \circ \phi) \\
&= \sum_{j=1}^{|\beta|} \sum_{\beta_1 + \ldots + \beta_j = \beta} \partial^{\beta_1} \phi \cdot \ldots \cdot \partial^{\beta_j} \phi \cdot (g^{(j)} \circ \phi).
\end{aligned}\ee
In the proof we make the more general assumption that $g$ is analytic and its derivatives grow subexponentially, meaning that there for any $C_4>0$ there exists $C_3$ such that 
\be
|g^{(j)}(0)| \leq C_3 C_4^j.
\ee
It follows that for any $C_4>0$ there exists $C_3$ such that, for fixed $f \in H^2$,
\be
\|g^{(j)} \circ f\|_{H^2} \leq C_3 C_4^j.
\ee
This property characterizes not only polynomials such as the case of interest, $g(z) = z^3$, but also functions of subexponential growth, such as $g(z)=\cosh(\sqrt z)$.

To avoid complications, we present the argument in the algebra $H^2$. We prove by induction on $|\beta|$ that, for properly chosen $C_1$ and $C_2$,
\be
\|\partial^{\beta} \phi\|_{H^2} \leq \frac{C_1 C_2^{|\beta|}}{|\beta|^2+1}.
\ee
Indeed, assuming that this induction hypothesis holds for all indices up to $|\beta|$, we obtain that
\be\begin{aligned}
\|\partial^{\beta} \phi\|_{H^4} & \leq C \|(-\Delta+1)(\partial^{\beta} \phi)\|_{H^2} \\
& \leq C C_3 C_2^{|\beta|} \sum_{j=1}^{|\beta|} \sum_{\beta_1 + \ldots + \beta_j = \beta} \frac{C_4^j C_1^j}{(|\beta_1|^2 + 1) \cdot \ldots \cdot (|\beta_j|^2 + 1)}.
\end{aligned}\ee
Observe that
\be
\sum_{k=0}^{n} \frac 1 {(k^2 +1)((n-k)^2 + 1)} < \frac 6 {n^2 + 1}.
\ee
By induction, we obtain that
\be
\sum_{n_1 + \ldots + n_j = n} \frac 1 {(n_1^2 + 1) \cdot \ldots \cdot (n_j^2 + 1)} \leq \frac {6^{j-1}} {n^2 + 1}. 
\ee
Therefore
\be
\|\partial^{\beta} \phi\|_{H^4} \leq \frac {C C_3 C_2^{|\beta|}}{|\beta|^2 + 1} \sum_{j=1}^{|\beta|} (6 C_4 C_1)^j.
\ee
By making $C_4$ sufficiently small and $C_2$ sufficiently large, we obtain that the sum is uniformly bounded, regardless of $|\beta|$, and
\be
\|\partial^{\beta} \phi\|_{H^4} \leq \tilde C \frac {C_2^{|\beta|}}{|\beta|^2 + 1} \leq \frac {C_1 C_2^{|\beta|+1}} {|\beta|^2+1}.
\ee
This proves that the induction assumption also holds for derivatives of order $|\beta|+1$.

The preceding proof shows that $\phi(\cdot - D)$ is an analytic $H^2$-valued map. However, the argument works for any algebra $A$ containing $\phi$ and with the property that
\be
\|d f\|_A \leq C \|(-\Delta + 1) f\|_A.
\ee
This means that the algebra $A$ has to be invariant under convolution with the kernels
\be
\frac {x_k e^{-|x|}}{|x|^3} + \frac {x_k e^{-|x|}}{|x|^2}.
\ee
Note that the algebra $A = e^{-(1-\epsilon)|x|}L^{\infty}$ fulfills both requirements. Thus, $\phi(\cdot - D)$ is an $A$-valued analytic map. This ensures the joint analyticity of $w(\pi)$ given by (\ref{2.170}) in the variables $D$, $v$, and $\Gamma$ around the point $\phi = w(1, 0, 0, 0)$. By symmetry transformations, this implies analyticity around any other point.

Concerning $\alpha$, observe that the derivative with respect to $\alpha$ is given by the generator of dilations, which is a combination of multiplication by $x$ and taking the gradient $\dl$:
\be
\partial_{\alpha} \phi = \alpha^{-1} \phi + x \dl \phi.
\ee
Therefore, analyticity in $\alpha$ follows from that with respect to to $D$ and $v$.

If $W$ is analytic, so are its derivatives up to any finite order, enabling us to conclude that $W$ is analytic in any Schwartz class seminorm.
%
\end{proof}
In a similar manner, the analyticity of $F^+(W(\pi))$ as a function of $\pi$ reduces to that of a fixed eigenfunction $F^+$. $F^+ = \bpm f^+ \\ \ov f^+ \epm$ (see Section \ref{spectru}) satisfies the equation
\be
(\Delta - 1 + 2 \phi^2) f^+ + \phi^2 \ov f^+ = i \sigma f^+.
\ee
Knowing that $\phi$ is analytic, proving the analyticity of $f^+$ proceeds in exactly the same manner as above.

Clearly, if $W(\pi)$ is analytic then so are its derivatives up to any finite order. 
The same goes for $F^+(W(\pi))$ and its derivatives. This immediately implies that $\tilde F$ is real analytic.
\end{proof}
%

Lemma \ref{lemma_11} has the following immediate consequence:
\begin{proposition} $\mc F$ given by (\ref{2.105}) is locally one-to-one and its inverse (defined on its range) is locally Lipschitz.
\lb{prop_9}
\end{proposition}
\begin{proof} The local invertibility of $\mc F$ follows immediately from the previous lemma. Indeed, one has
\be
\mc F(R, W) = \tilde F(W, R, h(R, W)).
\ee
For a sufficiently small $\delta_0$, $h(R, W)$ is close to zero and the previous lemma applies. In order to establish the Lipschitz property for the inverse, we can simply ignore the parameter $h$.
\end{proof}

Another consequence is that, if a function is sufficiently close to the manifold $\mc M$ of solitons, we can project it on the manifold as follows.

\begin{lemma} For every soliton $W$ there exists $\delta > 0$ such that whenever $\|\Psi-W\|_{\dot H^{1/2}} < \delta$ there exists $W_1$ such that $P_0(W_1)(\Psi-W_1) = 0$ and
\be
\|\Psi - W_1\|_{\dot H^{1/2}} \leq C \|\Psi-W\|_{\dot H^{1/2}}.
\ee
Furthermore, $W_1$ depends Lipschitz continuously on $\Psi$.
\lb{lemma_10}
\end{lemma}
Again, following the use of symmetry transformations, $\delta$ can be chosen without regard for $W$.
\begin{proof}
If $\delta$ is sufficiently small, $\Psi = \tilde F(W_1, R, h)$ for some $(W_1, R, h)$ close to $(W, 0, 0)$ by Lemma \ref{lemma_2.10}. Since the inverse of $\tilde F$ is bounded,
\be
\|W-W_1\|_{\dot H^{1/2}} + \|R\|_{\dot H^{1/2}} + |h| \leq C \|\Psi-W\|_{\dot H^{1/2}} \leq C \delta.
\ee
Since $R \in \mc N_0(W_1)$, by definition $P_0(W_1)(\Psi-W_1) = 0$.

To a first order $W-W_1$ lies in a direction tangent to the soliton manifold $\mc M$, meaning within the range of $P_0(W_1)$, so by means of a Taylor expansion we obtain
\be
\|(I-P_0(W_1))(W-W_1)\|_{\dot H^{1/2}} \leq C \|W-W_1\|_{\dot H^{1/2}}^2.
\ee
Thus
\be\begin{aligned}
\|\Psi - W_1\|_{\dot H^{1/2}} & = \|(I-P_0(W_1))(\Psi-W_1)\|_{\dot H^{1/2}} \\
& \leq C \|\Psi-W\|_{\dot H^{1/2}} + \|(I-P_0(W_1))(W-W_1)\|_{\dot H^{1/2}} \\
& \leq C (\|\Psi-W\|_{\dot H^{1/2}} + \|W-W_1\|_{\dot H^{1/2}}^2).
\end{aligned}\ee
On the other hand,
\be\begin{aligned}
\|W-W_1\|_{\dot H^{1/2}}^2 & \leq (\|\Psi-W\|_{\dot H^{1/2}} + \|\Psi - W_1\|_{\dot H^{1/2}})^2 \\
& \leq C \|\Psi-W\|_{\dot H^{1/2}} + C \delta \|\Psi - W_1\|_{\dot H^{1/2}}.
\end{aligned}\ee
For $\delta$ sufficiently small, the conclusion follows.
\end{proof}
This lemma would be superfluous if $P_0$ were an orthogonal projection and the constant could be taken to be one then. Also note that in this generality the conclusion still holds for $\dot W^{1/2, 6}$, with the same proof.

\begin{definition} By small asymptotically stable solution we mean one that can be written as $\Psi(t) = W_{\pi}(t) + R(t)$ where $W_{\pi}(t)$ is a moving soliton governed by the parameter path $\pi$ as in (\ref{8}) and
\be
\|(R, \pi)\|_X = \|R\|_{L^{\infty}_t \dot H^{1/2}_x \cap L^2_t \dot W^{1/2, 6}_x} + \|\dot \pi\|_1 < \delta_0.
\ee
\lb{definition_4}\end{definition}
$X$ is the space that appears in the contraction argument, see (\ref{X}).

We can rewrite any small asymptotically stable solution $\Psi$ as $W_{\tilde \pi}(t) + \tilde R(t)$ such that the orthogonality condition is satisfied:
\be
P_0(W(\tilde \pi(t))) \tilde R(t) = 0.
\ee
Following Lemma \ref{lemma_10}, $\tilde R(t)$ is still small in the space $L^{\infty}_t \dot H^{1/2}_x \cap L^2_t \dot W^{1/2, 6}_x$.

Furthermore, $W_{\tilde \pi}(t)$ depends Lipschitz continuously on $\Psi$. Writing the modulation equations explicitly as in (\ref{mod}), it follows that $\|\dot {\tilde \pi}\|_{L^1_t}$ is small too.

Thus, it makes no difference whether we assume the orthogonality condition initially, as part of Definition \ref{definition_4}, since we can produce it in this manner. We arrive at the same definition with or without the orthogonality condition, though it may be for different values of $\delta_0$.

Clearly, every solution with initial data on the manifold $\mc N$ is small and asymptotically stable. A partial converse is also true.
\begin{proposition}
If $\Psi(0)$ is the initial value of a small asymptotically stable solution $\Psi$ in the sense of Definition \ref{definition_4} to equation (\ref{NLS}), then $\Psi(0) \in \mc N$.
\lb{comparatie}
\end{proposition}
\begin{proof} Write $\Psi = W_{\pi}(t) + R(t)$, satisfying the orthogonality condition
\be
P_0(W_{\pi}(t)) R(t) = 0.
\ee
By Lemma \ref{lemma_11}, there exist $R_0 \in \mc N_0(W_{\pi}(0))$ and $h$ such that
\be
R(0) = R_0 + h_0 F^+(W_{\pi}(0)).
\ee

Then, we note that both the initial data
\be
\Psi(0) = W_{\pi}(0) + R_0 + h_0 F^+(W_{\pi}(0))
\ee
and the initial data
\be
\tilde \Psi(0) = W_{\pi}(0) + R_0 + h(R_0, W_{\pi}(0)) F^+(W_{\pi}(0))
\ee
give rise to small asymptotically stable solutions, in the form
\be
\Psi(t) = W_{\pi}(t) + R(t),\ \tilde \Psi(t) = W_{\tilde \pi}(t) + \tilde R(t).
\ee
The perturbation Lemma \ref{lemma_5} then applies, implying that
\be
\|(R, \pi) - (\tilde R, \tilde \pi)\|_Y \leq C \delta \|(R, \pi) - (\tilde R, \tilde \pi)\|_Y.
\ee
Otherwise put, we obtain that $(R, \pi) = (\tilde R, \tilde \pi)$. Applying the lemma once more, it follows that the $h$ values coincide as well, that is $h_0 = h(R_0, W_{\pi}(0))$. Therefore
\be
\Psi(0) = \tilde \Psi(0) = W_{\pi}(0) + R_0 + h(R_0, W_{\pi}(0)) F^+(W_{\pi}(0))
\ee
and thus $\Psi(0)$ belongs to $\mc N$.
\end{proof}

\begin{corollary} If $\Psi$ is a solution to (\ref{NLS}) whose initial data $\Psi(0)$ belongs to $\mc N$, then $\Psi(t)$ also belongs to $\mc N$ for all positive $t$ and for sufficiently small negative $t$.
\lb{stable}
\end{corollary}
\begin{proof}
Clearly, both for positive $t$ and for sufficiently small negative $t$ $\Psi(t)$ exists (due to the local existence theory, for negative $t$) and gives rise to a small asymptotically stable solution. Then, the previous proposition shows that $\Psi(t)$ must still be on the manifold.
\end{proof}

To recapitulate, we have investigated the properties of five maps that describe the solution of the nonlinear problem (\ref{NLS}), namely
\begin{align}
h(R_0, W_0)&: \mc N_0 \to \set R, \\
\mc F(R_0, W_0)&: \mc N_0 \to \mc N,\ \mc F(R_0, W_0) = \tilde F(W_0, R_0, h(R_0, W_0)),
\intertext{and the solution itself,}
\Psi=\Psi(R_0, W_0)&: \mc N_0 \to L^{\infty}_t \dot H^{1/2}_x \cap L^2_t \dot W^{1/2, 6/5}_x,\ \Psi(t) = W_{\pi}(t) + R(t), \\
\pi = \pi(R_0, W_0)&: \mc N_0 \to \dot W^{1, 1}_t, \\
R = R(R_0, W_0)&: \mc N_0 \to L^{\infty}_t \dot H^{1/2}_x \cap L^2_t \dot W^{1/2, 6/5}_x.
\end{align}
It turns out that all five are real analytic and one can also replace $\mc N_0$ by $\mc N$ --- change the variable to $\Psi(0)$, that is.
\begin{proposition} For fixed $W_0$ there exists $\rho>0$ such that the maps $h$, $\mc F$ considered as a map into $e^{t\rho}L^{\infty}_t \dot H^{1/2}_x$, $(R, \pi)$ regarded as a map into $Y$ (see (\ref{Y})), and $\Psi$, seen as map into $e^{t\rho} (L^{\infty}_t \dot H^{1/2}_x \cap L^2_t \dot W^{1/2, 6}_x)$, are real analytic in the variable $R_0$ on the fibre $\mc N_0(W_0)$ or, equally, in the variable $\Psi(0)$ on the fibre $\mc N(W_0)$.
\end{proposition}
Letting $W_0$ vary, note that the dependence of these maps on $W_0$ (across the different fibers of $\mc N_0$) is given by symmetry transformations, which are also analytic in a properly considered setting.

Also note that the analytic dependence of the solution on initial data and the analyticity of the manifold are closely tied to that of the nonlinearity in (\ref{NLS}). If the nonlinearity were of class $C^k$, we would expect that the manifold and the dependence on initial data would also be of class $C^k$.
\begin{proof} Firstly note that, since $\tilde F$ is analytic, it suffices to show that $h$ and $(R, \pi)$ are analytic.

The proof consists in the following steps: firstly we exhibit the first-order differential of $h$ and $(R, \pi)$ and show that it satisfies the definition of differentiability, meaning that the remainder in the Taylor expansion is quadratic in size. By recursion, we then introduce the $n$-th order differentials for arbitrary $n$.

Finally, we show that differentials and remainders in the Taylor expansion grow in norm at most exponentially, enabling us to conclude that $h$ and $(R, \pi)$ equal the sum of their own Taylor series within a sufficiently small convergence radius.

To begin with, for given initial data $W_0=W(\pi_0)$ and $R_0$, denote the corresponding solution to (\ref{NLS}) by $\Psi = W_{\pi} + R$. $R$ and $\pi$ satisfy the equation system, in the form (\ref{ec_liniara}),
\be\begin{aligned}\lb{ec_r0}
&i \partial_t R + \mc H_{\pi}(t) R = F(\pi, R), \\
&\dot f = F_f(\pi, R),\ f \in \{\alpha, \Gamma, v_k, D_k\},
\end{aligned}\ee
where
\be\begin{aligned}
\mc H_{\pi}(t) &= \Delta \sigma_3 + V_{\pi}(t),\\
F(\pi, R) &= -i L_{\pi} R + N(R, W_{\pi}) - N_{\pi}(R, W_{\pi}),
\end{aligned}\ee
and we introduced the notation
\be
F_f(\pi, R) = \left\{\begin{aligned}
4\alpha \|W_{\pi}\|_2^{-2} &(\langle R, (d_{\pi} \Xi_{f}(W_{\pi})) \dot \pi \rangle \\
&- i\langle N(R, W_{\pi}), \Xi_{f}(W_{\pi})\rangle),\ f \in \{\alpha, \Gamma\}\\
2\|W_{\pi}\|_2^{-2} &(\langle R, (d_{\pi} \Xi_{f}(W_{\pi})) \dot \pi \rangle \\
&- i\langle N(R, W_{\pi}), \Xi_{f}(W_{\pi})\rangle),\ f \in \{v_k, D_k\}.
\end{aligned}\right.
\ee
The initial data are given by
\be
R(0) = R_0 + h(R_0, W_0) F^+(W_0),\ \pi(0) = \pi_0.
\ee
This setup makes the orthogonality condition
\be\lb{2.207}
\langle R(t), \Xi_f(W_{\pi}(t)) \rangle = 0
\ee
valid for all times. Indeed, it holds initially at time $t=0$ and equation (\ref{ec_r0}) ensures that it still holds at any other time.
We aim to find an infinite expansion of the form
\be
R = R^0 + R^1 + R^2 + \ldots,\ \pi = \pi^0 + \pi^1 + \pi^2 + \ldots,\ h(R_0, W_0) = h^0 + h^1 + h^2 + \ldots,
\ee
where $R^n = n! d^nR$, $\pi^n = n! d^n\pi$, and $h^n = n! d^nh$ are the $n$-th order terms in the power expansions of $R$, $\pi$, and $h$, given by $n$-linear expressions in $R_0$, such that
\be\begin{aligned}
\|R - R^0 - R^1 - \ldots - R^n\|_{e^{t\rho}(L^{\infty}_t \dot H^{1/2}_x \cap L^2_t \dot W^{1/2, 6}_x)} &\leq C_n \|R_0\|_{\dot H^{1/2}}^{n+1}, \\
\|\dot \pi - \dot \pi^0 - \dot \pi^1 - \ldots - \dot \pi^n\|_{e^{t\rho} L^1_t} &\leq C_n \|R_0\|_{\dot H^{1/2}}^{n+1}, \\
|h - h^0 - h^1 - \ldots - h^n| &\leq C_n \|R_0\|_{\dot H^{1/2}}^{n+1}.
\end{aligned}\ee
Showing, in addition, that $C_n \leq C_1 C_2^n$ for some constants $C_1$, $C_2$ and every $n$ guarantees the analyticity of $R$, $\pi$, and $h$.

For technical reasons, it is easier to construct the differentials and prove analyticity at zero. Analyticity will then follow at every point within the radius of convergence.

The constant terms in the expansion are
\be
R^0(t) = 0,\ \pi^0(t) = \pi_0 = (1, 0, 0,0),\ h^0=0.
\ee

The natural guess is that first-order differentials $R^1$ and $\pi^1$ satisfy the following linearized version of (\ref{ec_r0}):
\be\begin{aligned}
&i \partial_t R^1 + \mc H_{\pi^0}(t) R^1 = (\partial_{\pi} F(\pi^0, R^0)) \pi^1 + (\partial_R F(\pi^0, R^0)) R^1 - (\partial_{\pi} V_{\pi^0})(t) \pi^1 R^0 \\
&\dot {f^1} = (\partial_{\pi} F_f(\pi^0, R^0)) \pi^1 +  (\partial_R F_f(\pi^0, R^0)) R^1,\ f \in \{\alpha, \Gamma, v_k, D_k\},
\end{aligned}\ee
with initial data
\be
R^1(0) = R_0 + h^1 F^+(W_0),\ \pi^1(0) = 0.
\ee
It is important to note that, because $R^0=0$, $\dot \pi^0 = 0$, the following terms cancel:
\be\begin{aligned}
&(\partial_{\pi} V_{\pi^0})(t) \pi^1 R^0 = 0,\ \partial_{\pi} F_f(\pi^0, R^0) = 0,\ \partial_R F_f(\pi^0, R^0) = 0, \\
&\partial_{\pi} F(\pi^0, R^0) = 0,\ \partial_R F(\pi^0, R^0) = 0.
\end{aligned}\ee
This improvement only holds when taking the differential at zero and is due to the fact that the nonlinearity is of order higher than two.

It follows that $\pi^1=0$ and
\be\begin{aligned}\lb{2.213}
&i \partial_t R^1 + \mc H_{\pi^0}(t) R^1 = 0.
\end{aligned}\ee
The orthogonality condition
\be
\langle R, \Xi_f(W_{\pi}) \rangle = 0
\ee
becomes
\be
\langle R^1, \Xi_f(W_{\pi^0}) \rangle + \langle R^0, (\partial_{\pi} \Xi_f(W_{\pi^0})) \pi^1 \rangle = 0;
\ee
consequently it still holds for $R^1$:
\be\lb{2.216}
\langle R^1, \Xi_f(W_{\pi^0}) \rangle = 0.
\ee
By means of Strichartz estimates we obtain, in the same manner as in the proof of stability, that $R^1$ is bounded for a unique value of $h^1$ and then it satisfies
\be
\|R^1\|_{L^{\infty}_t \dot H^{1/2}_x \cap L^2_t \dot W^{1/2, 6}_x} \leq C \|R_0\|_{\dot H^{1/2}}.
\ee
Note that $\|R^1\|_{L^{\infty}_t \dot H^{1/2}_x \cap L^2_t \dot W^{1/2, 6}_x}$ can be made arbitrarily small, but $\|\dot \pi^1\|_{L^1_t}=0$.

To show that $R^1$ thus constructed and $\pi^1=0$ are indeed the first-order differentials of $R$ and $\pi$, consider
\be
S^1 = R^0 + R^1, \Sigma^1 = \pi^0 + \pi^1.
\ee
If $F(\pi, R)$, $V_{\pi}$, and $F_f(\pi, R)$ were differentiable, it should be the case that
\be\begin{aligned}\lb{2.221}
F(\Sigma^1, S^1) &= F(R^0, \pi^0) + (\partial_{\pi} F(\pi^0, R^0)) \pi^1 + (\partial_R F(\pi^0, R^0)) R^1 + o_1(R^1, \pi^1), \\
V_{\Sigma^1} &= V_{\pi^0} + \partial_{\pi} V_{\pi^0} \pi^1 + o_2(\pi^1), \\
F_f(\Sigma^1, S^1) &= F_f(\pi^0, R^0) + (\partial_{\pi} F_f(\pi^0, R^0)) \pi^1 + (\partial_R F_f(\pi^0, R^0)) R^1 + o_3(R^1, \pi^1).
\end{aligned}\ee
In fact, even more is true: all three quantities are analytic, as one sees by examining their explicit forms.

The main issue is that the soliton $W_{\pi}(t)$ depends not only on the values of $\pi$ and $\dot \pi$ at time $t$, but also on the integral $\int_0^{t} \pi(s) \dd s$. Thus, even though $W_{\pi}$ depends analytically on $\pi$, each derivative produces a factor of $t$.

In this setting, the error terms in (\ref{2.221}) are quadratic:
\be\begin{aligned}
 \|o_1(R^1, \pi^1)\|_{\langle t \rangle L^2_t W^{1/2, 6/5}_x} &\leq C \|(R^1, \pi^1)\|_X^2 \\
 \|o_2(\pi^1)\|_{\langle t \rangle L^{\infty}_t (\dot W^{1/2, 6/5-\epsilon} \cap \dot W^{1/2, 6/5+\epsilon})} &\leq C \|(R^1, \pi^1)\|_X^2 \\
\|o_3(R^1, \pi^1)\|_{\langle t \rangle L^1_t} &\leq C \|(R^1, \pi^1)\|_X^2.
\end{aligned}\ee
Therefore
\be\begin{aligned}\lb{2.225}
&i \partial_t S^1 + \mc H_{\Sigma^1}(t) S^1 = F(\Sigma^1, S^1) + O^2(R^1, \pi^1) \\
&\dot f_{\Sigma^1} = F_f(\Sigma^1, S^1) + O_f^2(R^1, \pi^1),\ f \in \{\alpha, \Gamma, v_k, D_k\},
\end{aligned}\ee
where $O^2$, $O_f^2$ are error terms bounded by
\be
\|O^2(R^1, \pi^1)\|_{\langle t \rangle L^2_t W^{1/2, 6/5}_x} + \|O_f^2(R^1, \pi^1)\|_{\langle t \rangle L^1_t} \leq C \|(R^1, \pi^1)\|_X^2.
\ee
Following (\ref{2.216}), $S^1$ fulfills the orthogonality relation
\be
\langle S^1(t), \Xi_f(W_{\Sigma^1}(t)) \rangle = 0.
\ee
Comparing the equation system (\ref{2.225}) satisfied by $S^1$ and $\Sigma^1$ with the one satisfied by $R$ and $\pi$, (\ref{ec_r0}), we obtain that
\be
\|(R, \pi) - (S^1, \Sigma^1)\|_Y \leq C \|(R^1, \pi^1)\|_X^2 \leq C  \|R_0\|_{\dot H^{1/2}}^2.
\ee
The proof exactly follows that of Lemma \ref{lemma_5}.

We repeat this procedure for the higher-order terms in the expansion. 
Let $n \geq 1$ and consider a variation $\pi = \pi^0 + \delta \pi$ of $\pi^0$ such that $\delta \pi(0)=0$. Recall that solitons depend analytically on parameters: for any Schwartz seminorm $S$,
\be
\|d^n_{\pi} W(\pi^0)(\delta \pi))\|_S \leq C_1 C_2^n n! \|\delta \dot \pi\|_1^n.
\ee
However, by (\ref{analitic}) only derivatives in $\alpha$ and $v_k$ produce a factorial contribution:
$$
\|\partial^{\beta_{\alpha}}_{\alpha} \partial^{\beta_{\Gamma}}_{\Gamma} \partial^{\beta_{v_k}}_{v_k} \partial^{\beta_{D_k}}_{D_k} w\|_S \leq C_1 C_2^{|\beta|} \prod_{f \in \{\alpha, v_k\}} \beta_f!.
$$
Thus we obtain for $W_{\pi^0}$ given by (\ref{8}) and (\ref{w0})
$$\begin{aligned}
W_{\pi^0}(t) &= \bpm w_{\pi^0}(t) \\ \ov w_{\pi^0}(t) \epm,\\
w_{\pi^0}(t) &= w\bigg(\alpha^0(t), \Gamma^0(t) + \int_0^t ((\alpha^0(s))^2-|v^0(s)|^2) \dd s, v^0(t), D(t) + 2\int_0^t v^0(s) \dd s \bigg)
\end{aligned}$$
--- and likewise for $\partial_f W_{\pi^0}$, $F^+(W_{\pi^0})$, and all other quantities that depend on the moving soliton --- the following explicit expression of analyticity:
\be\lb{2.229}
\|\partial^{\beta_{\alpha}}_{\alpha} \partial^{\beta_{\Gamma}}_{\Gamma} \partial^{\beta_{v_k}}_{v_k} \partial ^{D_k}_{D_k} W_{\pi^0}(t)(\delta \pi))\|_S \leq C_1 C_2^{|\beta|} \prod_{f \in \{\alpha, v_k\}} \beta_f! \prod_{f \in \{\Gamma, D_k\}} \langle t \rangle^{\beta_f} \|\delta \dot \pi\|_1^{|\beta|}.
\ee
As noted before, this is an improvement over mere analyticity (and one that is necessary in the sequel), in that, while in half the variables $W$ is only analytic on a strip, in the other half $W$ extends to an analytic function of exponential type on the whole complex plane.

At our discretion, we pick parameters $a$ and $a_1$ such that $a< a_1 \leq\sigma(\alpha(t))$, for any time $t$, where $\alpha(t)$ is the scaling component of $\pi(t)$ and $\pi$ is any path that appears in this proof (their scaling parameters are uniformly bounded away from zero). Define the weights
\be\lb{2.230}
A_n(t) = \sum_{j=0}^n \frac{\langle at \rangle^j}{j!} < e^{\langle at \rangle}.
\ee
Then (\ref{2.229}) implies
\be
W_{\pi^0} = W_{\pi^0}^0 + W_{\pi^0}^1 + \ldots
\ee
such that in any Schwartz seminorm $S_x$, in the space variables only,
\be
\|W^n_{\pi^0}(\delta \pi)\|_{A_n(t) L^{\infty}_t S_x} \leq C_1 C_2^n \|\delta \dot \pi\|_1^n.
\ee
This immediately yields a power series expansion for $V_{\pi^0}$:
\be
V_{\pi^0} = V_{\pi^0}^0 + V_{\pi^0}^1 + \ldots,
\ee
with $V^n_{\pi^0}$ $n$-linear for each $n$ and
\be
\|V^n_{\pi^0}(\delta \pi)\|_{A_n(t) L^{\infty}_t (\dot W^{1/2, 6/5-\epsilon} \cap \dot W^{1/2, 6/5+\epsilon})} \leq C_1 C_2^n \|\delta \dot \pi\|_1^n.
\ee
The weights $A_n(t)$ given by (\ref{2.230}) have the property that
\be
\sum_{j=0}^n A_j(t) A_{n-j}(t) \leq C^n A_n(t).
\ee
Consequently, consider the spaces $\partial_t^{-1} A_n(t) L^1_t$. In this setting, for
\be
m_1 + \ldots + m_n = m,
\ee
and given $n$ path variations $\delta \pi_1$ to $\delta \pi_n$ such that $\delta \pi_j = 0$, $j=\ov{1, n}$, one has that
\be\begin{aligned}
\|V^n(\delta \pi_1, \ldots, \delta \pi_n)\|_{A_{m+n}(t) L^{\infty}_t (\dot W^{1/2, 6/5-\epsilon} \cap \dot W^{1/2, 6/5+\epsilon})} \leq \\
\leq C_1 C_2^n \|\delta \dot \pi_1\|_{A_{m_1}(t) L^1_t} \cdot \ldots \cdot \|\delta \dot \pi_n\|_{A_{m_n}(t) L^1_t}.
\end{aligned}\ee
Considering the explicit form of $F(R, \pi)$ and $F_f(R, \pi)$, note again that
\be
F(R^0, \pi^0) = dF(R^0, \pi^0) = 0,\ F_f(R^0, \pi^0) = dF_f(R^0, \pi^0) = 0.
\ee
This is due to linearizing around zero ($R^0 = 0$, $\dot \pi^0 = 0$), starting from a nonlinearity of order higher than two.

In fact, one needs to take two derivatives in $R$ or $\dot \pi$ before arriving at a nonzero derivative of $F$ or $F_f$. 
We obtain the power series expansion
\be
F = F^0 + F^1 + \ldots,\ F_f = F_f^0 + F_f^1 + \ldots,
\ee
where $F^n$ and $F_f^n$ are $n$-linear and, for $X$ defined by (\ref{X}),
\be\begin{aligned}
\|F^n\|_{A_{n-2}(t) L^2_t \dot W^{1/2, 6/5}_x} &\leq C_1 C_2^n \|(\delta R_1, \delta \pi_1)\|_X \cdot \ldots \cdot \|(\delta R_n, \delta \pi_n)\|_X, \\
\|F_f^n\|_{A_{n-2}(t) L^1_t} &\leq C_1 C_2^n \|(\delta R_1, \delta \pi_1)\|_X \cdot \ldots \cdot \|(\delta R_n, \delta \pi_n)\|_X.
\end{aligned}\ee
The gain from $n$ to $n-2$ is due to the fact that the order of the derivative in $\pi$ is two less than the total order.

In the scale of weighted spaces
\be
X^n=\{(R, \pi) \mid R \in A_n(t) (L^{\infty}_t \dot H^{1/2}_x \cap L^2_t W^{1/2, 6}_x), \dot \pi \in A_n(t) L^1_t\}
\ee
and with the notation
\be
m_1 + \ldots + m_n = m,
\ee
one also has that
\be\begin{aligned}
\|F^n\|_{A_{m+n-2}(t) L^2_t \dot W^{1/2, 6/5}_x} &\leq C_1 C_2^n \|(\delta R_1, \delta \pi_1)\|_{X^{m_1}} \cdot \ldots \cdot \|(\delta R_n, \delta \pi_n)\|_{X^{m_n}} \\
\|F^n\|_{A_{m+n-2}(t) L^1_t} &\leq C_1 C_2^n \|(\delta R_1, \delta \pi_1)\|_{X^{m_1}} \cdot \ldots \cdot \|(\delta R_n, \delta \pi_n)\|_{X^{m_n}}.
\end{aligned}\ee


Then, in the Taylor expansion of $F(\pi, R)$, where $\pi = \pi^0 + \pi^1 + \ldots $ and $R = R^0 + R^1 + \ldots$ are functions of $R_0$, the $n$-th order term is
\be
\sum_{j=1}^n \sum_{n_1 + \ldots + n_j=n} F^j((R^{n_1}, \pi^{n_1}), \ldots, (R^{n_j}, \pi^{n_j})).
\ee
By writing only the $n$-th order terms in the Taylor expansion of (\ref{ec_r0}), we see that $R^n$ and $\pi^n$ fulfill the equation
\be\begin{aligned}\lb{2.231}
i \partial_t R^n + \mc H_{\pi^0}(t) R^n &= \sum_{j=2}^n \sum_{n_1 + \ldots + n_j=n} F^j((R^{n_1}, \pi^{n_1}), \ldots, (R^{n_j}, \pi^{n_j})) \\
&- \sum_{j=1}^n \sum_{n_1 + \ldots + n_j + \tilde n=n} (\partial_{\pi}^j V_{\pi^0})(t) (\pi^{n_1}, \ldots, \pi^{n_j}) R^{\tilde n} \\
\dot f^n &= \sum_{j=2}^n \sum_{n_1 + \ldots + n_j=n} F_f^j((R^{n_1}, \pi^{n_1}), \ldots, (R^{n_j}, \pi^{n_j})).
\end{aligned}\ee
with initial conditions $R^n(0) = h^n F^+(W_0)$ and $\pi^n(0)=0$.

Since the expansion is around zero, note again the cancellations of all terms containing $R^0$ or $\pi^0$:
\be\begin{aligned}\lb{2.246}
F^1\mid_{(\pi^0, R^0)} = 0,\ F_f^1\mid_{(\pi^0, R^0)} = 0,\ (\partial_{\pi} V_{\pi^0})(t) \pi^n R^0 = 0.
\end{aligned}\ee
This means that highest-order terms, depending on $R^n$ and $\pi^n$, are completely absent from the right-hand side, which is consequently given by combinations of lower-order terms.

We assume that the \emph{induction hypothesis}
\be
\|(R^{m}, \pi^{m})\|_{X^{m-1}} \leq k_1 k_2^m \|R_0\|_{\dot H^{1/2}}^m
\ee
holds for every $m<n$ and some constants $k_1$ and $k_2$ to be established later. Then the right-hand sides of equations (\ref{2.231}) obey the bounds
\be\begin{aligned}\lb{2.248}
\bigg\|\sum_{j=2}^n \sum_{n_1 + \ldots + n_j=n} F^j((R^{n_1}, \pi^{n_1}), \ldots, (R^{n_j}, \pi^{n_j}))\bigg\|_{A_{n-1}(t) L^2_t \dot W^{1/2, 6/5}_x} \leq \\
\leq C_1 \sum_{j=2}^n (C_2 k_1) ^j k_2^n \|R_0\|_{\dot H^{1/2}}^n
\end{aligned}\ee
and likewise for the other two terms, that is
\be\begin{aligned}\lb{2.249}
\bigg\|\sum_{j=1}^n \sum_{n_1 + \ldots + n_j + \tilde n=n} (\partial_{\pi}^j V_{\pi^0})(t) (\pi^{n_1}, \ldots, \pi^{n_j}) R^{\tilde n} \bigg\|_{A_{n-1}(t) L^2_t \dot W^{1/2, 6/5}_x} \leq \\
\leq C_1 \sum_{j=2}^n (C_2 k_1) ^j k_2^n \|R_0\|_{\dot H^{1/2}}^n,
\end{aligned}\ee
respectively
\be\begin{aligned}\lb{2.250}
\bigg\| \sum_{j=2}^n \sum_{n_1 + \ldots + n_j=n} F_f^j((R^{n_1}, \pi^{n_1}), \ldots, (R^{n_j}, \pi^{n_j})) \bigg\|_{A_{n-1}(t) L^1_t} \leq \\
\leq C_1 \sum_{j=2}^n (C_2 k_1) ^j k_2^n \|R_0\|_{\dot H^{1/2}}^n.
\end{aligned}\ee
Importantly, summation starts from $j=2$ on the right-hand side, because the terms involving first-order differentials vanish, following (\ref{2.246}).

The powers of $t$ add up in correctly, since
\be
j-2 + \sum_{j=2}^n (n_j-1) = n-2 < n-1.
\ee
By making $k_1$ sufficiently small, we can get the sum $\sum_{j=2}^n (C_2 k_1) ^j$ to be bounded and small without regard to $n$ and thus we can put
\be
\tilde C k_1^2 k_2^n \|R_0\|_{\dot H^{1/2}}^n,
\ee
with arbitrarily small $\tilde C$, on the right-hand side of (\ref{2.248}--\ref{2.250}).

At this point we solve the equation system (\ref{2.231}) for $R^n$ and $\pi^n$ in the --- by now --- customary manner. For the modulation path $\pi^n$, (\ref{2.250}) is exactly the estimate we need.

We split $R^n$ into its projections on the three parts of the spectrum, absolutely continuous, null, and imaginary, and estimate each separately in the weighted Strichartz norm.

As previously (see (\ref{2.63}-\ref{2.67})), we first apply a unitary transformation $U(t)$ to the equation, such that
\be\begin{aligned}
U(t) &= e^{\textstyle\int_0^t(2 v^0(s) \dl + i ((\alpha^0)^2(s)-|v^0(s)|^2) \sigma_3) \dd s} \\
Z(t) &= U(t) R^n(t) \\
W(\pi^0) &= U(t) W_{\pi^0}.
\end{aligned}\ee
In case of (\ref{2.231}), this transformation takes a particularly simple form, as $\alpha^0$ and $v^0$ are constant.

For the projection on the continuous spectrum $P_c(t) Z(t)$, Strichartz estimates lead directly, for an arbitrarily small constant $\tilde C$,  to
\be
\|P_c(t) Z(t)\|_{A_n(t) (L^{\infty}_t \dot H^{1/2}_x \cap L^2_t \dot W^{1/2, 6}_x)} \leq \tilde C k_1^2 k_2^n \|R_0\|_{\dot H^{1/2}}^n.
\ee
Concerning the projection on the imaginary spectrum, note the convolution estimate
\be\begin{aligned}
\int_0^t e^{-a_1 (t-s)} A_n(s) \dd s &\leq C A_n(t), \\
\int_t^{\infty} e^{a_1 (t-s)} A_n(s) \dd s &= \int_0^{\infty} e^{-a_1 s} A_n(t+s) \dd s \leq C A_n(t).
\end{aligned}\ee
Since the right-hand side manifests only polynomial growth, it follows that there exists a unique subexponential solution to the ordinary differential equation that describes the projection on the imaginary spectrum, corresponding to a suitable value of the parameter $h^n$.

Then we obtain, for an arbitrarily small constant $\tilde C$,
\be
\|P_{im}(t) Z(t)\|_{A_n(t) (L^{\infty}_t \dot H^{1/2}_x \cap L^2_t \dot W^{1/2, 6}_x)} \leq \tilde C k_1^2 k_2^n \|R_0\|_{\dot H^{1/2}}^n.
\ee

The orthogonality condition is fulfilled approximately, in the sense that
\be
\langle R^n, \Xi_f(W_{\pi^0}) \rangle + \sum_{j=1}^n \sum_{n_1 + \ldots + n_j + \tilde n=n} \langle R^{\tilde n}, (\partial_{\pi}^j \Xi_f(W_{\pi^0})) (\pi^{n_1}, \ldots, \pi^{n_j}) \rangle = 0.
\ee
Under our induction hypothesis, this is enough for an appropriate bound on the $P_0$ component: for arbitrarily small $\tilde C$,
\be
\|P_0(t) Z(t)\|_{A_n(t) (L^{\infty}_t \dot H^{1/2}_x \cap L^2_t \dot W^{1/2, 6}_x)} \leq \tilde C k_1^2 k_2^n \|R_0\|_{\dot H^{1/2}}^n.
\ee

In conclusion, for the unique suitable value of the parameter $h^n$ and for an arbitrarily small constant $\tilde C$,
\be
\|(R^n, \pi^n)\|_{X^{n-1}} \leq \tilde C k_1^2 k_2^n \|R_0\|_{\dot H^{1/2}}^n.
\ee
By setting $k_1$ sufficiently small, we obtain
\be
\|(R^n, \pi^n)\|_{X^{n-1}} \leq k_1 k_2^n \|R_0\|_{\dot H^{1/2}}^n.
\ee
In order to compensate for the smallness of $k_1$, we also need to set $k_2$ to be large, so that the initial condition for induction
\be
\|(R^1, \pi^1)\|_{X^0} \leq k_1 k_2 \|R_0\|_{\dot H^{1/2}}
\ee
is verified.

Next, we verify that $R^n$ and $\pi^n$ are indeed the $n$-th order terms in the power series expansion of $R$ and $\pi$. Denote
\be
S^n = R^0 + R^1 + \ldots + R^n,\ \Sigma^n = \pi^0 + \pi^1 + \ldots + \pi^n.
\ee
We obtain that, up to an error $O^{n+1}$ of order $n+1$, $S^n$ and $\Sigma^n$ solve an equation system of the form (\ref{2.225}):
\be\begin{aligned}\lb{2.237}
&i \partial_t S^n + \mc H_{\Sigma^n}(t) S^n = F(\Sigma^n, S^n) + O^{n+1} \\
&\dot f_{\Sigma^n} = F_f(\Sigma^n, S^n) + O_f^{n+1},\ f \in \{\alpha, \Gamma, v_k, D_k\},
\end{aligned}\ee
with error terms of size
\be
\|O^{n+1}\|_{\langle t \rangle^N L^2_t W^{1/2, 6/5}_x} + \|O_f^{n+1}\|_{\langle t \rangle^N L^1_t} \leq C \|R_0\|_{\dot H^{1/2}}^{n+1}.
\ee
Following a comparison in the exponentially weighted space $Y$ between $(S^n, \Sigma^n)$ and $(R, \pi)$, we obtain that
\be
\|(R, \pi) - (S^n, \Sigma^n)\|_Y \leq C(n) \|R_0\|_{\dot H^{1/2}}^{n+1}.
\ee
This concludes the proof of analyticity for $R$ and $\pi$.


Finally, recall that, by (\ref{2.100}), $h(R_0, W_0)$ has the formula
\be\begin{aligned}\lb{2.266}
h(R_0, W_0) &= -\int_0^{\infty} e^{-\int_0^t \sigma(W_{\pi}(\tau)) \dd \tau} \big(\langle F,  \sigma_3 F^-(W_{\pi}(t)) \rangle - \\
&- 3 \dot \alpha(t) (\alpha(t))^{-4} \langle R, i \sigma_3 F^-(W_{\pi}(t)) \rangle + \\
&+(\alpha(t))^{-3} \langle R, i \sigma_3 (d_{\pi} F^-(W_{\pi}(t))) \dot \pi(t) \rangle\big) \dd t.
\end{aligned}\ee
Since all the components that enter this formula are analytic and grow in time more slowly than $e^{ita_1}$, we directly obtain a power series expansion for $h$.

Alternatively, we can write $h=h^0 + h^1 + \ldots$ and find each term $h^n$ as the initial data of $(R^n, \pi^n)$ in equation (\ref{2.231}). $h^n$ is the unique value that makes the solution bounded. By (\ref{2.266}), it follows that for each $n$ $h^n$ is given by an $n$-linear form in $R_0$, whose norm grows only exponentially with $n$.

Following a comparison between $(S^n, \Sigma^n)$ and $(R, \pi)$, we obtain as a byproduct an estimate for the difference
\be
|h-(h^0 + \ldots + h^n)| \leq C_1 C_2^n \|R_0\|_{\dot H^{1/2}}^{n+1}.
\ee
This gives us another way to establish the analyticity of $h$.
\end{proof}

\section{The centre-stable manifold}
Finally, we state the connection between the manifold $\mc N$ issued by our proof and the centre-stable manifold of \cite{bates}.
\begin{proposition} $\mc N$ is a centre-stable manifold in the sense of Bates--Jones.
\end{proposition}
\begin{proof}
To begin with, we rewrite equation (\ref{NLS}) to make it fit the framework of the theory of Bates--Jones \cite{bates}.

Consider a soliton $W_{\pi_0}(t)$, described by the constant path $\pi_0(t) = (1, 0, 0, 0)$ (without loss of generality). Thus
\be
W_{\pi_0}(t) = \bpm e^{it} \phi(\cdot, 1) \\ e^{-it} \phi(\cdot, 1) \epm
\ee
for all $t$. Linearizing the equation around this constant path yields, for $R = \Psi - W_{\pi_0}$, equation (\ref{rr}), which in this case takes the form
\be
\partial_t R - i \mc H_{\pi_0} R = N(R, W_{\pi_0}).
\ee
Making the substitution
\be
Z = e^{-it\sigma_3} R,
\ee
we see that
\be
i \partial_t Z + \mc H(W(\pi_0)) Z = N(Z, W(\pi_0)),
\lb{eq_2188}
\ee
where
\be
\mc H(W(\pi_0)) = \mc H = \bpm \Delta+2\phi^2(\cdot, 1) -1 & \phi^2(\cdot, 1) \\ -\phi^2(\cdot, 1) & -\Delta-2\phi^2(\cdot, 1)+1 \epm
\ee
and
\be
N(Z, W(\pi_0)) = \bpm -|z|^2 z - z^2 \phi(\cdot, 1) - 2|z|^2 \phi(\cdot, 1) \\ |z|^2 z + z^2 \phi(\cdot, 1) + 2|z|^2 \phi(\cdot, 1) \epm.
\ee
Note that the right-hand side terms are at least quadratic in $Z$, due to linearizing around a constant path.

The spectrum of $\mc H$ is $\sigma(\mc H) = (-\infty, -1] \cup [1, \infty) \cup \{0, \pm i \sigma\}$; for a more detailed discussion see Section \ref{spectru}. The stable spectrum is $-i\sigma$, the unstable spectrum is $i\sigma$, and everything else belongs to the centre.

One can check that all the conditions of \cite{bates} are met in regard to (\ref{eq_2188}), leading to the existence of a centre-stable manifold in this setting. This was done by \cite{ges}.

Indeed, in the Banach space $X=H^s$, $s>3/2$, $\mc H$ is a closed, densely defined operator with the required spectral properties. The nonlinearity $N(Z_1, W(\pi_0))$ has the Lipschitz property with arbitrarily small constant, since $H^s \subset L^{\infty}$ is an algebra:
\be
\|N(Z_1, W(\pi_0)) - N(Z_2, W(\pi_0))\|_{H^s} \leq C \max(\|Z_1\|_{H^s}, \|Z_2\|_{H^s}) \|Z_1-Z_2\|_{H^s}.
\ee

We have also exhibited another stable manifold, $\mc N \subset \dot H^{1/2}$ given by Definition \ref{def3}, 
which is invariant under symmetry transformations and under the time evolution induced by (\ref{NLS}). Our claim, which we prove in the sequel, is that $\Psi$ belongs to this manifold if and only if $Z$ belongs to the other one, that is
\be
\tilde{\mc N} = \mc N - \Phi,\ \Phi = \bpm \phi(\cdot, 1) \\ \phi(\cdot, 1) \epm
\ee
is a centre-stable manifold for (\ref{eq_2188}) relative to a neighborhood $\mc V$ of $0$, namely $\mc V = \{Z \mid \|Z\|_{\dot H^{1/2}} < \delta_0\}$ for some small $\delta_0$.

The main difference between the two results is that $\dot H^{1/2}$ is not an algebra, so the conditions for Bates--Jones's result, Theorem \ref{t1}, are not met in this space. Indeed, if $\psi \in \dot H^{1/2}$, it does not follow that $|\psi|^2 \psi \in \dot H^{1/2}$, much less that this nonlinearity is Lipschitz continuous (see Condition H3 preceding Theorem \ref{t1}).

However, even though the hypothesis fails in $\dot H^{1/2}$, the conclusion of \cite{bates} --- the existence of a centre-stable manifold $\tilde {\mc N}$ --- still holds, following a proper use of Strichartz inequalities.

To show this, we verify the three properties listed in Definition \ref{centr}: $\tilde{\mc N}$ is $t$-invariant with respect to a neighborhood of $\Phi$, $\pi^{cs}(\tilde{\mc N})$ contains a neighborhood of $0$ in $X^c \oplus X^s$, and $\tilde{\mc N} \cap W^u = \{0\}$.

The $t$-invariance of $\tilde{\mc N}$ relative to $\mc V$ follows from definition (see the introduction) and Proposition \ref{comparatie}. Indeed, the invariance established by Corollary \ref{stable} is strictly stronger than $t$-invariance, as it holds globally in time.

The fact that $\pi^{cs}(\tilde{\mc N})$ contains a neighborhood of $0$ in $X^c \oplus X^s$ is a consequence of the local invertibility of $\mc F$ established in Proposition \ref{prop_9}.

Finally, we show that $\tilde{\mc N} \cap W^u = \{0\}$, where $W^u$ is the unstable manifold of the equation (the set of solutions that decay exponentially at $-\infty$). To this purpose, in the sequel we follow the proof of \cite{bec} with slight modifications.

Consider a solution $Z^0 \in \tilde N \cap W^u$ of (\ref{eq_2188}), meaning that $Z^0$ is defined for all $t<0$,
\be
\|Z^0(t)\|_{\dot H^{1/2}} < \delta_0
\ee
for some small $\delta_0$ and all negative $t$, and $Z^0$ decays exponentially as $t \to -\infty$, meaning that for some constants $C_1$ and $C_2$ and $t\leq 0$
\be
\|Z^0(t)\|_{\dot H^{1/2}} \leq C_1 e^{C_2 t}
\ee
(even though polynomial decay is sufficient for a contradiction). Using Strichartz estimates for the free Schr\"{o}dinger evolution, we obtain that for every $T$
\be
\|Z^0\|_{L^2_t(T, T+1) \dot W^{1/2, 6}_x} \leq C \|Z^0\|_{L^{\infty}_t(T, T+1) \dot H^{1/2}_x}.
\ee
Therefore, for $T \leq 0$,
\be
\|Z^0\|_{L^2_t(-\infty, T] \dot W^{1/2, 6}_x} \leq C e^{C_2 T}.
\ee

Thus, $\Psi(t) = e^{it\sigma_3} Z^0(t) + W_{\pi_0}(t)$ is a small asymptotically stable solution of (\ref{NLS}) in a sense similar to Definition \ref{definition_4}, but as $t$ goes to $-\infty$. 

Therefore, for $t \leq 0$ one can write $\Psi = R(t) + W_{\pi}(t)$, such that the orthogonality condition
\be
P_0(W_{\pi}(t)) R(t) = 0
\ee
is satisfied and for $T \leq 0$
\be
\|R(T)\|_{\dot H^{1/2}} \leq C \delta,\ \|R\|_{L^{\infty}_t(-\infty, T] \dot H^{1/2}_x \cap L^2_t(-\infty, T] \dot W^{1/2, 6}_x} \leq C e^{C_2 T}. 
\ee
The path $\pi$ satisfies the corresponding modulation equations and
\be
\|\dot \pi\|_{L^1_t(-\infty, T]} \leq C e^{C_2 T}.
\ee
Since $\Psi(0) \in \mc N$, the path $\pi$ extends to all positive times and $R$ extends to a solution bounded in the Strichartz norm:
\be
\|R\|_{L^{\infty}_t[0, \infty) \dot H^{1/2}_x \cap L^2_t[0, \infty) \dot W^{1/2, 6}_x} \leq C \delta.
\ee
As in previous arguments, let $U(t)$ be the family of isometries defined for $\pi$
\be
U(t) = e^{\textstyle\int_0^t(2 v(s) \dl + i (\alpha^2(s)-|v(s)|^2) \sigma_3) \dd s}
\ee
and set
\be
Z(t) = U(t) R(t).
\ee
Then $Z(t)$ satisfies the equation
\be
i \partial_t Z - \mc H(t) Z = F.
\ee

Decompose $Z$ into its projections on the continuous, imaginary, and zero spectrum of $\mc H(t)$ and let
\be
\delta(T) = \|Z\|_{L^{\infty}_t(-\infty, T] \dot H^{1/2}_x \cap L^2_t(-\infty, T] \dot W^{1/2, 6}_x} + \|\dot {\pi}\|_{L^1_t(-\infty, T]}.
\ee
Observe that $\delta(t) \to 0$ as $t \to -\infty$, so we can assume it to be arbitrarily small.

By means of Strichartz estimates one obtains that
\be
\begin{aligned}
\|P_c(t) Z\|_{L^2_t(-\infty, T] \dot W^{1/2, 6}_x \cap L^{\infty}_t(-\infty, T] \dot H^{1/2}_x} &\leq C \|F\|_{L^2_t(-\infty, T] \dot W^{1/2, 6/5}_x + L^{1}_t(-\infty, T] \dot H^{1/2}_x} \\
&\leq C \delta(T) \|Z\|_{L^{\infty}_t (-\infty, T] \dot H^{1/2}_x},
\end{aligned}
\ee
because the right-hand side contains only quadratic or higher degree terms.

Let
\be
P_{im} Z(t) = b_-(t) F^-(t) + b_+(t) F^+(t).
\ee
Since $Z(t)$ is bounded as $t \to -\infty$, we can use Lemma \ref{hyp} for $t \to -\infty$ and obtain
\be\begin{aligned}
b_-(t) &= -\int_{-\infty}^t e^{-\sigma(t-s)} N_-(s) \dd s\\
b_+(t) &= e^{(t-T)\sigma} b_+(T) - \int_{t}^T e^{(t-s)\sigma} N_+(s) \dd s.
\end{aligned}\lb{PplusU}\ee
Therefore
\be
\|P_{im} Z\|_{L^{\infty}_t(-\infty, T] \dot H^{1/2}_x} \leq C (\|P_+ Z(T)\|_{\dot H^{1/2}_x} + \delta(T) \|Z\|_{L^{\infty}_t(-\infty, T] \dot H^{1/2}_x}).
\ee

We have constructed $Z$ such that the orthogonality condition holds, so
\be
P_0(t) Z(t) = 0.
\ee

Putting these estimates together, we arrive at
\be
\|Z\|_{L^{\infty}_t(-\infty, T] \dot H^{1/2}_x} \leq C (\delta(T) \|Z\|_{L^{\infty}_t(-\infty, T] \dot H^{1/2}_x} + \|P_+ Z(T)\|_{\dot H^{1/2}}).
\ee
For sufficiently negative $T_0$, it follows that $\|Z(t)\|_{\dot H^{1/2}} \leq C \|P_+ Z(t)\|_{\dot H^{1/2}}$, for any $t \leq T_0$. The converse is obviously true, so the two norms are comparable.

Furthermore, by reiteration one has that
\be\lb{2.294}
\|(I-P_+) Z(t)\|_{\dot H^{1/2}} \leq C \delta(t) \|P_+Z(t)\|_{\dot H^{1/2}}.
\ee

Next, assume that $Z(0)$ is on $\tilde {\mc N}$, meaning that $Z(0) + \Phi \in \mc N$.

If the size $\delta_0$ that appears in Definition \ref{definition_4} of $\mc N$ is sufficiently small, it follows that $\|Z(t)\|_{\dot H^{1/2}}$ is bounded from below as $t \to \infty$. Indeed, to a first order, $Z(t)$ is given by the time evolution of the initial data, following (\ref{2.213}), and this expression cannot vanish.

On the other hand, Lemma \ref{hyp} implies that
\be
\|P_+ Z(t)\|_{\dot H^{1/2}} \leq \int_t^{\infty} e^{(t-s)\sigma} |N_+(s)| \dd s
\ee
and thus $\|P_+ Z(t)\|_{\dot H^{1/2}}$ goes to zero and can be made arbitrarily small as $t \to \infty$.

Lemma 2.4 of \cite{bates} implies, under even more general conditions, that if the ratio $\|P_+ Z(T_0)\|_{\dot H^{1/2}}/\|(I-P_+) Z(T_0)\|_{\dot H^{1/2}}$ is small enough, it will stay bounded for all $t \leq T_0$. The proof of this result is based on Gronwall's inequality.

However, this contradicts our previous conclusion (\ref{2.294}), stating that
\be
\|(I-P_+) Z(t)\|_{\dot H^{1/2}}/\|P_+ Z(t)\|_{\dot H^{1/2}}  \leq C \delta(t)
\ee
goes to $0$ as $t$ goes to $-\infty$. Therefore, $Z$ can only be $0$.

This proves that $\tilde{\mc N} \cap W^u = \{0\}$. In other words, there are no exponentially unstable solutions in $\tilde{\mc N}$ in the sense of \cite{bates}. The final requirement for $\tilde{\mc N}$ to be a centre-stable manifold in the sense of Definition \ref{centr} is thus met.
\end{proof}

\section{Scattering}
Strichartz space bounds and estimates imply that the radiation term scatters like the solution of the free equation, meaning
\be
r(t) = e^{-it\Delta} r_{free} + o_{\dot H^{1/2}}(1)
\ee
for some $r_{free} \in \dot H^{1/2}$.

As a reminder, $R$ satisfies the equation (\ref{rr})
\be\lb{2.196}
i \partial_t R - \mc H_{\pi}(t) R = F,
\ee
where $\mc H_{\pi}(t) = \Delta \sigma_3 + V_{\pi}(t)$, and $R$ has finite Strichartz norm,
\be
\|R\|_{L^{\infty}_t \dot H^{1/2}_x \cap L^2_t \dot W^{1/2, 6}_x} < \infty,
\ee
while $F$ has finite dual Strichartz norm,
\be
\|F\|_{L^2_t \dot W^{1/2, 6/5}_x} < \infty.
\ee
Rewrite (\ref{2.196}) as
\be
i \partial_t R - \Delta \sigma_3 R = F - V_{\pi}(t) R.
\ee
By Duhamel's formula,
\be
R(t) = e^{-it\Delta \sigma_3} R(0) -i \int_0^t e^{-i(t-s)\Delta \sigma_3} (F(s) - V_{\pi}(s) R(s)) \dd s.
\ee
Let
\be
R_{free} = R(0) -i \int_0^{\infty} e^{it\Delta \sigma_3} (F(t) - V_{\pi}(t) R(t)) \dd t.
\ee
Then
\be
R(t) - e^{-it\Delta \sigma_3} R_{free} = ie^{-it\Delta \sigma_3} \int_t^{\infty} e^{is\Delta \sigma_3} (F(s) - V_{\pi}(s) R(s)) \dd t.
\ee
Note that
\be\begin{aligned}
& \|F - V_{\pi} R\|_{L^2_t \dot W^{1/2, 6/5}_x} \leq \\
\leq & \|F\|_{L^2_t \dot W^{1/2, 6/5}_x} + \|V_{\pi}\|_{L^{\infty}_t (\dot W^{1/2, 6/5-\epsilon}_x \cap \dot W^{1/2, 6/5+\epsilon}_x)} \|R\|_{L^2_t \dot W^{1/2, 6}_x} < \infty.
\end{aligned}\ee
implies
\be
\lim_{t \to \infty} \|\chi_{[t, \infty)}(s) (F(s) - V_{\pi}(s) R(s))\|_{L^2_t \dot W^{1/2, 6/5}_x} = 0.
\ee
$e^{-it\Delta \sigma_3}$ being an isometry, it follows that $R(t) - e^{-it\Delta \sigma_3} R_{free} \to 0$ in $\dot H^{1/2}$.

This leads to the desired conclusion, upon passing to the scalar functions $r$ and $r_{free}$, where $R = \bpm r \\ \ov r \epm$ and $R_{free} = \bpm r_{free} \\ \ov r_{free} \epm$.

\chapter{Linear estimates}
We seek a dispersive estimate for the linear time-dependent equation. The highest-order terms we need to take into account are of the form $v(t) \dl Z(t)$, where $\dot v(t)$ is small in the $L^1$ norm. $Z$ is a solution of the Schr\"{o}dinger equation and thus has a finite Strichartz norm.

We study the scalar case as a simplified model, together with the nonselfadjoint case in which we are properly interested.

The first piece is an ad hoc Wiener Theorem for abstract spaces, which requires some background.

\section{Motivation} As motivation for this approach, consider the linear Schr\"{o}dinger equation in $\set R^3$
\be
i \partial_t Z + \mc H Z = F,\ Z(0) \text{ given},
\ee
where
\be\lb{eq_3.2}
\mc H = \mc H_0 + V = -\Delta + V
\ee
in the scalar case and
\be
\mc H = \mc H_0 + V = \bpm \Delta-\mu & 0 \\ 0 & -\Delta+\mu \epm + \bpm |W_1| & W_2 \\ -\ov W_2 & -|W_1| \epm
\lb{3.2}
\ee
in the matrix nonselfadjoint case, in which we really are interested for the nonlinear applications. $W_1$ is always taken to be real-valued and the same is true for $W_2$ in the case of interest in the current work.

By Duhamel's formula,
\be\begin{aligned}
Z(t) &= e^{it\mc H} Z(0) - i \int_0^t e^{i(t-s) \mc H} F(s) \dd s \\
&= e^{it\mc H_0} Z(0) - i \int_0^t e^{i(t-s) \mc H_0} F(s) \dd s + i \int_0^t e^{i(t-s)\mc H_0} V Z(s) \dd s.
\end{aligned}\ee
In addition, for any multiplicative decomposition $V = V_1 V_2$ of the potential,
\be
V_2 Z(t) = V_2 \bigg(e^{it\mc H_0} Z(0) - i \int_0^t e^{i(t-s) \mc H_0} F(s) \dd s \bigg) + i \int_0^t (V_2 e^{i(t-s)\mc H_0} V_1) V_2 Z(s) \dd s.
\ee
In the scalar case we are especially interested in the decomposition
\be
V = V_1 V_2,\ V_1 = |V|^{1/2} \sgn V,\ V_2 = |V|^{1/2}.
\lb{3.5}
\ee
In the matrix nonselfadjoint case (\ref{3.2}), an analogous decomposition is
\be
V = V_1 V_2,\ V_1 = \sigma_3 \bpm |W_1| & W_2 \\ \ov W_2 & |W_1| \epm^{1/2},\ V_2 = \bpm |W_1| & W_2 \\ \ov W_2 & |W_1| \epm^{1/2},
\lb{3.6}
\ee
where $\sigma_3$ is the Pauli matrix
\be
\sigma_3 = \bpm 1 & 0 \\ 0 & -1 \epm.
\ee
Consider the kernel defined by
\be
(T_{V_2, V_1} F)(t)= \int_0^t (V_2 e^{i(t-s)\mc H_0} V_1) V_2 F(s) \dd s.
\ee
Thus, at least formally we can write
\be\begin{aligned}
V_2 Z(t) &= (I- i T_{V_2, V_1})^{-1} V_2 \bigg(e^{it\mc H_0} Z(0) - i \int_0^t e^{i(t-s) \mc H_0} F(s) \dd s \bigg).
\end{aligned}\ee
If the operator 
$I - iT_{V_2, V_1}$ can actually be inverted, then the computation is justified. In the sequel we set forth conditions under which this happens.

\section{Abstract theory} Let $H$ be a Hilbert space, $\mc L(H, H)$ be the space of bounded linear operators from $H$ to itself, and $M_t H$ be the set of $H$-valued measures of finite mass on the Borel algebra of $\set R$. $M_t H$ is a Banach space, with the norm
\be\lb{3.11}
\|\mu\|_{M_t H} = \sup \bigg\{\sum_{k=1}^n \|\mu(A_k)\|_H \mid A_k \text{ disjoint Borel sets}\bigg\}.
\ee
Note that the absolute value of $\mu \in M_t H$ given by
\be
|\mu|(A) = \sup\bigg\{\sum_{k=1}^n \|\mu(A_k)\|_H \mid \bigcup_{k=1}^n A_k = A,\ A_k\text{ disjoint Borel sets}\bigg\}
\ee
is a positive measure of finite mass (bounded variation) and $\|\mu(A)\|_H \leq |\mu|(A)$. By the Radon--Nikodym Theorem, $\mu$ is in $M_t H$ if and only if it has a decomposition
\be
\mu = \mu_{\infty} |\mu|
\ee
with $|\mu| \in M$ (the space of real-valued measures of finite mass), $\mu_{\infty} \in L^{\infty}_{d|\mu|(t)} H$.

Furthermore, $\|\mu\|_{M_t H} = \||\mu|\|_{M}$ and the same holds if we replace $H$ by any Banach space.

\begin{definition}\lb{def_k} Let $K = \mc L(H, M_t H)$ be the algebra of bounded operators from $H$ to $M_t H$.
\end{definition}
It has the following natural properties, which are, yet, not completely trivial:
\begin{lemma}
$K$ takes $M_t H$ into itself by convolution, is a Banach algebra under convolution, and multiplication by bounded continuous functions (and $L^{\infty}$ Borel measurable functions) is bounded on $K$:
\be
\|f k\|_K \leq \|f\|_{\infty} \|k\|_K.
\lb{3.8}
\ee
Furthermore, by integrating an element $k$ of $K$ over $\set R$ one obtains $\int_{\set R} k \in \mc L(H, H)$, with $\|\int_{\set R} k\|_{\mc L(H, H)} \leq \|k\|_K$.
\end{lemma}
\begin{proof} Boundedness of multiplication by continuous or $L^{\infty}$ functions follows from the decomposition $\mu = \mu_0 |\mu|$ for $\mu \in M_t H$. The last stated property is a trivial consequence of the definition of $M_t H$.

Let $\mu \in M_t H$, $k \in K$. Consider the product measure $\tilde \mu$ first defined on product sets $A \times B \subset \set R \times \set R$ by $\tilde \mu(A \times B) = k(\mu(B))(A)$. This is again a measure of finite mass, $\tilde \mu \in M_{t, s} H$, and
\be
\|\tilde \mu\|_{M_{t, s} H} \leq \|k\|_K \|\mu\|_{M_t H}.
\ee
We then naturally define the convolution of an element of $K$ with an element of $M_t H$, by setting $k(\mu)(A)=\tilde \mu(\{(t, s) \mid t+s \in A\})$.

Thus, each $k \in K$ defines a bounded translation-invariant linear map from $M_t H$ to itself:
\be
\|k(\mu)\|_{M_t H} \leq \|k\|_K \|\mu\|_{M_t H}.
\lb{3.9}
\ee
The correspondence is bijective, as any translation-invariant $\tilde k \in \mc L(M_t H, M_t H)$ defines an element $k \in K$ by $k(h) = \tilde k(\delta_{t=0}h)$. These operations are indeed inverses of one another.

Associativity follows from Fubini's Theorem. $K$ is a Banach space by definition. The algebra property of the norm is immediate from (\ref{3.9}).
\end{proof}

\begin{observation}\lb{beurling} The Wiener algebra $K$ characterized above is the widest among several spaces that arise naturally. Another choice, $K_1$, comes from replacing $M_t$, the space of measures of finite mass, with $L^1_t$ in Definition \ref{def_k}.

The Beurling subalgebras $B_{p, a}$ arise by substituting $L^1_t$ with its subspaces $\langle t \rangle^{a} L^p_t$, for
\be\lb{Eq3.16}
a \leq 0,\ 1 \leq p \leq \infty,\ \frac a 3 + \frac 1 p > 1.
\ee

We may strengthen $\mc L(H, M_t H)$ to $M_t \mc L(H, H)$ and likewise for all the other examples, thus obtaining a different family of \emph{strong} algebras that we respectively denote by $K_s$, $K_{1s}$, and $B_{p, a s}$.

Among these algebras, only $K$ and $K_s$ are unital. However, adding the unit (the identity operator) together with its multiples to any of the other algebras considered above gives rise to unital algebras in that case as well. We use a subscript $u$ to mark these, e.g.\ $B_{p, a s u}$.
\end{observation}





Returning to $K$, note that, due to our choice of a Hilbert space $H$, if $k \in K$ then $k^* \in K$ also.

Define the Fourier transform of any element in $K$ by
\be
\widehat k(\lambda) = \int_{\set R} e^{-it\lambda} \dd k(t).
\ee
This is a bounded operator from $H$ to itself.
By dominated convergence, $\widehat k(\lambda)$ is a strongly continuous (in $\lambda$) family of operators for each $k$ and, for each $\lambda$,
\be
\|\widehat k(\lambda)\|_{H \to H} \leq \|k\|_K.
\ee
This follows from (\ref{3.8}).

The Fourier transform of the identity is $\widehat I(\lambda) = I$ for every $\lambda$; $\widehat{k^*} = (\widehat k)^*$. Also, the Fourier transform takes convolution to composition.

Trivially, if a kernel $k \in K$ has both a left and a right inverse, they must be the same, $b = b * I = b * (k * B) = (b * k) * B = I * B = B$.

As usual, fix a continuous cutoff $\chi$ supported on a compact set and which equals one on some neighborhood of zero. We also specify that the inverse Fourier on $\set R$ is
\be
f^{\vee}(t) = \frac 1 {2\pi} \int_{\set R} e^{it\lambda} f(\lambda) \dd \lambda.
\ee
\begin{theorem}\lb{thm7}
Let $K$ be the operator algebra of Definition \ref{def_k}. If $A \in K$ is invertible then $\widehat A(\lambda)$ is invertible for every $\lambda$. Conversely, assume $\widehat A(\lambda)$ is invertible for each $\lambda$, $A = I + L$, and
\be
\lim_{\epsilon \to 0} \|L(\cdot + \epsilon) - L\|_K = 0,\ \lim_{R \to \infty}\|(1-\chi(t/R)) L(t)\|_K =0
\ee
Then $A$ is invertible. Furthermore, if $L$ is in any of the aforementioned unital subalgebras of $K$ ($K_{1u}$, $B_{p, a u}$, $K_{su}$, etc.), then its inverse will also belong to the same.
\end{theorem}

Further note that the set of equicontinuous operators, that is
\be
\{L \mid \lim _{\epsilon \to 0} \|L(\cdot + \epsilon) - L\|_K = 0\}
\ee
is a closed ideal, is translation invariant, contains the set of those kernels which are  strongly measurable and $L^1$ (but can be strictly larger), and $I$ is not in it. We could, though, form a Banach algebra $E$ consisting of just multiples of $I$ and this ideal.

Likewise, the set of kernels $L$ that decay at infinity, that is
\be
D = \{L \mid \lim_{R \to \infty}\|\chi_{|t|>R} L(t)\|_K =0\},
\ee
is a closed subalgebra. It contains the strong algebras that we defined above. Note that for operators $A \in D$ the Fourier transform is also a norm-continuous family of operators, not only strongly continuous.

As a final observation, the construction will ensure that, if $L$ belongs to the intersection $E \cap D$ and is invertible, then its inverse is also in it.
\begin{proof}
The proof goes through the usual paces. Firstly, if $A$ is invertible, that is $A * A^{-1} = A^{-1} * A = I$, then applying the Fourier transform yields
\be
\widehat A (\lambda) \widehat {A^{-1}}(\lambda) = \widehat {A^{-1}}(\lambda) \widehat A (\lambda) = I
\ee
for each $\lambda$, so $\widehat A(\lambda)$ is invertible.

Conversely, assume $\widehat A(\lambda)$ is invertible for every $\lambda$. Without loss of generality, we can take $\widehat A$ to be self-adjoint and non-negative for every $\lambda$, by replacing $A$ with $A * A^*$. Then at each $\lambda$ $\widehat A(\lambda)$ is invertible and bounded if and only if
\be
\inf_{\|f\|_{{}_H} = 1} \langle f, \widehat A(\lambda) f \rangle > 0.
\ee
Fix $\lambda_0 \in \set R$. With the help of a smooth cutoff function $\chi$ of compact support, equal to one on some neighborhood of zero, define
\be
\widehat A_{\epsilon}(\lambda) = \Big(1-\chi\big(\frac{\lambda-\lambda_0} {\epsilon}\big)\Big) I + \chi\big(\frac {\lambda-\lambda_0} {\epsilon}\big) \widehat A(\lambda) / \|A\|_K.
\lb{3.18}
\ee
We next prove that $A_{\epsilon}$ is invertible. Without loss of generality we can take $\lambda_0$ to be zero. 

For any kernel $B \in K$ that decays at infinity ($B \in D$, that is, in our previous notation) and for $\chi_{\epsilon} = \nobreak \frac 1 {\epsilon}\chi^{\vee}(\epsilon \cdot)$,
\be
\bigg\|(\chi_{\epsilon} * B)(t) - \chi_{\epsilon}(t) \int_{\set R} B(s) ds\bigg\|_K \to 0.
\ee
This follows, as usual, by fixing some large radius $R$ and integrating separately within and outside that radius:
\be
\|\chi_{|s|>R} B(s)\|_K \to 0,
\ee
\be
\bigg\|\int_{|s|\leq R} (\chi_{\epsilon}(t) - \chi_{\epsilon}(t-s)) B(s) \dd s\bigg\|_K \leq \|B\|_K \cdot \Big\|\sup_{|s|\leq R}(\chi_{\epsilon}(t) - \chi_{\epsilon}(t - s))\Big\|_{L^1_t} \to 0.
\ee
Thus $\chi_{\epsilon} * B$ gets close to the operator
\be
\chi_{\epsilon}(t) \bigg(\int_{\set R} B(s) ds\bigg)
\ee
whose norm equals $\|\chi_{\epsilon}\|_1 \|\widehat B(0)\|_{H \to H}$.

If $\|\widehat B(0)\| < 1/C$, where $C = \|\chi_{\epsilon}\|_{L^1_t}$ is a constant independent of scaling, then $1-\chi_{\epsilon} * B$ is invertible for small enough $\epsilon$. If we only assume that $\|\widehat B(0)\| < 1$, then we replace $B$ by $B^n$ for some large $n$ in the above and get that $1 - \chi_{\epsilon} * B^n$ is invertible for small $\epsilon$. This implies that
\be\lb{Eq3.30}
I-((\chi_{\epsilon}^{\wedge})^{1/n})^{\vee} * B
\ee
is invertible.

We choose the cutoff function $\chi$ in the above so that not only does it have compact support, but it is smooth as well. Then, for $\chi_{\epsilon} = \nobreak \frac 1 {\epsilon}\chi^{\vee}(\epsilon \cdot)$ and for large finite $n$, $(\chi_{\epsilon}^{\wedge})^{1/n}$ is a smooth compactly supported function and its Fourier transform is integrable. This makes the inverse of (\ref{Eq3.30}) belong to the algebra $K$, as required.

In particular, this applies to $B = I - A_{\epsilon}$. Indeed, since $\widehat A(\lambda)$ is positive and invertible, it follows that
\be
(I-A_{\epsilon})^{\wedge}(0) = I - \widehat A(0) / \|A\|_K
\lb{3.23}
\ee
is nonnegative and strictly less than one. Thus there exists an operator (namely $A_{\epsilon}$) whose Fourier transform equals that of $A$ on some neighborhood of $\lambda_0$ and which is invertible.

We have to consider infinity separately. Let
\be
\widehat A_R(\lambda) = (1-\chi(\lambda/R)) \widehat A(\lambda) + \chi(\lambda/R) I.
\ee
The difference between $A_R$ and $I$ is given by
\be
(I - \chi_R) * (I - A),
\ee
where $\chi_R = \frac 1 R \chi^{\vee}(R \cdot)$. At this step we use the equicontinuity assumption of the hypothesis, namely
\be
\lim_{\epsilon \to 0}\|(I-A)-(I-A)(\cdot + \epsilon)\|_K = 0.
\ee
Since $\chi_R$ is a good kernel, we separate it into two parts, away from zero and close to zero, and obtain
\be\begin{aligned}
\limsup_{R \to \infty} \|\chi_{[-\epsilon, \epsilon]} \chi_R * (I-A) - (\int_{\set R} \chi_{[-\epsilon, \epsilon]} \chi_R) (I-A)\|_K = o_{\epsilon}(1),\\
\lim_{R \to \infty} \|(1-\chi_{[-\epsilon, \epsilon]}) \chi_R\|_1 = 0.
\end{aligned}\ee
Therefore
\be
\lim_{R \to \infty} \|(I - \chi_R) * (I - A)\|_K = 0.
\ee
Thus we can invert $A_R$ for large $R$. It follows that on some neighborhood of infinity the Fourier transform of $A$ equals that of an invertible operator.

Finally, using a finite partition of unity subordinated to those neighborhoods we have found above, we explicitly construct the inverse of $A$. Indeed, consider a finite open cover of $\set R$ of the form
\be
\set R = D_{\infty} \cup \bigcup_{j=1}^n D_j,
\ee
where $D_j$ are open sets and $D_{\infty}$ is an open neighborhood of infinity. Also assume that for $1 \leq j \leq n$ and for $j=\infty$ we have $\widehat A^{-1} = \widehat A_j^{-1}$ on the open set $D_j$. Take a smooth partition of unity subordinated to this cover, that is
\be
1 = \sum_j \chi_j,\ \supp \chi_j \subset D_j.
\ee
Then the inverse of $A$ is given by
\be
A^{-1} = \sum_{j=1}^n \widehat \chi_j * A_j^{-1} + (I - \sum_{j=1}^n \widehat \chi_j) * A_{\infty}^{-1}.
\ee
Given our use of smooth cutoff functions, this construction also preserves the subalgebras we defined.
\end{proof}

We are also interested in whether, if $A$ is upper triangular (meaning that $A$ is supported on $[0, \infty)$), the inverse of $A$ is also upper triangular.
\begin{lemma}\lb{lemma8}
Given $A \in K$ upper triangular with $A^{-1} \in K$, $A^{-1}$ is upper triangular if and only if $\widehat A$ can be extended to a weakly analytic family of invertible operators in the lower half-plane, which is continuous up to the boundary, uniformly bounded, and with uniformly bounded inverse.
\end{lemma}
\begin{proof}
Given that $A^{-1}$ and $A$ are upper triangular, one can construct $\widehat A(\lambda)$ and $\widehat {A^{-1}}(\lambda)$ in the lower half-plane, as the integral converges there. Strong continuity follows by dominated convergence and weak analyticity by means of the Cauchy integral formula. Furthermore, both $\widehat A(\lambda)$ and $\widehat {A^{-1}}(\lambda)$ are bounded by the respective norms and they are inverses of one another.

Conversely, consider $A_- = \chi_{(-\infty, 0]} A^{-1}$. On the lower half-plane, $\widehat {A^{-1}} = (\widehat A)^{-1}$ is uniformly bounded by assumption. Likewise, $\widehat A_+$ is bounded as the Fourier transform of an upper triangular operator. Since $A_- = A - A_+$, it too is bounded on the lower half-plane.

However, $A_-$ is lower triangular, so its Fourier transform is also bounded in the upper half-plane. By Liouville's theorem, then, $\widehat A_-$ it must be constant, so $A_-$ can only have singular support at zero. Therefore $A$ is upper triangular.
\end{proof}

In none of the above did we use compactness or the Fredholm alternative explicitly. (Still, it is interesting to note that a subset of $L^1_t H$ is precompact if and only if its elements are uniformly bounded, equicontinuous, and decay uniformly at infinity --- conditions that we actually employed).

We next apply this abstract theory to the particular case of interest.

\section{The free evolution and resolvent in three dimensions}
We return to the concrete case (\ref{eq_3.2}) or (\ref{3.2}) of a linear Schr\"odinger equation  on $\set R^3$ with scalar or matrix nonselfadjoint potential $V$. For simplicity, the entire subsequent discussion revolves around the case of three spatial dimensions.

We mainly consider $L^2$ estimates, with no derivatives, which are the most natural. However, for the nonlinear problem, the most relevant are $\dot H^{1/2}$ estimates, involving half-derivatives.

In order to apply the abstract Wiener theorem, Theorem \ref{thm7}, it is necessary to exhibit an operator-valued measure of finite mass. Accordingly, we start with the following:
\begin{proposition}\lb{prop_21} Let $\mc H_0$ be as in (\ref{eq_3.2}) or (\ref{3.2}), that is
$$
\mc H_0 = -\Delta \text{ or } \mc H_0 = \bpm \Delta-\mu & 0 \\ 0 & -\Delta+\mu \epm.
$$
Then
\be\lb{eqn_3.39}
\int_0^{\infty} \|e^{it \mc H_0} f\|_{L^{6, \infty}} \dd t \leq C \|f\|_{L^{6/5, 1}}.
\ee
\end{proposition}
Thus, the free time evolution $e^{it \mc H_0}$, when considered between the proper spaces, is indeed an operator-valued measure of finite mass. As a reminder, $L^{p, q}$ are Lorentz spaces (see Appendix \ref{spaces}); $L^{p, \infty}$ is also known as weak-$L^p$.

The proof produces a stronger result, namely
\be\lb{eq_3.39}
\sum_{n \in \set Z} 2^n \sup_{t \in [2^n, 2^{n+1})} \|e^{it \mc H_0} f\|_{L^{6, \infty}} \leq C \|f\|_{L^{6/5, 1}}.
\ee
\begin{proof} We provide two proofs --- a shorter one based on real interpolation and a longer one, using the atomic decomposition of Lorentz spaces (Lemma \ref{lemma_30}), that exposes the proof machinery underneath.

Following the first approach, note that by duality (\ref{eq_3.39}) is equivalent to
\be\lb{eq_3.40}
\sum_{n \in \set Z} 2^n \int_{2^n}^{2^{n+1}} |\langle e^{it\mc H_0} f, g(t) \rangle| \dd t \leq C \|f\|_{L^{6/5, 1}} \|g\|_{L^1_t L^{6/5, 1}_x}
\ee
holding for any $f \in L^{6/5, 1}$ and $g \in L^1_t L^{6/5, 1}_x$.

From the usual dispersive estimate
\be
\|e^{-it\Delta} f\|_{p'} \leq t^{3/2(1-2/p)} \|f\|_p
\ee
we obtain that
\be
\int_{2^n}^{2^{n+1}} |\langle e^{it\mc H_0} f, g(t) \rangle| \dd t \leq 2^{n(3/2-3/p)} \|f\|_p \|g\|_{L^1_t L^p_x}.
\ee
Restated, this means that the bilinear mapping
\be\begin{aligned}
& T:L^p \times L^1_t L^p_x \to \ell^{\infty}_{3/p-3/2} \text{ (following the notation of Proposition \ref{prop_33})},\\
& T=(T_n)_{n \in \set Z},\ T_n(f, g) = \int_{2^n}^{2^{n+1}} \langle e^{it\mc H_0} f, g(t) \rangle \dd t
\end{aligned}\ee
is bounded for $1 \leq p \leq 2$.

Interpolating between $p=1$ and $p=2$, by using the real interpolation method (Theorem \ref{thm_32}) with $\theta=1/3$ and $q_1=q_2=1$, directly shows that
\be
T: (L^1, L^2)_{1/3, 1} \times (L^1_t L^1_x, L^1_t L^2_x)_{1/3, 1} \to (\ell^{\infty}_{3/2}, \ell^{\infty}_0)_{1/3, 1}
\ee
is bounded. By Proposition \ref{prop_33},
\be\begin{aligned}
(L^1, L^2)_{1/3, 1} &= L^{6/5, 1},\\
(L^1_t L^1_x, L^1_t L^2_x)_{1/3, 1} &= L^1_t (L^1_x, L^2_x)_{1/3, 1} = L^1_t L^{6/5, 1}_x,\\
(\ell^{\infty}_{3/2}, \ell^{\infty}_0)_{1/3, 1} &= \ell^1_1.
\end{aligned}
\ee
Hence $T$ is bounded as a map from $L^{6/5, 1} \times L^1_t L^{6/5, 1}_x$ to $\ell^1_1$, which implies (\ref{eq_3.39}).

The alternative approach is based on the atomic decomposition of $L^{6/5, 1}$. By Lemma \ref{lemma_30},
\be\begin{aligned}
f = \sum_{j \in \set Z} \alpha_j a_j, g(t) = \sum_{k \in \set Z} \beta_k(t) b_k(t),
\end{aligned}\ee
where $a_j$ and, for each $t$, $b_k(t)$ are atoms with
\be\begin{aligned}
\mu(\supp(a_j)) &= 2^j,& \esssup |a_j| &= 2^{-5j/6},\\
\mu(\supp(b_k(t))) &= 2^k,& \esssup |b_k(t)| &= 2^{-5k/6}
\end{aligned}\ee
(here $\mu$ is the Lebesgue measure on $\set R^3$), and the coefficients $\alpha_j$ and $\beta_k(t)$ satisfy
\be
\sum_{j \in \set Z} |\alpha_j| \leq C \|f\|_{L^{6/5, 1}},\ \sum_{k \in \set Z} |\beta_k(t)| \leq C \|g(t)\|_{L^{6/5, 1}}.
\ee
Integrating in time and exchanging summation and integration lead to
\be
\sum_k \int_0^{\infty} |\beta_k(t)| \dd t \leq C \|g\|_{L^1_t L^{6/5, 1}_x}.
\ee
Since (\ref{eq_3.40}) is bilinear in $f$ and $g$, it suffices to prove it for only one pair of atoms. Fix indices $j_0$ and $k_0 \in \set Z$; the problem reduces to showing that
\be\lb{eq_3.51}
\sum_{n \in \set Z} 2^n \int_{2^n}^{2^{n+1}} \langle e^{it \mc H_0} a_{j_0}, \beta_{k_0}(t) b_{k_0}(t) \rangle \dd t \leq C \int_0^{\infty} |b_{k_0}(t)| \dd t.
\ee
The reason for making an atomic decomposition is that atoms are in $L^1 \cap L^{\infty}$, instead of merely in $L^{6/5, 1}$, enabling us to employ both $L^1$ to $L^{\infty}$ decay and $L^2$ boundedness estimates in the study of their behavior. For each $n$ one has
\be\begin{aligned}\lb{eq_3.52}
\int_{2^n}^{2^{n+1}} \langle e^{it \mc H_0} a_{j_0}, \beta_{k_0}(t) b_{k_0}(t) \rangle \dd t &\leq C 2^{-3n/2} \|a_{j_0}\|_1 \sup_t \|b_{k_0}(t)\|_1 \int_{2^n}^{2^{n+1}} |\beta_{k_0}(t)| \dd t \\
&= C 2^{-3n/2} 2^{j_0/6} 2^{k_0/6} \int_{2^n}^{2^{n+1}} |\beta_{k_0}(t)| \dd t
\end{aligned}\ee
as a consequence of the $L^1 \to L^{\infty}$ decay estimate. At the same time, by the $L^2$ boundedness of the evolution,
\be
\begin{aligned}\lb{eq_3.53}
\int_{2^n}^{2^{n+1}} \langle e^{it \mc H_0} a_{j_0}, \beta_{k_0}(t) b_{k_0}(t) \rangle \dd t & \leq C \|a_{j_0}\|_2 \sup_t \|b_{k_0}(t)\|_2 \int_{2^n}^{2^{n+1}} |\beta_{k_0}(t)| \dd t \\
&= C 2^{-j_0/3} 2^{-k_0/3} \int_{2^n}^{2^{n+1}} |\beta_{k_0}(t)| \dd t.
\end{aligned}
\ee
Using the first estimate (\ref{eq_3.52}) for large $n$, namely $n\geq j_0/3 + k_0/3$, and the second estimate (\ref{eq_3.53}) for small $n$, $n<j_0/3 + k_0/3$, we always obtain that
\be
\int_{2^n}^{2^{n+1}} \langle e^{it \mc H_0} a_{j_0}, \beta_{k_0}(t) b_{k_0}(t) \rangle \dd t \leq C 2^{-n} \int_{2^n}^{2^{n+1}} |\beta_{k_0}(t)| \dd t.
\ee
Multiplying by $2^n$ and summing over $n \in \set Z$ we retrieve (\ref{eq_3.51}), which in turn proves (\ref{eq_3.40}).
\end{proof}

The resolvent of the unperturbed Hamiltonian, $R_0(\lambda) = (\mc H_0 - \lambda)^{-1}$, is given by
\be\lb{eq_3.55}
R_0(\lambda^2)(x, y) = \frac 1 {4\pi} \frac {e^{i \lambda |x-y|}}{|x-y|}
\ee
in the scalar case (\ref{3.2}) and
\be\lb{eq_3.56}
R_0(\lambda^2+\mu)(x, y) = \frac 1 {4\pi} \bpm -\frac {e^{-\sqrt{\lambda^2+2\mu} |x-y|}}{|x-y|} & 0 \\ 0 & \frac {e^{i \lambda |x-y|}}{|x-y|} \epm
\ee
in the matrix case (\ref{eq_3.2}).

In either case, $R_0(\lambda) = (\mc H_0 - \lambda)^{-1}$ is an analytic function, on $\set C \setminus [0, \infty)$ or respectively on $\set C \setminus ((-\infty, -\mu] \cup [\mu, \infty))$. It can be extended to a continuous function in the closed lower half-plane or the closed upper half-plane, but not both at once, due to a jump discontinuity on the real line.

The resolvent is the Fourier transform of the time evolution. We formally state this known connection between $e^{it\mc H_0}$ and the resolvent $R_0 = (\mc H_0 - \lambda)^{-1}$.
\begin{lemma}\lb{lemma_22}
Let $\mc H_0$ be given by (\ref{eq_3.2}) or (\ref{3.2}). For any $f \in L^{6/5, 1}$ and $\lambda$ in the lower half-plane, the integral
\be
\lim_{\rho \to \infty} \int_0^{\rho} e^{-it\lambda} e^{it \mc H_0} f \dd t
\lb{3.47}
\ee
converges in the $L^{6, \infty}$ norm and equals $i R_0(\lambda) f$ or $i R_0(\lambda-i0) f$ in case $\lambda \in \set R$.

Furthermore, for real $\lambda$,
\be
\lim_{\rho \to \infty} \int_{-\rho}^{\rho} e^{-it\lambda} e^{it \mc H_0} f \dd t = i(R_0(\lambda-i0) - R_0(\lambda+i0)) f,
\lb{3.39}\ee
also in the $L^{6, \infty}$ norm.
\end{lemma}
\begin{proof}

Note that (\ref{3.47}) is dominated by (\ref{eqn_3.39}),
\be
\int_0^{\infty} \|e^{it \mc H_0} f\|_{L^{6, \infty}},
\ee
and this ensures its absolute convergence. Next, both (\ref{3.47}), as a consequence of the previous argument, and $i R_0(\lambda+i0)$ are bounded operators from $L^{6/5, 1}$ to $L^{6, \infty}$. To show that they are equal, it suffices to address this issue over a dense set. Observe that
\be
\int_0^{\rho} e^{-it(\lambda-i\epsilon)} e^{it \mc H_0} f \dd t = iR_0(\lambda-i\epsilon) (f - e^{-i\rho(\lambda-i\epsilon)} e^{i\rho \mc H_0} f).
\ee
Thus, if $f \in L^2 \cap L^{6/5, 1}$, considering the fact that $e^{it \mc H_0}$ is unitary and $R_0(\lambda-i\epsilon)$ is bounded on $L^2$,
\be
\lim_{\rho \to \infty} \int_0^{\rho} e^{-it(\lambda-i\epsilon)} e^{it \mc H_0} f \dd t = iR_0(\lambda-i\epsilon) f.
\lb{3.54}
\ee
Letting $\epsilon$ go to zero, the left-hand side in (\ref{3.54}) converges, by dominated convergence, to (\ref{3.47}), while the right-hand side (also by dominated convergence, using the explicit form (\ref{eq_3.55})-(\ref{eq_3.56}) of the operator kernels) converges to $iR_0(\lambda-i0) f$. Statement (\ref{3.39}) follows directly.
\end{proof}

\section{The exceptional set and the resolvent}
We explore further properties of the resolvent in the $L^2$ setting. Important in this context is the \emph{Birman-Schwinger operator},
\be
\widehat T_{V_2, V_1} (\lambda) = i V_2 R_0(\lambda) V_1,
\ee
where $V = V_1 V_2$ and $V_1$, $V_2$ are as in (\ref{3.5}) or (\ref{3.6}).

The relation between the Birman-Schwinger operator and the resolvent $R_V=(\mc H_0 + V - \lambda)^{-1}$ is that
\be\lb{eq_3.64}
R_V(\lambda) = R_0(\lambda) - R_0(\lambda) V_1 (I + V_2 R_0(\lambda) V_1)^{-1} V_2 R_0(\lambda)
\ee
and
\be\lb{eq_3.65}
(I + V_2 R_0(\lambda) V_1)^{-1} = I - V_2 R_V(\lambda) V_1.
\ee

Note that, for $V \in L^{3/2, \infty}$, $V_2 R_0(\lambda) V_1$ is $L^2$-bounded for every value of $\lambda$. At the $\dot H^{1/2}$ level, assuming that $V = V_1 V_2$ and $V_1$, $V_2 \in \dot H^{1/2}$, it follows that $\widehat T_{V_2, V_1}(\lambda)$ is $\dot H^{1/2}$-bounded.

Indeed, making use of the sharp form of the Sobolev embedding $\dot H^{1/2} \subset L^{3, 2}$, it follows that $R_0(\lambda) V_1 f$ is in $\dot W^{1/2, 6} \cap L^{\infty}$ for $f \in \dot H^{1/2}$ and therefore, by the fractional Leibniz rule, $V_2 R_0(\lambda) V_1 f$ is in $\dot H^{1/2}$ again.

However, to avoid such endpoint-related complications, in the sequel we assume when proving $\dot H^{1/2}$ estimates that $V = V_1 V_2$ with $V_1$, $V_2 \in \dot W^{1/2, 2+\epsilon} \cap \dot H^{1/2}$.

The Rollnick class is the set of measurable potentials $V$ whose Rollnick norm
\be
\|V\|_{\mc R} = \int_{(\set R^3)^2} \frac {|V(x)| |V(y)|}{|x-y|^2} \dd x \dd y
\ee
is finite. The Rollnick class $\mc R$ contains $L^{3/2}$. For a potential $V \in \mc R$, the operator $\widehat T_{V_2, V_1}$ described above is Hilbert-Schmidt for every value of $\lambda$ in the lower half-plane up to the boundary. We obtain that $\widehat T_{V_2, V_1}$ is compact whenever $V$ is in $L^{3/2, \infty}$ and can be approximated by bounded compactly supported functions.

\begin{definition}\lb{def_7} Given $V \in L^{3/2}$, its exceptional set $\mc E$ is the set of $\lambda$ in the complex plane for which $I - i\widehat T_{V_2, V_1}(\lambda)$ is not invertible from $L^2$ to itself.
\end{definition}

Other choices of $V_1$ and $V_2$ such that $V = V_1 V_2$, $V_1$, $V_2 \in L^{3}$ lead to the same operator up to conjugation.

Below we summarize a number of observations concerning the exceptional sets of operators in the form (\ref{eq_3.2}) or (\ref{3.2}).
\begin{proposition} Assume $V \in L^{3/2}$ is a potential as in (\ref{eq_3.2}) or (\ref{3.2}) and denote its exceptional set by $\mc E$.

\begin{itemize}\item $\mc E$ is bounded and discrete outside $\sigma(\mc H_0)$, but can accumulate toward $\sigma(\mc H_0)$. $\mc E \cap \sigma(\mc H_0)$ has null measure (as a subset of $\set R$). Elements of $\mc E \setminus \sigma(\mc H_0)$ are eigenvalues of $\mc H = \mc H_0 + V$.

\item If $V$ is real-valued and scalar as in (\ref{eq_3.2}), then the only possible exceptional value in $\sigma(\mc H_0)$ is zero (the endpoint of the spectrum) and $\mc E$ is located on the real axis. If $V$ is scalar as in (\ref{eq_3.2}) and complex-valued, then embedded exceptional values can occur, but only zero can can be an eigenvalue.

\item If $V$ is real and matrix-valued as in (\ref{3.2}), then embedded exceptional values must be eigenvalues, except for the endpoints of $\sigma(\mc H_0)$, which need not be eigenvalues. If $V$ is complex matrix-valued as in (\ref{3.2}), there is no restriction on embedded exceptional values.

\item If $V$ is complex scalar as in (\ref{eq_3.2}) or complex matrix-valued as in (\ref{3.2}), then $\mc E$ is symmetric with respect to the real axis. In case $V$ is real-valued and as in (\ref{3.2}), $\mc E$ is symmetric with respect to both the real axis and the origin.
\end{itemize}
\end{proposition}

\begin{proof}
By the analytic and meromorphic Fredholm theorems (for statements see \cite{reesim}, p.\ 101, and \cite{reesim4}, p.\ 107), the exceptional set $\mc E$ is closed, bounded, and consists of at most a discrete set outside $\sigma(\mc H_0)$, which may accumulate toward $\sigma(\mc H_0)$, and a set of measure zero contained in $\sigma(\mc H_0)$.

Assuming that $V \in L^{3/2}$ is real-valued and scalar, the exceptional set resides on the real line. Indeed, if $\lambda$ is exceptional, then by the Fredholm alternative (\cite{reesim1}, p.\ 203) the equation
\be
f = -V_2 R_0(\lambda) V_1 f
\ee
must have a solution $f \in L^2$. Then $g = R_0(\lambda) V_1 f$ is in $|\dl|^{-2} L^{6/5, 2} \subset L^{6, 2}$ and satisfies
\be
g = - R_0(\lambda) V g.
\lb{3.471}
\ee

If $\lambda \in \mc E \setminus \sigma(\mc H_0)$, the kernel's exponential decay implies that $\lambda$ is an eigenvalue for $\mc H$ and that the corresponding eigenvectors must be at least in $\langle \dl \rangle^{-2} L^{6/5, 2}$.

Furthermore, by applying $\mc H_0 -\lambda$ to both sides we obtain
\be
(\mc H_0 + V - \lambda) g = 0.
\ee
In the case of a real scalar potential $V$, $\mc H_0+V$ is self-adjoint, so this is a contradiction for $\lambda \not \in \set R$. In general, exceptional values off the real line can indeed occur.

Next, we analyze exceptional values embedded in $\sigma(\mc H_0)$. In the real scalar case (\ref{eq_3.2}), the reasoning of Goldberg--Schlag \cite{golsch} (see Lemma 9 there) implies that zero is the only possible embedded exceptional value. Indeed, take $\lambda > 0$ to be an exceptional value for $V \in L^{3/2}$, which is real-valued and scalar. Then the pairing
\be
\langle g, V g \rangle = \langle R_0(\lambda\pm i0) V g, V g \rangle
\ee
is well-defined and yields a real value. Therefore $\widehat{V g}(\xi) = 0$ on the sphere of radius $|\xi|=\sqrt \lambda$. Following Proposition 7 of \cite{golsch}, for $1 \leq p < 4/3$, $\delta=1/2-1/p'>0$, and any function $h$ whose Fourier transform vanishes on the sphere,
\be
\|\langle x \rangle^{-1/2+\delta} R_0(\lambda\pm i0) h\|_2 \leq C \|h\|_p.
\ee
Setting $h = Vg$, we infer that $\langle x \rangle^{-1/2+\delta} g \in L^2$.
Then, one can apply Ionescu--Jerison's result of \cite{ionjer} and conclude that $g=0$.

In the scalar case (\ref{eq_3.2}), if $V \in L^{3/2}$ is not real-valued, then the exceptional set need not consist only of eigenvalues or zero. Rather, the opposite is true:
\begin{lemma}
Consider a real-value or complex-valued potential $V \in L^{3/2}$. Then there exists $z \in \set C$ such that $1$ is an embedded exceptional value for the potential $zV$. 
\end{lemma}
Here $1$ can be replaced with any other number in $(0,\infty)$.
\begin{proof}
Let $\mc H_0 = -\Delta$, $R_0(\lambda) = (\mc H_0 - \lambda)^{-1}$, $V_1 = |V|^{1/2}$, $V_2 = |V|^{1/2} \sgn V$, and consider the equation
\be
f = - z V_2 R_0(1-i0) V_1 f
\ee
for $z \in \set C$. The existence of a nonzero solution $f \in L^2(\supp V)$ is equivalent to $-1/z$ being in $\sigma(V_2 R_0(1-i0) V_1)$, by the Fredholm alternative.

The spectrum of $V_2 R_0(1-i0) V_1$, as an operator on $L^2(\supp V)$, is nonempty by Liouville's Theorem. We next ascertain whether $0 \in \sigma(V_2 R_0(1-i0) V_1)$.

If $0 \in \sigma(V_2 R_0(1-i0) V_1)$, then for some nonzero $f \in L^2(\supp V)$
\be
V_2 R_0(1-i0) V_1 f = 0.
\ee
This would immediately imply that $\langle V_1 f, R_0(1-i0) V_1 f \rangle = 0$ and consequently $(V_1 f)^{\wedge}(\xi) = 0$ on the sphere $|\xi|=1$.

Let $R_0(1-i0) V_1 f = h \in L^6$. $h$ has the property that $(-\Delta-1)h = V_1 f$ and $V_1 h = 0$ almost everywhere. We infer that $h \in W^{2, 6/5}_{loc}$ and $\ov h(x) (-\Delta-1) h(x) = 0$ pointwise in $x$ almost everywhere.

Consequently, $-\Delta |h|^2 \leq 2 |h|^2$ almost everywhere and thus $|h|$ is a continuous function. Therefore, the set on which $h \ne 0$ is open and there one has $(-\Delta - 1)h = 0$. It follows that $h$ is continuous and it is analytic on the set where it is nonzero.

Furthermore, since gradients vanish on the inverse of small sets (Theorem 6.19, Lieb and Loss p.\ 154 \cite{lieb}, following Almgren-Lieb), $\Delta h = 0$ almost everywhere where $h=0$. It follows that $-\Delta h = h$ almost everywhere. Therefore $V_1 f = 0$ almost everywhere. This contradicts our assumption that $f \ne 0$ in $L^2(\supp V)$. 

Thus $0 \not \in \sigma(V_2 R_0(1-i0) V_1)$ and the spectrum must contain some nonzero value $-1/z$. Then $1-i0$ is an embedded exceptional value for the potential $zV$.
\end{proof}

\cite{ionjer} still implies in this case that any embedded eigenvalue must be zero (a solution $g$ of (\ref{3.471}) must be zero if $\langle x \rangle^{-1/2+\delta} g \in ~L^2$ and $V \in L^{3/2}$, regardless of whether $V$ is real or complex-valued), but one cannot always bootstrap $g$ to this space. Embedded exceptional values can occur, as seen previously, but they cannot be eigenvalues.

For real-valued $V \in L^{3/2}$ having the matrix form (\ref{3.2}) we retain the other half of the argument. \cite{ionjer} does not apply and embedded eigenvalues can occur. The half that still applies is that any embedded exceptional values must be eigenvalues, following the argument of Lemma 4 of Erdogan--Schlag \cite{erdsch2}.

Explicitly, consider $\lambda \in \mc E \cap \sigma(\mc H_0) \setminus \{\pm \mu\}$; without loss of generality $\lambda>\mu$. It corresponds to a nonzero solution $G \in L^6$ of
\be\lb{eq_3.75}
G = - R_0(\lambda-i0) V G.
\ee
We show that $G \in L^2$ and that it is an eigenfunction of $\mc H$. Let
\be
G = \bpm g_1 \\ g_2 \epm,\ V = \bpm W_1 & W_2 \\ -W_2 & -W_1 \epm,\ \mc H_0 = \bpm \Delta - \mu & 0 \\ 0 & -\Delta + \mu \epm,
\ee
where $W_1$ and $W_2$ are real-valued. We expand (\ref{eq_3.75}) accordingly into
\be\begin{aligned}\lb{eq_3.77}
g_1 &= (-\Delta + \lambda + \mu)^{-1}(W_1 g_1 + W_2 g_2) \\
g_2 &= (-\Delta - (\lambda - \mu - i0))^{-1}(W_2 g_1 + W_1 g_2).
\end{aligned}\ee
This implies that $g_1 \in L^1 \cap L^6$ and
\be\begin{aligned}
\langle g_2, W_2 g_1 + W_1 g_2 \rangle &= \langle (-\Delta - (\lambda - \mu - i0)^{-1}(W_2 g_1 + W_1 g_2), (W_2 g_1 + W_1 g_2) \rangle, \\
\langle g_1, W_2 g_2 \rangle &= \langle (-\Delta + \lambda + \mu)^{-1}(W_1 g_1 + W_2 g_2), W_2 g_2 \rangle, \\
\langle g_1, W_1 g_1 \rangle &= \langle (-\Delta + \lambda + \mu)^{-1}(W_1 g_1 + W_2 g_2), W_1 g_1 \rangle.
\end{aligned}\ee
However,
\be
\langle g_2, W_2 g_1 + W_1 g_2 \rangle = \ov{\langle g_1, \ov W_2 g_2 \rangle} + \langle g_2, W_1 g_2 \rangle. 
\ee
Since $W_2$ is real-valued, it follows that
\be
\langle (-\Delta - (\lambda - \mu - i0)^{-1}(W_2 g_1 + W_1 g_2), (W_2 g_1 + W_1 g_2) \rangle
\ee
is real-valued. Therefore the Fourier transform vanishes on a sphere:
\be
(W_2 g_1 + W_1 g_2)^{\wedge}(\xi) = 0
\ee
for $|\xi|^2 = \lambda-\mu$. We then apply Agmon's bootstrap argument, as follows. By Corollary 13 of \cite{golsch}, if $f \in L^1$ has a Fourier transform that vanishes on the sphere, meaning $\hat f(\xi) = 0$ for every $\xi$ such that $|\xi|^2 = \lambda \ne 0$, then
\be
\|R_0(\lambda \pm i0) f\|_2 \leq C_{\lambda} \|f\|_1.
\ee
Interpolating between this and
\be
\|R_0(\lambda \pm i0) f\|_{4} \leq C_{\lambda} \|f\|_{4/3},
\ee
which holds without conditions on $\hat f$, we obtain that for $1 \leq p \leq 4/3$ and for $\hat f = 0$ on the sphere of radius $\sqrt \lambda>0$
\be
\|R_0(\lambda \pm i0) f\|_{2p/(2-p)} \leq C_{\lambda} \|f\|_{p}.
\ee
Thus, starting with the right-hand side of (\ref{eq_3.77}) in $L^{6/5}$, we obtain that $g_2 \in L^3$, a gain over $L^2$. Iterating once, we obtain that $g_2 \in L^2$. Therefore $g$ is an $L^2$ eigenvector.

Thus, in case of a real-valued $V \in L^{3/2}$ having the matrix form (\ref{3.2}), the exceptional set consists only of eigenvalues, potentially together with the endpoints of the continuous spectrum $\pm \mu$.

For a complex potential of the form (\ref{3.2}), neither of the previous arguments holds. Embedded exceptional values can occur and they need not be eigenvalues.

Next, we examine symmetries of the exceptional set $\mc E$. If $V$ is real-valued and scalar, we have already characterized $\mc E$ as being situated on the real line. If $V$ is scalar, but complex-valued, then consider an exceptional value $\lambda$, for which, due to compactness, there exists $f \in L^2$ such that
\be
f = -|V|^{1/2} \sgn V R_0(\lambda) |V|^{1/2} f.
\ee
Then
\be
(\sgn V \ov f) = -|V|^{1/2} R_0(\ov \lambda) |V|^{1/2} \sgn {\ov V} (\sgn V \ov f),
\ee
so the adjoint has an exceptional value at $\ov \lambda$. However, $\sigma(\widehat T_{V_2, V_1}(\lambda)) = \sigma(\widehat T_{V_1, V_2}(\lambda)^*)$, so all this proves that the exceptional set $\mc E$ is symmetric with respect to the real axis.

If $V$ has the matrix form (\ref{3.6}), then note that $\sigma_1 V \sigma_1 = -\ov V$, $\sigma_3 V \sigma_3 = V^*$, where $\sigma_1$ is the Pauli matrix
\be
\sigma_1 = \bpm 0 & 1 \\ 1 & 0 \epm,\ \sigma_1 \sigma_3 = -\sigma_3 \sigma_1.
\ee
Let $\lambda$ be an exceptional value, for which
\be
f = -\sigma_3 (\sigma_3 V)^{1/2} R_0(\lambda) (\sigma_3 V)^{1/2} f.
\ee
Here $\sigma_3 V = \bpm W_1 & W_2 \\ \ov W_2 & W_1 \epm$ is a selfadjoint matrix.

Then
\be\begin{aligned}
\ov f &= -\sigma_3 (\sigma_3 \ov V)^{1/2} R_0(\ov \lambda) (\sigma_3 \ov V)^{1/2} \ov f \\
&= -\sigma_3 (\sigma_3 V)^{1/2} \sigma_3 R_0(\ov \lambda) \sigma_3 (\sigma_3 V)^{1/2} \sigma_3 \ov f \\
&= -\sigma_3 (\sigma_3 V)^{1/2} R_0(\ov \lambda) (\sigma_3 V)^{1/2}\sigma_3 \ov f
\end{aligned}\ee
since $R_0$ commutes with $\sigma_3$, so whenever $\lambda$ is an exceptional value so is $\ov \lambda$.

If $V$ as in (\ref{3.2}) is a real-valued matrix, then by the same methods we obtain that $-\lambda$ is an exceptional value whenever $\lambda$ is an exceptional value.
\end{proof}

We next consider the $\dot H^{1/2}$ case. In the sequel we introduce the following notation: $p-$ means $p-\epsilon$, for various small positive values of $\epsilon$.
\begin{lemma}
Assume that $V = V_1 V_2$ and that $V_1$, $V_2$ are in $\dot H^{1/2}$. Then $V_2 R_0(\lambda) V_1$ is bounded and compact from $\dot H^{1/2}$ to itself for any $\lambda$ in the lower or upper half-plane.
\end{lemma}
\begin{proof} 
Boundedness follows from the fractional Leibniz formula:
\be\begin{aligned}
\|V_1 f\|_{\dot W^{1/2, 6/5} \cap L^{3/2, 1}} &\leq C \|V_1\|_{\dot H^{1/2}} \|f\|_{\dot H^{1/2}},\\
\|R_0(\lambda) f\|_{\dot W^{1/2, 6} \cap L^{\infty}} &\leq C \|f\|_{\dot W^{1/2, 6/5} \cap L^{3/2, 1}},\\
\|V_1 f\|_{\dot H^{1/2}} &\leq C \|V_1\|_{\dot H^{1/2}} \|f\|_{\dot W^{1/2, 6} \cap L^{\infty}}.
\end{aligned}
\ee
In a similar manner we get that, for $V_1$, $V_2 \in \dot H^{1/2+\epsilon_0} \cap \dot W^{1/2+\epsilon_0, 2-}$, $V_2 R_0(\lambda) V_1$ is bounded on $\dot H^{1/2+\epsilon_0}$.

Indeed, multiplication by $V_1$ takes $\dot H^{1/2+\epsilon_0}$ to $\dot W^{1/2+\epsilon_0, 6/5} \cap \dot W^{1/2+\epsilon_0, 6/5-} \subset L^{3/2-} \cap L^{3/2+}$. Then, the resolvent takes this to $\dot W^{1/2+\epsilon_0, 6} \cap L^{\infty}$. Finally, $V_2$ takes $\dot W^{1/2+\epsilon_0, 6} \cap L^{\infty}$ to $\dot H^{1/2+\epsilon_0}$.

Compactness follows by interpolation, see Persson \cite{persson}. Given a mapping between two pairs of spaces, which is compact between one of the pairs, under certain approximation assumptions it is also compact between all intermediate spaces obtained by interpolation.

Take $V_1$, $V_2 \in \dot H^{1/2+\epsilon_0} \cap \dot W^{1/2+\epsilon_0, 2-}$. The mapping is compact from $L^2$ to itself, bounded from $\dot H^{1/2+\epsilon_0}$ to itself, and therefore also compact on $\dot H^{1/2}$. By approximation we obtain compactness for $V_1$, $V_2 \in \dot H^{1/2}$.
\end{proof}

In the sequel, in order to avoid any endpoint-related complications, when proving $\dot H^{1/2}$ estimates we assume that $V_1$, $V_2 \in \dot W^{1/2, 2+\epsilon} \cap \dot W^{1/2, 2-\epsilon}$.

In the same manner as for $L^2$, we define the $\dot H^{1/2}$ exceptional set as the set of points $\lambda$ where
\be
I + V_2 R_0(\lambda) V_1
\ee
is not invertible from $\dot H^{1/2}$ to itself. However, under reasonable assumptions they coincide:
\begin{lemma}
The $L^2$ exceptional set and the $\dot H^{1/2}$ exceptional set are the same, assuming that $V_1$, $V_2 \in \dot W^{1/2, 2+\epsilon} \cap \dot W^{1/2, 2-\epsilon}$.
\end{lemma}
\begin{proof}
Consider an exceptional value $\lambda$, for which there exists a nonzero $f \in \dot H^{1/2}$ such that
\be
f = - V_2 R_0(\lambda) V_1 f.
\ee
Starting with $f \in \dot H^{1/2}$ on the right-hand side, we retrieve that $f$ is in $L^{3-\epsilon}$ on the left-hand side. Iterating, we gain $\epsilon$ at each step and in the end arrive at $f \in L^{3, 2} \cap L^1 \subset L^2$. Thus, the $\dot H^{1/2}$ exceptional set is a subset of the $L^2$ exceptional set.

Conversely, assume $\lambda$ is in the $L^2$ exceptional set, meaning that there exists $f \in L^2$ such that
\be
f = - V_2 R_0(\lambda) V_1 f.
\ee
Then $g = R_0(\lambda) V_1 f$ is in $L^6$ and therefore $f=V_2 g$ itself is in $L^{2+\epsilon}$. By iterating and gaining $\epsilon$ at each step, we eventually obtain that $f$ is in $L^2 \cap L^{\infty}$. One final iteration also yields a half-derivative.
\end{proof}

\section{The time evolution and projections}\lb{sect_3.5}
Before making any further claims about the perturbed Hamiltonian, we prove the following basic lemma, which endows the evolution $e^{it \mc H}$ with a precise meaning.
\begin{lemma}\lb{lemma_24}
Assume $V \in L^{\infty}$ and the Hamiltonian is described by (\ref{eq_3.2}) or (\ref{3.2}). Then the equation
\be
i \partial_t Z + \mc H Z = F,\ Z(0) \text{ given},
\lb{3.58}\ee
admits a weak solution $Z$ for $Z(0) \in L^2$, $F \in L^{\infty}_t L^2_x$ and
\be
\|Z(t)\|_2 \leq C e^{t \|V\|_{\infty}} \|Z(0)\|_2 + \int_0^t e^{(t-s) \|V\|_{\infty}} \|F(s)\|_2 \dd s.
\lb{3.59}\ee
\end{lemma}
\begin{proof} We introduce an auxiliary variable and write
\be
i \partial_t Z + \mc H_0 Z = F - V Z_1,\ Z(0) \text{ given}.
\ee
Over a sufficiently small time interval $[T, T+\epsilon]$, whose size $\epsilon$ only depends on $\|V\|_{\infty}$, the map that associates $Z$ to some given $Z_1$ is a contraction, in a sufficiently large ball in $L^{\infty}_t L^2_x$. The fixed point of this contraction mapping is then a solution to (\ref{3.58}).

This shows that the equation is locally solvable and, by bootstrapping, since the length of the interval is independent of the size of $F$ and of the initial data $Z(T)$, we obtain an exponentially growing global solution. The bound (\ref{3.59}) follows by Gronwall's inequality.
\end{proof}

A nonempty exceptional set $\mc E$ destroys Strichartz estimates, so we need a means to project it away. In preparation, we define an algebra generated by $\mc H$, then spectral projections within it.

\begin{lemma}\lb{lemma_28} Consider $V \in L^{3/2, \infty}$, either scalar as in (\ref{eq_3.2}) or matrix-valued as in (\ref{3.2}), and $\chi \in L^{\infty}(\set R)$, such that, for some $\epsilon>0$, $\chi(\lambda) = 0$ on $\{\lambda \mid d(\lambda, \mc E) \leq \epsilon\}$. Then
\be
\chi(\mc H) =  \frac i {2\pi} \int_{\set R} \chi(\lambda) (R_V(\lambda+i0) - R_V(\lambda-i0)) \dd \lambda
\ee
is a bounded operator from $L^2$ to itself, of norm at most $C_{\epsilon} \|\chi\|_{\infty}$.

$\chi(\mc H)$ is defined in the weak sense that for any $f$, $g \in L^2$, the function under the integral
\be
\langle \chi(\mc H) f, g \rangle =  \frac i {2\pi} \int_{\set R} \chi(\lambda) \langle (R_V(\lambda+i0) - R_V(\lambda-i0)) f, g \rangle \dd \lambda
\ee
is absolutely integrable. Furthermore, these operators commute with the evolution and with one another:
\be
e^{it \mc H} \chi(\mc H) = \chi(\mc H) e^{it \mc H} = (e^{it\lambda} \chi)(\mc H),\ \chi_1(\mc H) \chi_2(\mc H) = (\chi_1 \chi_2)(\mc H).
\ee
In particular, $e^{it \mc H} \chi(\mc H)$ is a uniformly $L^2$-bounded family of operators.
\end{lemma}
\begin{proof}
Let $V = V_1 V_2$, with $V_1$, $V_2 \in L^{3, \infty}$. We expand the resolvent as follows:
\be\begin{aligned}
R_V(\lambda) &= R_0(\lambda) - R_0(\lambda) V_1 (I + V_2 R_0(\lambda) V_1)^{-1} V_2 R_0(\lambda).
\end{aligned}\ee
Outside the exceptional set, which is symmetric with respect to the real axis, $(I + V_2 R_0(\lambda\pm i0) V_1)^{-1}$ is uniformly bounded on compact sets as an operator from $L^2$ to itself, so it suffices to note that $|V|^{1/2}$ is an $\mc H_0$-smooth operator, meaning
\be
\int_{\set R} \||V|^{1/2} R_0 (\lambda\pm i0) f\|_2^2 \dd \lambda \leq C \|f\|_2^2.
\ee
Also, $\langle (R_0(\lambda+i0) - R_0(\lambda-i0)) f, g \rangle$ is absolutely integrable, though each of its two terms may not be, since
\be\begin{aligned}\lb{3.100}
&\int_{-\infty}^{\infty} |\langle (R_0(\lambda+i0) - R_0(\lambda-i0)) f, g \rangle| \dd \lambda \\
&\leq \int_{-\infty}^{\infty} \frac 1 {2i} (\langle (R_0(\lambda+i0) - R_0(\lambda-i0)) f, f \rangle + \langle (R_0(\lambda+i0) - R_0(\lambda-i0)) g, g \rangle) \dd \lambda \\
&= \frac 1 2 (\|f\|_2^2 + \|g\|_2^2).
\end{aligned}\ee

Following this, we actually obtain an explicit constant $C_{\epsilon}$ such that
\be
\|\chi(\mc H)\|_{2 \to 2} \leq C_{\epsilon} \|\chi\|_{\infty},
\ee
namely
\be
C_{\epsilon} = 1 + C\|V\|_{L^{3/2, \infty}} \cdot \sup_{d(\lambda, \mc E) \geq \epsilon}\|(I + V_2 R_0(\lambda\pm i0) V_1)^{-1})\|_{2 \to 2}.
\ee

Further note that, for any $\lambda$ not in the spectrum of $\mc H$,
\be
\mc H R_V(\lambda) = R_V(\lambda) \mc H = \lambda R_V(\lambda)
\ee
(in particular, $R_V(\lambda) f$ belongs to $\Dom(\mc H)$). Then, for sufficiently small $\epsilon$, due to the need to avoid exceptional values,
\be\begin{aligned}
&\mc H  \frac i {2\pi} \int_{-R}^R \chi(\lambda) (R_V(\lambda+i \epsilon) - R_V(\lambda-i\epsilon)) \dd \lambda = \\
&=  \frac i {2\pi} \int_{-R}^R \chi(\lambda) \big((\lambda+i\epsilon)R_V(\lambda+i \epsilon) - (\lambda-i\epsilon)R_V(\lambda-i\epsilon)\big) \dd \lambda.
\end{aligned}\ee
Letting $\epsilon$ go to zero, we have convergence in the same weak sense as above. Thus we get that, for every bounded $\chi$ compactly supported away from $\mc E$,
\be
\mc H \chi(\mc H) = \chi(\mc H) \mc H = (\lambda \chi(\lambda))(\mc H).
\ee
It immediately follows that, also for bounded $\chi$ of compact support away from $\mc E$, the evolution can simply be expressed as
\be
e^{it \mc H} \chi(\mc H) = \chi(\mc H) e^{it \mc H} = (e^{it\lambda} \chi)(\mc H).
\ee
Take $f \in \Dom(\mc H)$, that is $f \in L^2$, $\mc H f = g \in L^2$. Then for any $\chi \in L^{\infty}(\set R)$ supported away from $\mc E$ one has that
\be\begin{aligned}
\rho (\chi_{(-\infty, -\rho] \cup [\rho, \infty)} \chi)(\mc H) f &= (\rho \chi_{(-\infty, -\rho] \cup [\rho, \infty)} \chi)(\mc H) f \\
&= ((\rho/\lambda) \chi_{(-\infty, -\rho] \cup [\rho, \infty)} \chi)(\mc H) H f \\
&= ((\rho/\lambda) \chi_{(-\infty, -\rho] \cup [\rho, \infty)} \chi)(\mc H) g
\end{aligned}\ee
and is uniformly bounded in $L^2$. Therefore $(\chi_{[-\rho, \rho]} \chi)(\mc H)f$ converges to $\chi(\mc H)f$ in the $L^2$ norm as $\rho$ goes to infinity. Since $\Dom(\mc H)$ is dense in $L^2$, it follows that $(\chi_{[-\rho, \rho]} \chi)(\mc H)$ converges strongly, but not necessarily in norm, to $\chi(\mc H)$. Therefore the identity
\be
e^{it \mc H} \chi(\mc H) = \chi(\mc H) e^{it \mc H} = (e^{it\lambda} \chi)(\mc H)
\ee
extends to $\chi$ without compact support.

In order to take this strong limit, we first approximate $V$ by $L^{\infty}$ potentials, for which the evolution is $L^2$-bounded. These approximations may move the boundary of the exceptional set by some small amount, but, since $\chi$ is supported some positive distance away from the exceptional set, this brings no prejudice.

In particular, this shows that $e^{it \mc H} \chi(\mc H)$ is a uniformly bounded family of $L^2$ operators.

Next, for $\chi$ of compact support
\be
\mc H \chi(\mc H) = \chi(\mc H) \mc H = (\lambda \chi)(\mc H)
\ee
implies
\be
R_V(\lambda_0) \chi(\mc H) = \chi(\mc H) R_V(\lambda_0) = \Big(\frac{\chi(\lambda)}{\lambda-\lambda_0}\Big)(\mc H)
\ee
for any $\lambda_0$ not in the spectrum. By passing to the strong limit, we remove the condition that $\chi$ should have compact support. Integrating, we obtain that for $\chi_1$ of compact support and any $\chi_2$
\be\begin{aligned}
 \frac i {2\pi} \int_{\set R} \chi_1(\lambda) (R_V(\lambda+i\epsilon) - R_V(\lambda-i\epsilon)) \chi_2(\mc H) \dd \lambda = \\
 = \bigg(\int_{\set R} \int_{\set R} \frac 1 {\pi} \frac{\epsilon \chi_1(\lambda_1) \dd \lambda_1}{(\lambda_1-\lambda_2)^2 + \epsilon^2} \chi_2(\lambda_2) \dd \lambda_2\bigg)(\mc H),
\end{aligned}\ee
where we recognize the Poisson kernel. Letting $\epsilon$ go to zero we get that
\be
\chi_1(\lambda) \chi_2(\lambda) = (\chi_1 \chi_2)(\lambda).
\ee
Then, by passing to the strong limit, we remove the condition that $\chi_1$ should have compact support.
\end{proof}

This leads to a natural definition of $L^2$ spectral projections for $\mc H$, on the real axis. Namely, for a set $A \subset \set R$ at a positive distance away from the exceptional set $\mc E$, we can define $P_A = \chi_A(\mc H)$. This is a projection in the sense that $P_A^2 = P_A$, but it need not be self-adjoint.

For a clearer picture, we are also interested in the following fact. Under more stringent assumptions, this appeared in Schlag \cite{schlag}.
\begin{lemma}\lb{lemma_31} Consider $V \in L^{3/2, \infty}$ that can be approximated by bounded functions. Then for sufficiently large $y$
\be
\langle f, g \rangle = \frac i {2\pi} \int_{\set R} \langle (R_V(\lambda + iy) - R_V(\lambda - iy)) f, g \rangle \dd \lambda
\ee
and the integral is absolutely convergent.

Furthermore, for every $\epsilon>0$
\be
\langle f, g \rangle = \frac i {2\pi} \int_{\set R} \langle (R_V(\lambda + i\epsilon) - R_V(\lambda - i\epsilon)) f, g \rangle \dd \lambda + \sum_{k=1}^n P^0_{\zeta_k}
\ee
where $P^0_{\zeta_k}$ are projections corresponding to the finitely many exceptional points $\zeta_k$ of imaginary part greater than $\epsilon$.

As a consequence, the rate of growth of $\|e^{it \mc H}\|_{2 \to 2}$ can be given an asymptotic expansion in terms of exponentials.
\end{lemma}
\begin{proof} Assume at first that $V \in L^{\infty}$ and take $y > \|V\|_{\infty}^{1/2}$. Then
\be
(I + V_2 R_0(\lambda \pm iy) V_1)^{-1}
\ee
must be invertible. Indeed, $V_1$ and $V_2$ are bounded $L^2$ operators of norm at most $\|V\|^{1/2}_{\infty}$ and
\be
\|R_0(\lambda \pm iy)\|_{2 \to 2} \leq \frac 1 {4 \pi} \int_{\set R^3} \frac{e^{-y|x|}}{|x|} \dd x = 1/y^2.
\ee
Therefore one can expand $(I + V_2 R_0(\lambda \pm iy) V_1)^{-1}$ into a series. Thus
\be\begin{aligned}\lb{3.99}
R_V(\lambda\pm iy) &= R_0(\lambda \pm iy) - R_0(\lambda \pm iy) V R_0(\lambda \pm iy) + \\
&+ R_0(\lambda\pm iy) V_1 (I + V_2 R_0(\lambda\pm iy) V_1)^{-1} V_2 R_0(\lambda \pm iy)
\end{aligned}\ee
is a bounded $L^2$ operator.

By Lemma \ref{lemma_24}
\be
\chi_{t \geq 0} \langle e^{it \mc H} e^{-yt} f, g \rangle
\ee
is an exponentially decaying function and its Fourier transform is
\be
\int_0^{\infty} \langle e^{-(y+i\lambda)t} e^{it \mc H} f, g \rangle \dd y = -i\langle R_V(\lambda-iy) (I -\lim_{t \to \infty} e^{-i(\lambda-iy)t} e^{it \mc H}) f, g \rangle.
\ee
Combining this with the analogous result for the positive side, we see that
\be
(\langle e^{it \mc H} e^{-yt} f, g \rangle)^{\wedge} = i \langle (R_V(\lambda+iy)-R_V(\lambda-iy)) f, g \rangle.
\ee
However, assuming that $V \in L^{3/2, \infty}$, now the right-hand side is absolutely integrable due to (\ref{3.99}), (\ref{3.100}), and to smoothing estimates, though neither half has to be absolutely integrable on its own. Therefore, by the Fourier inversion formula,
\be\lb{3.103}
\frac i {2\pi} \int_{\set R} \chi(\lambda) \langle (R_V(\lambda + iy) - R_V(\lambda - iy)) f, g \rangle \dd \lambda = \langle f, g \rangle.
\ee
We can also shift $y$, provided we do not encounter any eigenvalue.

From the beginning, we have assumed that $V \in L^{\infty}$. Now consider the general case $V \in L^{3/2, \infty}$ and a sequence of approximations by bounded potentials $V^n = V_1^n V_2 ^n \in L^{\infty}$, such that $\|V^n - V\|_{L^{3/2, \infty}} \to 0$ as $n \to \infty$. Let $\mc E$ be the exceptional set of $V$. On the set $\{\lambda \mid d(\lambda, \mc E) \geq \epsilon\}$, the norm
\be
\|(I + V_2^n R_0(\lambda) V_1^n)^{-1}\|_{L^2 \to L^2}
\ee
is uniformly bounded for large $n$. For some sufficiently high $n$, then, $\mc E(V_n) \subset \{\lambda \mid d(\lambda, \mc E) < \epsilon\}$. If
\be
y_0 = \sup \{|\Im \lambda| \mid \lambda \in \mc E\},
\ee
then for any $y > y_0$ and sufficiently large $n$
\be\lb{3.104}
\frac i {2\pi} \int_{\set R} \chi(\lambda) \langle (R_{V^n}(\lambda + iy) - R_{V^n}(\lambda - iy)) f, g \rangle \dd \lambda = \langle f, g \rangle.
\ee
Both for $V$ and for $V^n$ the integrals (\ref{3.103}) and (\ref{3.104}) converge absolutely and as $n \to \infty$ (\ref{3.104}) converges to (\ref{3.103}). To see this, subtract the corresponding versions of (\ref{3.99}) from one another and evaluate.

Consider $V \in L^{3/2, \infty}$ approximable by bounded functions. We then shift this contour arbitrarily close to the real line, leaving behind contour integrals around elements of the exceptional set. It is easy to prove, in the same manner as previously, that
\be\begin{aligned}
\mc H \frac i {2\pi} \int_{-R}^R \langle (R_V(\lambda+iy)-R_V(\lambda-iy)) f, g \rangle \dd \lambda = \\
= \frac i {2\pi} \int_{-R}^R \langle ((\lambda+iy)R_V(\lambda+iy)-(\lambda-iy)R_V(\lambda-iy)) f, g \rangle \dd \lambda.
\end{aligned}\ee
Then
\be\begin{aligned}
e^{it \mc H} \frac i {2\pi} \int_{-R}^R \langle (R_V(\lambda+iy)-R_V(\lambda-iy)) f, g \rangle \dd \lambda = \\
= \frac i {2\pi} \int_{-R}^R \langle (e^{it(\lambda+iy)}R_V(\lambda+iy)-e^{it(\lambda-iy)}R_V(\lambda-iy)) f, g \rangle \dd \lambda.
\end{aligned}\ee
We let $R$ tend to infinity and obtain that the same holds for the whole contour. Indeed, over any horizontal line $\lambda \pm iy$, $y \ne 0$, that does not intersect the exceptional set $\mc E$, the integral
\be
\int_{\set R} \big|\langle R_V(\lambda\pm iy) - R_0(\lambda \pm iy) f, g \rangle\big| \dd \lambda \leq C \|f\|_2 \|g\|_2
\ee
converges absolutely due to the resolvent identity (\ref{3.99}) and smoothing estimates, while the remaining part
\be
\frac i {2\pi} \int_{-R}^R \langle (\chi(\lambda+iy) R_0(\lambda+iy)- \chi(\lambda-iy) R_0(\lambda-iy)) f, g \rangle \dd \lambda
\ee
converges to $\langle \chi(\mc H_0) f, g \rangle$ for both $\chi=1$ and $\chi(\lambda)=e^{it\lambda}$.

Therefore $\|e^{it \mc H}\|_{2 \to 2}$ grows no faster than $e^{|t| (y_0 +\epsilon)}$. In particular, if $\mc E$ is situated on the real line, $\|e^{it \mc H}\|_{2 \to 2}$ grows more slowly than any exponential.
\end{proof}

To conclude, we prove the following lemma about Riesz projections.
\begin{lemma} Assume $V \in L^{3/2, \infty}$. To each element $\zeta$ of the exceptional set of $\mc H$ not in $\sigma(\mc H_0)$ there corresponds a family of operators
\be
P^k_{\zeta} = \frac 1 {2\pi i} \int_{|z-\zeta| = \epsilon} R_V(z) (z-\zeta)^k \dd z.
\lb{3.111}
\ee
They have finite rank, $P^k_{\zeta} = 0$ for all $k \geq n$, for some $n$, $P^0_{\zeta}=(P^0_{\zeta})^2$, and more generally $(P_{\zeta}^k)(P_{\zeta}^{\ell}) = P_{\zeta}^{k+\ell}$.

Moreover, $P^k_{\zeta}$ are bounded from $L^{\infty} + L^{6/5, 2}$ to $L^1 \cap L^{6, 2}$ and $P^k_{\zeta}$, $\zeta$ depend continuously on $V$.

In the $\dot H^{1/2}$ setting, assume that $V=V_1 V_2$, with $V_1$, $V_2 \in \dot W^{1/2, 2-\epsilon} \cap \dot W^{1/2, 2+\epsilon}$. Then $P^k_{\zeta}$ takes $\dot W^{1/2, 6}$ to $\dot W^{1/2, 6/5}$ and $P^k_{\zeta}$, $\zeta$ depend continuously on $V$.
\end{lemma}
The $\dot H^{1/2}$ estimate is decidedly suboptimal, but appropriate for our purposes.
\begin{proof} If $\mc H_0 + V - \zeta$ is invertible, then $\zeta$ is not in the exceptional set and vice-versa, as a consequence of (\ref{eq_3.64}) and (\ref{eq_3.65}).

We form the contour integral, following Schlag \cite{schlag} and Reed--Simon \cite{reesim1},
\be
P^k_{\zeta} = \frac 1 {2\pi i} \int_{|z-\zeta| = \epsilon} R_V(z) (z-\zeta)^k \dd z.
\ee
This integral is independent of $\epsilon$ if $\epsilon$ is sufficiently small and $P^k_{\zeta} = 0$ for $k \geq n$. Using the Cauchy integral, it immediately follows that $(P_{\zeta}^k)(P_{\zeta}^{\ell}) = P_{\zeta}^{k+\ell}$. Furthermore,
\be\lb{3.135}
\mc H P^0_{\zeta} = P^1_{\zeta} + \zeta P^0_{\zeta}.
\ee
Then $e^{it \mc H} P^0_{\zeta}$ can be described explicitly as
\be\begin{aligned}
e^{it \mc H} P^0_{\zeta} &= e^{it\zeta} P^0_{\zeta} + \frac {e^{it\zeta} - 1}{\zeta} P^1_{\zeta} + \frac {e^{it\zeta} - 1 - it\zeta}{\zeta^2} P^2_{\zeta} + \ldots \\
&+ \frac{e^{it\zeta}-1-\ldots-(it\zeta)^{n-2}/{(n-2)!}}{\zeta^{n-1}} P^{n-1}_{\zeta}.
\end{aligned}\ee
Its rate of growth is at most $e^{|\Im \zeta| |t|} |t|^{n-1}$.


It is a consequence of Fredholm's theorem that the range of $P^0_{\zeta}$ is finite dimensional; from (\ref{3.111}) it follows that $\Ran(P^0_{\zeta}) \subset L^2 \cap L^{6, 2}$. Also, $\Ran(P^0_{\zeta})$ is the generalized eigenspace of $\mc H - \zeta$, meaning
\be
\Ran(P^0_{\zeta}) = \bigcup_{k \geq 0} \Ker((\mc H - \zeta)^k).
\ee
One inclusion follows from (\ref{3.135}) and the fact that $P^k_{\zeta} = 0$ for $k \geq n$. The other inclusion is a consequence of the fact that, if $(\mc H - \zeta) f = 0$, then $R_V(z) f = (\zeta - z)^{-1} f$ and, using the definition (\ref{3.111}), $P^0_{\zeta} f = f$. For higher values of $k$ we proceed by induction.

Furthermore, $\Ran(P^0_{\zeta})$ consists of functions in $\langle \dl \rangle^{-2} L^{6/5, 2}$. If $f$ is a generalized eigenfunction, meaning $f \in L^2 \cap L^{6, 2}$ and $(\mc H - \zeta)^n f = 0$, then
\be
(\mc H_0 + V) f = \zeta f + g,
\ee
where $g$ is also a generalized eigenfunction. Assuming by induction that $g \in \langle \dl \rangle^{-2} L^{6/5, 2}$ (or is zero, to begin with), the same follows for $f$. Furthermore, if $g \in L^1$ we can infer the same about $f$.

The range of $(P^0_{\zeta})^*$ is the generalized eigenspace of $\mc H^* - \ov \zeta$, which means that it is also finite-dimensional and spanned by functions in $L^1 \cap \langle \dl \rangle^{-2} L^{6/5, 2}$.

Thus, each such projection is bounded from $L^{6/5, 2} + L^{\infty}$ to $L^1 \cap L^{6, 2}$.

Regarding the $\dot H^{1/2}$ case, assume that $V = V_1 V_2$, with $V_1$, $V_2 \in \dot W^{1/2, 2-\epsilon} \cap \dot W^{1/2, 2+\epsilon}$. We obtain that each generalized eigenfunction is in $L^1 \cap L^{\infty} \cap \dot W^{1/2, 6/5}$. This implies the boundedness of $P^k_{\zeta}$ from $\dot W^{1/2, 6}$ to $\dot W^{1/2, 6/5}$.

We next prove continuity with respect to $V$. Consider two potentials and their decomposition
\be
V = V_1 V_2,\ \tilde V = \tilde V_1 \tilde V_2.
\ee
We express the resolvent in a manner similar to (\ref{3.99}):
\be
R_V(z) = R_0(z) - R_0(z) V R_0(z) + R_0(z) V_1 (I + V_2 R_0(z) V_1)^{-1} V_2 R_0(z).
\ee
It follows that
\be\begin{aligned}
P^k_{\zeta}(\tilde V) - P^k_{\zeta}(V) &= \frac 1 {2\pi i} \int_{|z-\zeta|=\epsilon} (z-\zeta)^k \big(R_0(z) V_1 (I + V_2 R_0(z) V_1)^{-1} V_2 R_0(z) - \\
&-R_0(z) \tilde V_1 (I + \tilde V_2 R_0(z) \tilde V_1)^{-1} \tilde V_2 R_0(z) \big) \dd z.
\end{aligned}\ee
Keeping the integration contour fixed ($\zeta$ and $\epsilon$ fixed), for each $z$ on the contour one obtains by the resolvent identity that
\be
\|(I + V_2 R_0(z) V_1)^{-1} - (I + \tilde V_2 R_0(z) \tilde V_1)^{-1}\|_{2 \to 2} \leq C (\|\tilde V_1 - V_1\|_{L^{3, \infty}} + \|\tilde V_2 - V_2\|_{L^{3, \infty}})
\ee
as long as one has
\be
\|\tilde V_1 - V_1\|_{L^{3, \infty}} \|V_2\|_{L^{3, \infty}} + \|\tilde V_2 - V_2\|_{L^{3, \infty}} \|\tilde V_1\|_{L^{3, \infty}} \leq \frac C {\|(I + V_2 R_0(z) V_1)^{-1}\|_{2 \to 2}}.
\ee
Consequently, whenever the right-hand side is sufficiently small,
\be\begin{aligned}\lb{3.145}
&\|P^k_{\zeta}(\tilde V) - P^k_{\zeta}(V)\|_{\mc L(L^2, L^2)} \leq C (\|\tilde V_1 - V_1\|_{L^{3, \infty}} + \|\tilde V_2 - V_2\|_{L^{3, \infty}}).
\end{aligned}\ee
Likewise, at the $\dot H^{1/2}$ level
\be\begin{aligned}
&\|P^k_{\zeta}(\tilde V) - P^k_{\zeta}(V)\|_{\mc L(\dot H^{1/2}, \dot H^{1/2})} \\
&\leq C (\|\tilde V_1 - V_1\|_{\dot W^{1/2, 2-\epsilon} \cap \dot W^{1/2, 2+\epsilon}} + \|\tilde V_2 - V_2\|_{\dot W^{1/2, 2-\epsilon} \cap \dot W^{1/2, 2+\epsilon}}).
\end{aligned}\ee
%

Following Simon \cite{opuc}, p.\ 819, we quote the following theorem concerning projections, due to Kato:
\begin{theorem}\lb{simonkato}
Consider two projections (not necessarily orthogonal) $P$ and $Q$ on a Hilbert space, such that $\|P - Q\| < 1$. Then there exists an invertible operator \be
U(P, Q) = QP + (I-Q)(I-P)
\ee
such that
\be
U(P, Q) P U(P, Q)^{-1} = Q,
\ee
$U(P, P) = I$, and $U(P, Q)$ is analytic in $P$ and $Q$.
\end{theorem}
In particular, if $P$ and $Q$ have finite rank, it must be the same.

Therefore, for a fixed value of $\epsilon$ in the definition (\ref{3.111}), whenever $\|\tilde V_1 - V_1\|_{L^{3, \infty}} + \|\tilde V_2 - V_2\|_{L^{3, \infty}}$ is sufficiently small, $P^0_{\zeta}(\tilde V)$ is an operator of the same rank as $P^0_{\zeta}(V)$ on $L^2$ (and likewise in the $\dot H^{1/2}$ setting). This shows the continuity of $\zeta$, as a multivalued function of $V$.
\end{proof}

Obtaining estimates in this generality can be quite complicated, especially if the exceptional set has infinitely many elements or if there are embedded exceptional values in $\sigma(\mc H_0)$. Therefore, in the sequel we make a simplifying assumption. Namely, we assume that there are no exceptional values embedded in $\sigma(\mc H_0)$. By Fredholm's analytic theorem, this implies that there are finitely many exceptional values overall.

Then we can define $P_c$, the projection on the continuous spectrum. Let $\chi$ be a function that equals one on $\sigma(\mc H_0)$ and zero outside a small neighborhood.
\begin{definition}
Assume that there are no exceptional values of $\mc H$ embedded in the spectrum of $\mc H_0$. Let $P_c$ be defined as $P_c=\chi(\mc H)$, for $\chi(\mc H)$ given by Lemma \ref{lemma_28}. 
\end{definition}
$P_c$ commutes with $\mc H$, other operators in the algebra of the form $\chi_1(\mc H)$, and with $e^{it \mc H}$, as a direct consequence of the definition and of Lemma \ref{lemma_28}.

From the previous Lemma \ref{lemma_31}, it follows that
\be
P_c = \chi(\mc H) = \frac i {2\pi} \int_{\sigma(\mc H_0)} \big(R_V(\lambda+i0) - R_V(\lambda-i0)\big) \dd \lambda
\ee
and $I-P_c$ is a finite range operator,
\be
I-P_c = P_p = \sum_{k=1}^n P_{\zeta_k}^0.
\ee
$P_c$ is bounded on $L^2$, but, since each projection $P_{\zeta_k}^0$ is bounded from $L^{\infty} + L^{6/5, 2}$ to $L^{6, 2} \cap L^1$, the same holds for $P_p = I-P_c$.

Therefore $P_c$ is bounded on $L^{6/5, q}$, $q \leq 2$, and on $L^{6, q}$, $q\geq 2$, as well as on intermediate spaces.





The following technical lemma is important in the proof of Strichartz estimates.
\begin{lemma}\lb{lemma32}
Consider $V \in L^{3/2, 1}(\set R^3)$ and $\mc H = \mc H_0 + V$ as in (\ref{eq_3.2}) or (\ref{3.2}) such that $\mc H$ has no exceptional values embedded in $\sigma(\mc H_0)$. Additionally, assume that zero is not an exceptional value for $\mc H$, either. Then there exists a decomposition
\be
V - P_p \mc H = F_1(V) F_2(V),
\ee
where $P_p = I - P_c$, such that $F_1(V) \in \mc L(L^2, L^{6/5, 1})$, $F_2(V) \in \mc L(L^{6, \infty}, L^2)$, and the Fourier transform of $I-iT_{F_2(V), F_1(V)}$ is invertible in the lower half-plane up to the boundary, where
\be
(T_{F_2(V), F_1(V)}F)(t) = \int_{-\infty}^t F_2(V) e^{i(t-s)\mc H_0} F_1(V) F(s) \dd s.
\ee
Furthermore, $F_1(V)$ and $F_2(V)^*$ can be approximated in the $\mc L(L^2, L^{6/5, 1})$ norm by operators that are bounded from $L^2$ to $\langle x \rangle^{-N} L^2$, for any fixed $N$.

Assuming that $V = V_1 V_2$ with $V_1$, $V_2 \in \dot W^{1/2, 2-\epsilon} \cap W^{1/2, 2+\epsilon}$, one can take
\be
F_1(V) \in \mc L(\dot H^{1/2}, |\dl|^{-1/2} L^{6/5, 1}),\ F_2(V) \in \mc L(|\dl|^{-1/2} L^{6, \infty}, \dot H^{1/2}),
\ee
such that $F_1(V)$ is approximable by operators in $\mc L(\dot H^{1/2},\langle x \rangle^{-N} L^2)$ and $F_2(V)$ is approximable by operators in $\mc L(\langle x \rangle^{-N} L^2, \dot H^{1/2})$.
\end{lemma}
More generally, a similar result holds for $V$ belonging to the $L^{3/2, \infty}$ closure of the set of bounded compactly supported functions.

\begin{proof}
Consider a potential $V$ such that $\mc H$ has no exceptional values embedded in $\sigma(\mc H_0)$. Being a finite-rank operator, $P_p$ has the form
\be
P_p f = \sum_{k=1}^n \langle f, f_k \rangle g_k 
\ee
where $f_k$ and $g_k$ belong to $L^{6, 2} \cap L^1$. 
It follows that $A$ can be expressed as
\be
A f = \sum_{j, k=1}^n a_{jk} \langle f, f_k \rangle f_j.
\ee
Take the standard polar decomposition of $P_p$, with respect to $L^2$:
\be
P_p = U A,
\ee
where $A = (P_p^* P_p)^{1/2} \geq 0$ is a nonnegative $L^2$ operator of finite rank. $A$ maps the span of $f_k$ to itself and $U$ is a partial $L^2$ isometry defined on the range of $A$. $U$ maps the span of $f_k$ to the span of $g_k$ and can be extended by zero on the orthogonal complement:
\be\begin{aligned}
U f &= \sum_{j, k=1}^n u_{jk} \langle f, f_k \rangle g_j.
\end{aligned}\ee
From these explicit forms we see that both $U$ and $A$ are bounded operators from $L^{6/5, 2} + L^{\infty}$ to $L^{6, 2} \cap L^1$.


Also let $V = V_1 V_2$, where $V_2 \geq 0$ is a nonnegative operator on $L^2$, meaning $\langle f, V_2 f \rangle \geq 0$ for every $f \in \Dom(V_2)$, and $V_1$, $V_2 \in L^{3, 2}$.

Then, define the bounded $L^2$ operators $G_1 = V_2/(V_2 + A)$ and $G_2 = (\mc H P_p)/(V_2 + A)$, initially on $\Ran(V_2 + A)$, by setting
\be
G_1((V_2+A)f) = V_2 f,\ G_2((V_2+A)f) = \mc H P_p f
\ee
and extending them by continuity to $\ov{\Ran(V_2 + A)}$. On the orthogonal complement
\be
\Ran(V_2 + A)^{\perp}=\{f \mid P_p f =0,\ V f = 0\}
\ee
we simply set $G_1=G_2=0$. We then make the construction
\be\begin{aligned}
V - P_p \mc H &= F_1(V) F_2(V), \\
F_2(V) &= V_2 + A, \\
F_1(V) &= V_1 G_1 - G_2.
\end{aligned}\ee
$F_1(V)$ and $F_2(V)$ are bounded from $L^2$ to $L^{6/5, 1}$ and from $L^{6, \infty}$ to $L^2$ respectively.

We next prove that $F_1(V)$ and $F_2(V)$ can be approximated by operators in better classes, as claimed. Firstly, consider a family of smooth compactly supported functions $\chi_n$ such that $0 \leq \chi_n \leq 1$ and $\chi_n \to 1$ as $n \to \infty$. Let
\be\begin{aligned}\lb{eq_3.177}
V_1^n &= \chi_n V_1, &V_2^n &= \chi_n V_2,\\
A^n &= \chi_n A \chi_n,\\
G_1^n &= V_2^n / (V_2^n + A^n), &G_2^n &= (\chi_n \mc H P_p \chi_n) /(V_2^n + A^n),\\
F_1^n &= V_1^n G_1^n - G_2^n, &F_2^n &= V_2^n + A^n.
\end{aligned}\ee

It is plain that $F_1^n$ and $(F_2^n)^*$ take $L^2$ to $\langle x \rangle^{-N} L^2$. We also prove that they also approximate $F_1(V)$ and $F_2(V)$ in $\mc L(L^2, L^{6/5, 1})$. To begin with, as $n \to \infty$
\be\begin{aligned}
\|V_j^n - V_j\|_{\mc L(L^2, L^{6/5, 1})} \to 0,\\
\|A^n - A\|_{\mc L(L^2, L^{6/5, 1})} \to 0.
\end{aligned}\ee
The second statement is due to the fact that, for finite rank operators, weak and norm convergence are equivalent. Note that $I - G_1^n$ is also a finite rank operator, so the weak convergence of $I - G_1^n$ to $I-G_1$ implies norm convergence also, in the $\mc L(L^2, L^2)$ norm. This reasoning also proves that $G_2^n \to G_2$ in $\mc L(L^2, L^{6/5, 1})$.

The same construction works in the $\dot H^{1/2}$ case, except that we have to consider the polar decomposition corresponding to $\dot H^{1/2}$ instead of $L^2$. Keeping exactly the same notations, let $P_p = UA$ be the polar decomposition corresponding to $\dot H^{1/2}$ and let $V = V_1 V_2$, where $V_1$, $V_2 \in \dot W^{1/2, 2-\epsilon} \cap W^{1/2, 2+\epsilon}$ and $V_2 \geq 0$ as a $\dot H^{1/2}$ operator. The last property is ensured by (\ref{3.5}) or (\ref{3.6}) and is essentially a consequence of the fact that $V_2$ is real-valued (or self-adjoint). Setting
\be
G_1((V_2+A)f) = V_2 f,\ G_2((V_2+A)f) = \mc H P_p f
\ee
results in bounded $\dot H^{1/2}$ operators. Then take
\be\begin{aligned}
V - P_p \mc H &= F_1(V) F_2(V), \\
F_2(V) &= V_2 + A, \\
F_1(V) &= V_1 G_1 - G_2.
\end{aligned}\ee
Upon inspection, we immediately see that $F_1(V)$ and $F_2(V)$ are bounded from $|\dl|^{-1/2} L^{6, \infty}$ to $\dot H^{1/2}$ and from $\dot H^{1/2}$ to $|\dl|^{-1/2} L^{6/5, 1}$ respectively.

Again, we use cutoffs to approximate $F_1(V)$ and $F_2(V)$, exactly as in (\ref{eq_3.177}).



Finally, we show that the Fourier transform $I-i\widehat{T}_{F_2(V), F_1(V)}(\lambda)$ is nonvanishing. Let the Fourier transform vanish at $\lambda$ in the lower half-plane, meaning that there exists a nonzero $f \in L^2$ such that
\be
f = -F_2(V) R_0(\lambda) F_1(V) f.
\ee
Then $g=R_0(\lambda) F_1(V) f$ belongs to $|\dl|^{-2} L^{6/5, 2}$ and
\be
g = -R_0(\lambda) (V - P_p \mc H) g.
\ee
Applying $-\Delta+\lambda$ to both sides, we obtain that
\be
\mc H P_c g = \lambda g.
\ee
Let $P_c g = \tilde g \in |\dl|^{-2} L^{6/5, 2}$. One has $\mc H \tilde g = \lambda \tilde g$ and thus, for $\tilde f = V_2 \tilde g$ (see \ref{3.2}),
\be
f = -V_2 R_0(\lambda) V_1 f.
\ee
Therefore $\lambda$ is an exceptional value for $\mc H$, according to Definition \ref{def_7}. Since, by assumption, there are no exceptional values in $\sigma(\mc H_0)$, it follows that $\tilde g$ is an eigenvector of $\mc H$ and therefore $P_p \tilde g = \tilde g$. This, together with $P_c \tilde g = \tilde g$, implies that $\tilde g = 0$. Therefore $\lambda g = 0$, so either $g=0$, in which case either $f=0$ or $\lambda=0$.

However, we specifically assumed that $\lambda=0$ is not an exceptional value for $\mc H$. It follows that the Fourier transform of $I - i T_{F_2(V), F_1(V)}$ is invertible in the lower half-plane.
\end{proof}

\section{Strichartz estimates}
In the sequel, by $e^{it \mc H}$ we designate the solution of the homogenous equation, which is well-defined at all times as a bounded $L^2$ operator, for $V \in L^{\infty}$. We extend its meaning gradually to other cases.

To begin with, we prove the following straightforward result, which makes no use of spectral projections.



\begin{theorem}\lb{theorem_25} Consider a potential $V \in L^{3/2, 1}$, scalar or matrix-valued as in (\ref{3.2}), real or complex-valued, whose exceptional set $\mc E$ is empty in the lower half-plane up to the boundary. Then Strichartz estimates hold for $\mc H = \mc H_0+V$: for the equation
\be
i \partial_t Z + \mc H Z = F,\ Z(0) \text{ given}
\ee
one has, for $t \geq 0$,
\be
\|Z\|_{L^{\infty}_t L^2_x \cap L^2_t L^{6, 2}_x} \leq C \Big(\|Z(0)\|_2 + \|F\|_{L^1_t L^2_x + L^2_t L^{6/5, 2}_x}\Big)
\ee
and
\be
\|Z\|_{L^1_t L^{6, \infty}_x} \leq C (\|Z(0)\|_{L^{6, 1}} + \|F\|_{L^1_t L^{6/5, 1}_x}).
\ee
Suppose that $V = V_1 V_2$ and $V_1$, $V_2 \in \dot W^{1/2, 2-\epsilon}\cap \dot W^{1/2, 2+\epsilon}$. Then, assuming there are no $\dot H^{1/2}$ exceptional values in the lower half-plane or on its boundary, for $t \geq 0$
\be
\|Z\|_{L^{\infty}_t \dot H^{1/2}_x \cap L^2_t \dot W^{1/2, 6}_x} \leq C \Big(\|Z(0)\|_{\dot H^{1/2}} + \|F\|_{L^1_t \dot H^{1/2}_x + L^2_t \dot W^{1/2, 6/5}_x}\Big).
\ee
\end{theorem}

\begin{proof}
Take a multiplicative decomposition $V = V_1 V_2$ with $V_1$, $V_2 \in L^{3, 2}$ and let 
\be
(T_{V_2, V_1} F)(t) = \int_{-\infty}^t V_2 e^{i(t-s)\mc H_0} V_1 F(s) \dd s.
\ee
If the exceptional set is empty, then
\be
\|(I + V_2 R_0(\lambda) V_1)^{-1}\|_{2 \to 2}
\ee
is bounded in the lower half-plane up to the boundary. Therefore, one can approximate $V$ in the $L^{3/2, 1}$ norm by $L^{\infty}$ potentials whose exceptional sets are still empty (in the lower half-plane). If the conclusion stands for each of these approximations, then we pass to the limit and it also holds for $V$ itself. Likewise, we can approximate $F$ by functions located in $L^1_t L^2_x$. For these approximations, Lemma \ref{lemma_24} provides the evolution with a natural meaning in $L^{\infty}_{t, loc} L^2_x$.

Firstly, we need to establish that $T_{V_1, V_2} \in K$, where $K$ is the Wiener algebra of Definition \ref{def_k}. By Proposition \ref{prop_21} and Minkowski's inequality, we see that $T$ given by
\be
TF(t) = \int_{-\infty}^t e^{i(t-s) \mc H_0} F(s) \dd s
\ee
takes $L^1_t L^{6, 1}_x$ to $L^1_t L^{6/5, \infty}_x$.
Therefore
\be
T_{V_2, V_1} F(t) = \int_{-\infty}^t V_2 e^{i(t-s) \mc H_0} V_1 F(s) \dd s
\ee
takes $L^1_t L^2_x$ to itself and, for the Hilbert space $L^2_x$, belongs to the algebra $K$ of Theorem \ref{thm7}. Following Lemma \ref{lemma_22}, the Fourier transform of $T_{V_2, V_1}$ is given by
\be
\widehat T_{V_2, V_1}(\lambda) = iV_2 R_0(\lambda) V_1.
\ee
Having an empty exceptional set $\mc E$ simply ensures that $I-i\widehat T_{V_2, V_1}(\lambda)$ is invertible in the lower half-plane.

Furthermore, for $V_1$, $V_2 \in L^{3+\epsilon}$, $T_{V_2, V_1}$ decays at infinity like $|t|^{-1-\epsilon}$ in norm. By approximating in the $L^{3, 2}$ norm, it follows that $T_{V_2, V_1}$ belongs to the subalgebra $D \subset K$ of kernels that decay at infinity for any $V_1$, $V_2 \in L^{3, 2}$.

Likewise, assume $V_1$, $V_2$ are smooth of compact support. Then for each $t$
\be
\|(V_2 e^{-i(t+\epsilon)\Delta} V_1 - V_2 e^{-it\Delta} V_1)f\|_2 \leq C \epsilon^{1/4} \|e^{-it\Delta} V_1 f\|_{\dot H^{1/2}_{loc}}
\ee
and therefore
\be\begin{aligned}
&\int_{-t}^t \|(V_2 e^{i(t+\epsilon)\mc H_0} V_1 - V_2 e^{it\mc H_0} V_1)f\|_2 \dd t \\
&\leq C \epsilon^{1/4} t^{1/2} \Big(\int_{-t}^t \|e^{-it\Delta} V_1 f\|_{\dot H^{1/2}_{loc}}^2 \dd t\Big)^{1/2}\\
&\leq C \epsilon^{1/4} t^{1/2} \|f\|_2.
\lb{3.52}
\end{aligned}\ee
Since $T_{V_2, V_1}$ decays at infinity, (\ref{3.52}) implies that $T_{V_2, V_1}$ is equicontinuous. Again, by approximating we find that the same holds for any $V_1$, $V_2 \in L^{3, 2}$.

Therefore $I-iT_{V_2, V_1}$ satisfies all the hypotheses of Theorem \ref{thm7}, with respect to the Hilbert space $L^2$ and the algebra $K$ of Definition \ref{def_k}. Given that $\mc E = \varnothing$, one can invert and $(I-iT_{V_2, V_1})^{-1}$ is an upper triangular operator in $K$ by Lemma ~\ref{lemma8}.

Let $T_{1, V_1}$ and $T_{V_2, 1}$ be respectively given by
\be
(T_{1, V_1}F)(t) = \int_{-\infty}^t e^{i(t-s)\mc H_0} V_1 F(s) \dd s
\ee
and
\be
(T_{V_2, 1} F)(t) = \int_{-\infty}^t V_2 e^{i(t-s)\mc H_0} F(s) \dd s.
\ee
Now write the evolution as
\be
Z(t) = e^{it \mc H_0} Z(0) + \int_0^t e^{i(t-s)\mc H_0} F(s) \dd s + T_{1, V_1} (I-iT_{V_2,  V_1})^{-1} T_{V_2, 1} (F + \delta_{t=0} Z(0)).
\lb{3.51}
\ee
By Strichartz estimates, it follows that $T_{V_2, 1}$ takes right-hand side terms in the Strichartz space and $L^2$ initial data, $F + \delta_{t=0} Z(0)$, to $L^2_t L^{2, 1}_x \subset L^2_{t, x}$. The convolution kernel $(I-T_{V_2, V_1})^{-1}$ then takes $L^2_{t, x}$ into itself again, while the last operator $T_{1, V_1}$ takes $L^2_{t, x} \subset L^2_t L^{2, \infty}_x$ into the dual Strichartz space.

With no further effort we obtain that
\be
\|Z\|_{L^1_t L^{6, \infty}_x} \leq C (\|Z(0)\|_{L^{6, 1}} + \|F\|_{L^1_t L^{6/5, 1}_x}).
\ee
This results from (\ref{3.51}) by the estimates we already have, in particular by Proposition \ref{prop_21}. Indeed, one has that $T_{V_2, 1}$ takes $L^1_t L^{6/5, 1}$ to $L^1_t L^2_x$, then by Theorem \ref{thm7} $(I-iT_{V_2,  V_1})^{-1}$ takes $L^1_t L^2_x$ into itself, and finally $T_{1, V_1}$ takes $L^1_t L^2_x$ to $L^1_t L^{6, \infty}$.

Concerning the $\dot H^{1/2}$ case, the conditions of Theorem \ref{thm7} are again satisfied. Indeed, assume that $V_1$ and $V_2$ are smooth and of compact support. One then has decay at infinity of the kernel, meaning
\be
\|V_2 e^{-it\Delta} V_1\|_{\dot H^{1/2} \to \dot H^{1/2}} \leq C \|e^{-it\Delta}\|_{\dot W^{1/2, 6/5-\epsilon} \to W^{1/2, 6+\epsilon}} \leq C |t|^{-1-\epsilon}
\ee
and
\be
\lim_{t \to \infty} \int_t^{\infty} |s|^{-1-\epsilon} \dd s= 0.
\ee
Furthermore,
\be\begin{aligned}
&\int_{-t}^t \|(V_2 e^{i(t+\epsilon)\mc H_0} V_1 - V_2 e^{it\mc H_0} V_1)f\|_{\dot H^{1/2}} \dd t \\
&\leq C \epsilon^{1/4} t^{1/2} \Big(\int_{-t}^t \|e^{-it\Delta} V_1 f\|_{\dot H^{1}_{loc}}^2 \dd t\Big)^{1/2}\\
&\leq C \epsilon^{1/4} t^{1/2} \|f\|_{\dot H^{1/2}}.
\end{aligned}\ee
Finally, we have to check that the kernel is integrable in time. Since $\mc H_0$ commutes with differentiation,
\be
\int_0^{\infty} \|e^{it \mc H_0} f\|_{|\dl|^{-1/2} L^{6, \infty}} \dd t \leq C \|f\|_{|\dl|^{-1/2} L^{6/5, 1}}.
\ee
In other words,
\be
TF(t) = \int_{-\infty}^t e^{i(t-s) \mc H_0} F(s) \dd s
\ee
takes $L^1_t |\dl|^{-1/2} L^{6/5, 1}_x$ into $L^1_t |\dl|^{-1/2} L^{6, \infty}_x$. Multiplying by $V_2$ on the left and $V_1$ on the right produces $T_{V_2, V_1}$, a $\dot H^{1/2}$ bounded operator.

Then, following approximation of $V_1$ and $V_2$ in $\dot W^{1/2, 2-\epsilon} \cap \dot W^{1/2, 2+\epsilon}$ by smooth, compactly supported functions, we obtain that Theorem \ref{thm7} applies. Since the $\dot H^{1/2}$ exceptional set is empty in the lower half-plane, we obtain that $I-i T_{V_1, V_2}$ is invertible from $L^2_t \dot H^{1/2}_x$ to itself.

By (\ref{3.51}) we then obtain Strichartz estimates for $\dot H^{1/2}$. Indeed, $T_{V_2, 1}$ takes $L^2_t \dot W^{1/2, 6/5}_x + L^1_t \dot H^{1/2}_x$ to $L^2_t \dot H^{1/2}_x$. Then, an application of $(I-i T_{V_1, V_2})^{-1}$ preserves $L^2_t \dot H^{1/2}_x$. Finally, $T_{1, V_1}$ takes $L^2_t \dot H^{1/2}_x$ to $L^2_t \dot W^{1/2, 6}_x$.
\end{proof}

\begin{observation} An easy proof can be given, if $V \in L^1 \cap L^{\infty}$ has an empty exceptional set, to $\langle t \rangle^{-3/2}$ decay estimates from $L^1 \cap L^2$ to $L^2 + L^{\infty}$. We can conduct the proof in the Beurling algebra of kernels (see Remark \ref{beurling})
\be
Be_{\infty, -3/2 s} = \{A \mid \|A(t)\|_{Be} = \sup_t (\langle t \rangle^{3/2} \|A(t)\|_{L^2 \to L^2}) < \infty\}. 
\ee
The proof follows that of Theorem \ref{theorem_25}, but is more straightforward.
\end{observation}

\begin{theorem}\lb{theorem_26}
Let $Z$ be a solution of the linear Schr\"{o}dinger equation
\be
i \partial_t Z + \mc H Z = F,\ Z(0) \text{ given}.
\ee
Assume that $\mc H = \mc H_0 + V$, $V \in L^{3/2, 1}$ is scalar as in (\ref{eq_3.2}) or is matrix-valued as in (\ref{3.2}), and that no exceptional values of $\mc H$ are contained in $\sigma(\mc H_0)$. Then Strichartz estimates hold: for $t \geq 0$
\be\lb{3.137}
\|P_c Z\|_{L^{\infty}_t L^2_x \cap L^2_t L^{6, 2}_x} \leq C \Big(\|Z(0)\|_2 + \|F\|_{L^1_t L^2_x + L^2_t L^{6/5, 2}_x}\Big)
\ee
and
\be\lb{3.138}
\|P_c Z\|_{L^1_t L^{6, \infty}_x} \leq C \Big(\|Z(0)\|_{6/5, 1} + \|F\|_{L^1_t L^{6/5, 1}_x}\Big).
\ee
Assume that $V = V_1 V_2$ as in (\ref{3.5}) or (\ref{3.6}) and $V_1$, $V_2 \in \dot W^{1/2, 2-\epsilon} \cap \dot W^{1/2, 2+\epsilon}$. Then
\be
\|P_c Z\|_{L^{\infty}_t \dot H^{1/2}_x \cap L^2_t \dot W^{1/2, 6}_x} \leq C \Big(\|Z(0)\|_{\dot H^{1/2}} + \|F\|_{L^1_t \dot H^{1/2}_x + L^2_t \dot W^{1/2, 6/5}_x}\Big).
\ee
\end{theorem}
\begin{proof}



Fix $\lambda_0 \in \set R$ and let
\be
\tilde Z = e^{it\lambda_0} P_c Z,\ \tilde F = e^{it\lambda_0} P_c F,\ \tilde {\mc H} = \mc H + \lambda_0,\ \tilde {\mc H}_0 = \mc H_0 + \lambda_0.
\ee
The equation for $\tilde Z$ becomes
\be
i \partial_t \tilde Z + \tilde{\mc H} P_c \tilde Z = \tilde F,\ \tilde Z(0) = P_c Z(0) \text{ given}.
\ee
Write, initially for $F \in L^1_t L^2_x$,
\be
\tilde Z(t) = \int_{-\infty}^t e^{i(t-s) (\mc H + \lambda_0) P_c} (-i\tilde F(s) + \delta_{s=0} \tilde Z(0)) \dd s.
\ee
It follows by our estimates of Lemma \ref{lemma_28} that
\be
\|\tilde Z\|_{L^{\infty}_t L^2_x} \leq C (\|\tilde Z(0)\|_2 + \|\tilde F\|_{L^1_t L^2_x}) \leq C (\|\tilde Z(0)\|_2 + \|\tilde F\|_{L^1_t L^2_x}).
\ee
To begin with, assume the potential $V$ is bounded of compact support. 
Consider the following arrangement:
\be
P_c \tilde{\mc H} = \tilde{\mc H}_0 + V - P_p \tilde{\mc H}.
\ee
The spectrum and exceptional set of $\tilde{\mc H}$ are the same as those of $\mc H$, only translated by $\lambda_0$. For a proper choice of $\lambda_0$, we can always ensure that $\tilde{\mc H}$ does not have an exceptional value at zero.

For $V \in L^{3/2, 1}$, Lemma \ref{lemma32} provides a multiplicative decomposition
\be
V - P_p \tilde{\mc H} = F_1(V) F_2(V),
\ee
such that both $F_1(V)$ and $F_2^*(V)$ are bounded operators from $L^2$ to $L^{6/5, 1}$.
Let
\be
(T_{F_2(V), F_1(V)} F)(t) = \int_{-\infty}^t F_2(V) e^{i(t-s)\tilde{\mc H}_0} F_1(V) F(s) \dd s
\ee
and
\be
(T_{F_2(V), 1} F)(t) = \int_{-\infty}^t F_2(V) e^{i(t-s)\tilde{\mc H}_0} F(s) \dd s.
\ee

Duhamel's formula implies that
\be
F_2(V) \tilde Z = T_{F_2(V), 1} (-i\tilde F + \delta_{t=0} \tilde Z(0)) + i T_{F_2(V), F_1(V)} (F_2(V) \tilde Z)
\ee
and, further,
%
\be\lb{3.143}
\big(I - i T_{F_2(V), F_1(V)}\big) F_2(V) \tilde Z = T_{F_2(V), 1} (-i\tilde F + \delta_{t=0} \tilde Z(0)).
\ee

By Lemma \ref{lemma32}, the Fourier transform of $I - i T_{F_2(V), F_1(V)}$ is invertible in the lower half-plane. One can check the other requirements of Theorem \ref{thm7} in the same manner as in the proof of Theorem \ref{theorem_25}. Then we apply Theorem \ref{thm7} and conclude that
\be
\big(I - i T_{F_2(V), F_1(V)}\big)^{-1}
\ee
exists as a bounded operator from $L^p_t L^2_x$ to itself for $1 \leq p \leq \infty$.

Thus we can invert in (\ref{3.143}), obtaining
\be\lb{3.134}
F_2(V) P_c Z = \big(I - i T_{F_2(V), F_1(V)}\big)^{-1} T_{F_2(V), 1} (-i\tilde F + \delta_{t=0} \tilde Z(0)).
\ee
Strichartz estimates for the free Schr\"{o}dinger evolution imply, then, that
\be
\|V_2 P_c Z\|_{L^2_t L^2_x} \leq C \Big(\|Z(0)\|_2 + \|F\|_{L^1_t L^2_x + L^2_t L^{6/5, 2}_x}\Big).
\ee


Applying the Duhamel formula again results in the Strichartz estimate (\ref{3.137}). Proposition \ref{prop_21} combined with (\ref{3.143}) leads instead to a proof of (\ref{3.138}).
%
\end{proof}


\section{The time-dependent case}\lb{sec3.7}
We turn to time-dependent equations. Many of the previous observations, concerning the Fourier transform and Wiener's Theorem, apply here as well. However, the Fourier transform of a kernel $T(t, s)$ which is not invariant under time translation is no longer a multiplier $\widehat T(\lambda):H \to H$ for each $\lambda$; it is a family of non-local operators instead. Such a generalization was studied by Howland \cite{how}.

We shall not follow this direction in the current work. Instead, we only look at small perturbations of time-independent operators.

We prove a concrete result in the case of interest, while at the same time placing it within a more general framework.

Given parameters $A(t)$ and $v(t) = (v_1(t), v_2(t), v_3(t))$, consider the family of isometries
\be\lb{3.155}
U(t) =  e^{\textstyle\int_0^t (2v(s) \dl + i A(s) \sigma_3) \dd s}.
\ee
The rate of change of $U(t)$ is then controlled by the norm
\be
\|A(t)\|_{L^{\infty}_t} + \|v(t)\|_{L^{\infty}_t}.
\ee
The linearized Schr\"{o}dinger equation under study has the form
\be\lb{3.157}
i \partial_t R(t) + (\mc H_0 + U(t)^{-1} V U(t)) R(t) = F(t),\ R(0) \text{ given}.
\ee
Observe that the Hamiltonian at time $t$ is $U(t)^{-1} (\mc H_0 + V) U(t)$, since $\mc H$ and $U$ commute, i.e.\ it is $\mc H_0 + V = \mc H$ conjugated by $U(t)$.

Let $Z(t) = U(t) R(t)$. We rewrite the equation in the new variable $Z$, obtaining
\be\lb{3.158}
i \partial_t Z(t) - i \partial_t U(t) U(t)^{-1} Z(t) + \mc H_0 Z(t) + V Z(t) = U(t) F(t),\ Z(0) = R(0).
\ee

In our study of the time-dependent equations (\ref{3.157}) and (\ref{3.158}), we employ the following algebras, whose definitions parallel Definition \ref{def_k}.
\begin{definition}\lb{def_kt} The algebras $\tilde K_{1su}$ and $\tilde K$ are
\be\begin{aligned}
\tilde K_{1su} = \Big\{T(t, s) \mid &\sup_s \int_{\set R} \|T(t, s)\|_{2 \to 2} \dd t \leq C, \\
&\sup_t \int_{\set R} \|T(t, s)\|_{2 \to 2} \dd s \leq C\Big\} \oplus \{z I \mid z \in \set C\}.\\
\tilde K = \{T(t, s) \mid &\sup_s \|T(t, s) f\|_{M_t L^2_x} \leq C \|f\|_2,\\
&\sup_s \|T^*(t, s) f\|_{M_t L^2_x} \dd s \leq C\|f\|_{2}\}.
\end{aligned}\ee
\end{definition}
This expands the definition of $K_{1su}$, from Remark \ref{beurling}, by allowing operators that are not trans\-la\-tion-invariant.

Either one of the two conditions alone would be enough to define an algebra, but having both satisfied makes $\tilde K_{1su}$ a $C^*$ algebra. The addition of the unit turns it into a unital algebra.

$\tilde K$ also contains operators that are not invariant under translation, is naturally endowed with a unit, and is a $C^*$ algebra.

In addition to these two algebras, when proving $\dot H^{1/2}$ estimates we employ the algebras $|\dl|^{-1/2} \tilde K_{1su} |\dl|^{1/2}$ and $|\dl|^{-1/2} \tilde K |\dl|^{1/2}$, for which a specific notation will not be assigned.

In the subsequent lemma we encapsulate all the properties of $U$ that we actually use in our study of (\ref{3.157}) and (\ref{3.158}).

\begin{lemma}\lb{lem_32} Let $U(t)$ be defined by (\ref{3.155}).
\begin{enumerate}
\item $U(t)$ is a strongly continuous family of $\dot W^{s, p}$ isometries, for $s \in \set R$, $1 \leq p < \infty$.
\item For every $t$ and $s \geq 0$, $U(t)$ and $U(s)$ commute with $\mc H_0$ and each other.
\item For some $N$
\be
\|\langle x \rangle^{-N} (U(t, s) e^{i(t-s) \mc H_0} - e^{i(t-s) \mc H_0}) \langle x \rangle^{-N}\|_{\tilde K_{1su}} \leq C
\ee
and there exists $\epsilon(N)>0$ such that
\be\lb{3.163}\begin{aligned}
\|\langle x \rangle^{-N} (U(t, s) e^{i(t-s) \mc H_0} - e^{i(t-s) \mc H_0}) \langle x \rangle^{-N}\|_{\tilde K_{1su}} \leq \\
\leq C (\|A(t)\|_{L^{\infty}_t} + \|v(t)\|_{L^{\infty}_t})^{\epsilon(N)}.
\end{aligned}\ee
\end{enumerate}
\end{lemma}
These properties can be further refined. More specifically, we may replace  $\langle x\rangle^{-N} L^2$ by $L^{6/5-\epsilon} \cap L^{6/5+\epsilon}$ and by $\langle x \rangle^{-\epsilon} (L^{6/5-\epsilon} \cap L^{6/5+\epsilon})$, respectively. However, this sharpening is unnecessary.
\begin{proof} It is easy to verify the first two properties directly. In particular, the Laplace operator $\Delta$ commutes with translations.

Concerning the last two properties, we compare
\be\begin{aligned}
T(t, s) = e^{i(t-s) \mc H_0}
\end{aligned}\ee
and
\be\begin{aligned}
\tilde T(t, s) = e^{i(t-s) \mc H_0} e^{\int_s^t (2v(\tau) \dl + i A(\tau) \sigma_3) \dd \tau}.
\end{aligned}\ee

On one hand, from dispersive estimates,
\be\lb{3.167}
\|T(t, s) - \tilde T(t, s)\|_{1 \to \infty} \leq C |t-s|^{-3/2}.
\ee
On the other hand,
\be
\|T(t, s) - \tilde T(t, s)\|_{2 \to 2} \leq C.
\ee
It follows that for $N>1$
\be
\|\langle x \rangle^{-N} (T - \tilde T) \langle x \rangle^{-N}\|_{\tilde K_{1su}} < C.
\ee
This holds with a constant independent of $A$ and $v$.

Finally, we prove the last stated property. Assume $A$ and $v$ are small. We first consider the case when there is no translation movement due to $v$ and we only have to handle oscillation, due to $A$. Denote this kernel by $\tilde T_{osc}$, that is
\be\begin{aligned}
\tilde T_{osc}(t, s) &= e^{i(t-s) \mc H_0} e^{\int_s^t (i A(\tau) \sigma_3) \dd \tau}.
\end{aligned}\ee
One has
\be
e^{ia} -1 \leq C \min(1, a)
\ee
and thus
\be\lb{3.172}
\|T(t, s) - \tilde T_{osc}(t, s)\|_{2 \to 2} \leq C \min(1, \|A\|_{\infty} |t-s|) \leq C \|A\|_{\infty}^{\epsilon} |t-s|^{\epsilon}.
\ee
Therefore, for large enough $N$, it follows from (\ref{3.167}) and from (\ref{3.172}) that
\be
\|\langle x \rangle^{-N} (T - \tilde T_{osc}) \langle x \rangle^{-N}\|_{\tilde K_{1su}} \leq C \|A\|_{\infty}^{\epsilon/3}.
\ee

Next, we consider the case when $v$ is not necessarily zero. Let $d(t) = \int_0^t v(\tau) \dd \tau$. Then
\be\begin{aligned}
&e^{-i(t-s) \Delta} e^{(\int_s^t 2v(\tau) \dd \tau) \dl} = \\
&\textstyle= \frac 1 {(-4\pi i)^{3/2}}(t-s)^{-3/2} e^{i\big(\frac{|x-y|^2}{4(t-s)} - \frac{(x-y) (d(t) - d(s))}{t-s} + \frac {(d(t)-d(s))^2}{t-s}\big)}.
\end{aligned}\ee

We treat the last factor $e^{i \frac {(d(t)-d(s))^2}{t-s}}$ in the same manner in which we treated the factors containing $A$. Consider the kernel that contains those factors together with $e^{i \frac {(d(t)-d(s))^2}{t-s}}$, leaving aside $e^{\frac{(x-y) (d(t) - d(s))}{t-s}}$:
\be\begin{aligned}
\tilde T_1(t, s) &= e^{i(t-s) \mc H_0} e^{\int_s^t i(A(\tau) \sigma_3) \dd \tau} e^{i \frac {(d(t)-d(s))^2}{t-s}}.
\end{aligned}\ee
For this kernel as well, one has
\be\begin{aligned}
\|\langle x \rangle^{-N} (T - \tilde T_1) \langle x \rangle^{-N}\|_{\tilde K_{1su}} \leq C (\|A\|_{\infty}^{\epsilon/3} + \|v\|_{\infty}^{\epsilon/3}).
\end{aligned}\ee
Considering the fact that
\be
\Big|e^{i \frac{(x-y) (d(t) - d(s))}{t-s}} -1\Big| \leq C \min(1, \|v\|_{\infty} (|x| + |y|) \leq C \|v\|_{\infty}^{\epsilon} (|x| + |y|)^{\epsilon},
\ee
it follows that for large enough $N$
\be\begin{aligned}
\|\langle x \rangle^{-N} (\tilde T(t, s) - \tilde T_1(t, s)) \langle x \rangle^{-N}\|_{2 \to 2} \leq C \|v\|_{\infty}^{\epsilon} |t-s|^{-3/2}.
\end{aligned}\ee
We also have the trivial bound
\be\begin{aligned}
\|\tilde T(t, s) - \tilde T_1(t, s)\|_{2 \to 2}\leq C.
\end{aligned}\ee
Therefore
\be
\|\langle x \rangle^{-N} (\tilde T - \tilde T_1) \langle x \rangle^{-N}\|_{\tilde K_{1su}} \leq C \|v\|_{\infty}^{\epsilon/3}.
\ee
Overall, we find that
\be
\|\langle x \rangle^{-N} (\tilde T - T) \langle x \rangle^{-N}\|_{\tilde K_{1su}} \leq C (\|A\|_{\infty}^{\epsilon/3} + \|v\|_{\infty}^{\epsilon/3}).
\ee
\end{proof}
\begin{theorem}\lb{theorem_13}
Consider equation (\ref{3.157}), for $\mc H = \mc H_0 + V$ as in (\ref{3.2}) and $V \in L^{3/2, 1}$ not necessarily real-valued:
\be
i \partial_t Z - i v(t) \dl Z + A(t) \sigma_3 Z + \mc H Z = F,\ Z(0) \text{ given},
\ee
\be
\mc H = \bpm \Delta - \mu & 0 \\ 0 & -\Delta + \mu \epm + \bpm |W_1| & W_2 \\ -\ov W_2 & |W_1| \epm,
\ee
and assume that $\|A\|_{\infty}$ and $\|v\|_{\infty}$ are sufficiently small (in a manner that depends on $V$) and there are no exceptional values of $\mc H$ embedded in $\sigma(\mc H_0)$. Then
\be
\|P_c Z\|_{L^{\infty}_t L^2_x \cap L^2_t L^{6, 2}_x} \leq C \Big(\|Z(0)\|_2 + \|F\|_{L^1_t L^2_x + L^2_t L^{6/5, 2}_x}\Big).
\ee
Under the same conditions, assume in addition that, when factoring $V = V_1 V_2$ as in (\ref{3.6}), $V_1$, $V_2 \in \dot W^{1/2, 2-\epsilon} \cap \dot W^{1/2, 2+\epsilon}$. Then
\be
\|P_c Z\|_{L^{\infty}_t \dot H^{1/2}_x \cap L^2_t \dot W^{1/2, 6}_x} \leq C \Big(\|Z(0)\|_{\dot H^{1/2}} + \|F\|_{L^1_t \dot H^{1/2}_x + L^2_t \dot W^{1/2, 6/5}_x}\Big).
\ee
\end{theorem}

\begin{proof} To begin with, we assume that $V$ is smooth of compact support.
This guarantees the existence of a solution $R$, albeit one that may grow exponentially. Once we have obtained the desired result under this restriction, we use approximation and Strichartz estimates to retrieve it in the general case as well.

As in the time-independent case, let
\be\begin{aligned}
&\tilde Z = e^{it\lambda_0} P_c Z,\ \tilde F = e^{it\lambda_0} P_c F - 2iv(t)[P_c, \dl] \tilde Z + A(t) [P_c, \sigma_3] \tilde Z,\\
&\tilde {\mc H} = \mc H + \lambda_0,\ \tilde {\mc H}_0 = \mc H_0 + \lambda_0.
\end{aligned}\ee
The equation becomes
\be\begin{aligned}
&i \partial_t \tilde Z - i v(t) \dl \tilde Z + A(t) \sigma_3 \tilde Z + \tilde {\mc H} \tilde Z = \tilde F,\ \tilde Z(0) = P_c R(0) \text{ given}.
\end{aligned}\ee
The commutation terms
\be
2iv(t)[P_c, \dl] \tilde Z,\ A(t) [P_c, \sigma_3] \tilde Z
\ee
are small in the dual Strichartz norm (for small $\|v\|_{\infty}$ and $\|A\|_{\infty}$) and thus can be controlled by Strichartz inequalities and a simple fixed point argument.

For a proper choice of $\lambda_0$, zero is not an exceptional value for $\tilde{\mc H}$. Lemma \ref{lemma32} then provides a decomposition
\be
V - P_p \tilde {\mc H} = F_1(V) F_2(V),
\ee
where $F_1(V)$ and $F_2^*(V)$ are in $\mc L(L^2, L^{6/5, 1})$ and can be approximated in this space by operators $F_1^n$ and $F_2^n$, such that $F_1^n$ and $(F_2^n)^*$ are bounded from $L^2$ to $\langle x \rangle^{-N} L^2$.


Denote
\be\lb{3.144}\begin{aligned}
\tilde T_{F_2(V), F_1(V)} F(t) &= \int_{-\infty}^t F_2(V) P_c e^{i(t-s) \tilde{\mc H}_0} U(t, s) F_1(V) F(s) \dd s, 
\end{aligned}\ee
respectively
\be\begin{aligned}
T_{F_2(V), 1} F(t) &= \int_{-\infty}^t F_2(V) P_c e^{i(t-s) \tilde {\mc H}_0} F(s) \dd s.
\end{aligned}\ee
By Duhamel's formula,
\be\lb{3.186}\begin{aligned}
F_2(V) \tilde Z(t) &= i \tilde T_{F_2(V), F_1(V)} V_2 \tilde Z(t) + T_{F_2(V), 1} (-iU(s) \tilde F(s) + \delta_{s=0} \tilde Z(0)).
\end{aligned}\ee
We compare the time-dependent kernel $\tilde T_{F_2(V), F_1(V)}$ with the time-independent one
\be\begin{aligned}
T_{F_2(V), F_1(V)} F(t) &= \int_{-\infty}^t F_2(V) P_c e^{i(t-s) \tilde {\mc H}_0} F_1(V) F(s) \dd s. 
\end{aligned}\ee
For each $n$, by Lemma \ref{lem_32} we obtain that
\be
\lim_{\substack{\|A\|_{\infty} \to 0 \\ \|v\|_{\infty} \to 0}} \|T_{F_2^n, F_1^n} - \tilde T_{F_2^n, F_1^n}\|_{\tilde K_{1su}} = 0.
\ee
At the same time, since $F_2^n$ and $F_1^n$ are approximations of $F_2(V)$ and $F_1(V)$,
\be\begin{aligned}
&\lim_{n \to \infty} \|T_{F_2^n, F_1^n} - T_{F_2(V), F_1(V)}\|_{\tilde K} = 0, \\
&\lim_{n \to \infty} \|\tilde T_{F_2^n, F_1^n} - \tilde T_{F_2(V), F_1(V)}\|_{\tilde K} = 0.
\end{aligned}\ee
Therefore,
\be\lb{3.189}
\lim_{\substack{\|A\|_{\infty} \to 0 \\ \|v\|_{\infty} \to 0}} \|T_{F_2(V), F_1(V)} - \tilde T_{F_2(V), F_1(V)}\|_{\tilde K_{1su}} = 0.
\ee
As a consequence of Lemma \ref{lemma32}, the operator $I - i T_{F_1(V), F_2(V)}$ is invertible in the translation-invariant algebra $K$ and, thus, also in the time-dependent $\tilde K$ algebra of Definition \ref{def_kt}. By (\ref{3.189}), it follows that when $\|A\|_{\infty}$ and $\|v\|_{\infty}$ are small enough $I - i \tilde T_{F_1(V), F_2(V)}$ is also invertible in $\tilde K$.

This immediately implies the desired Strichartz estimates.

In regard to the $\dot H^{1/2}$ case, assume that $V_1$ and $V_2$ belong to $\dot W^{1/2, 2-\epsilon} \cap \dot W^{1/2, 2+\epsilon}$. Then we have the decomposition
\be
V - P_p \tilde {\mc H} = F_1(V) F_2(V),
\ee
where
\be
F_1(V) \in \mc L(\dot H^{1/2}, |\dl|^{-1/2} L^{6/5, 1}),\ F_2(V) \in \mc L(|\dl|^{-1/2} L^{6, \infty}, \dot H^{1/2})
\ee
respectively. $F_1(V)$ can be approximated in this space by operators
\be
F_1^n \in \mc L(\dot H^{1/2}, \langle x \rangle^{-N} L^2)
\ee
and $F_2(V)$ can be approximated by $F_2^n \in \mc L(\langle x \rangle^N L^2, \dot H^{1/2})$, for some large $N$.

Proceeding in the same manner, with the help of Lemma \ref{lemma32} it follows that
\be
\lim_{\substack{\|A\|_{\infty} \to 0 \\ \|v\|_{\infty} \to 0}} \|T_{F_2(V), F_1(V)} - \tilde T_{F_2(V), F_1(V)}\|_{|\dl|^{-1/2}\tilde K|\dl|^{1/2}} = 0.
\ee
Therefore we can invert $\tilde T_{F_2(V), F_1(V)}$ in the operator algebra $|\dl|^{-1/2}\tilde K|\dl|^{1/2}$ and the proof of $\dot H^{1/2}$ Strichartz estimates proceeds in exactly the same manner.
\end{proof}

\appendix
\chapter{Spaces of functions}\lb{spaces} Our computations take place mostly in Lebesgue and Sobolev spaces of functions defined on $\set R^{3+1}$. This  corresponds to three spatial dimensions and one extra dimension that accounts for time.

To begin with, we only consider measurable complex-valued functions or ones that take values within a finite-dimensional Banach space. In general, dealing with more general Banach space-valued functions, one must distinguish between weak and strong measurability.

We denote the Lebesgue norm of $f$ by by $\|f\|_p$. The Sobolev spaces of integral order $W^{n, p}$ are then defined by
\be
\|f\|_{W^{n, p}} = \bigg(\sum_{|\alpha| \leq n} \|\partial^{\alpha} f\|_p^p\bigg)^{1/p}
\ee
for $1 \leq p < \infty$ and
\be
\|f\|_{W^{n, \infty}} = \sup_{|\alpha| \leq n} \|\partial^{\alpha} f\|_{\infty}.
\ee
when $p=\infty$. In addition, we consider Sobolev spaces of fractional order, both homogenous and inhomogenous:
\be
\|f\|_{W^{s, p}} = \|\langle \dl \rangle^s f\|_p, \text{ respectively } \|f\|_{\dot W^{s, p}} = \||\dl| f\|_p.
\ee
Here $\langle \dl \rangle^s$ and $|\dl|^s$ denote Fourier multipliers --- multiplication on the Fourier side by $\langle \xi \rangle^s = (1+|\xi|^2)^{s/2}$ and $|\xi|^{s}$ respectively.

When $p=2$, the alternate notation $H^s = W^{s, 2}$ or $\dot H^s = \dot W^{s, 2}$ is customary.

Strichartz estimates may involve mixed space-time norms of the form
\be
\|f\|_{L^p_t \dot W^{s, q}_x} = \Bigg(\int_{-\infty}^{\infty} \|f(x, t)\|_{\dot W^{s, q}_x}^p \dd t\Bigg)^{1/p}.
\ee

Our computations naturally lead to Lorentz spaces. Given a measurable function $f$ defined on a measure space $(X, \mu)$, consider its distribution function
\be
m(\sigma, f) = \mu(\{x \mid |f(x)|>\sigma\}).
\ee
\begin{definition}
A function $f$ belongs to the Lorentz space $L^{p, q}$ if its decreasing rearrangement
\be\lb{A6}
f^*(t) = \inf\{\sigma \mid m(\sigma, f) \leq t\}
\ee
fulfills
\be
\|f\|_{L^{p, q}} = \bigg(\int_0^{\infty} (t^{1/p} f^*(t))^q \frac {\dd t} t \bigg)^{1/q} < \infty
\lb{lorentz}
\ee
or, respectively,
\be
\|f\|_{L^{p, \infty}} = \sup_{0 \leq t < \infty} t^{1/p} f^*(t) < \infty
\ee
when $q=\infty$.
\end{definition}
We list several important properties of Lorentz spaces.
\begin{lemma}\begin{enumerate}\item $L^{p, p} = L^p$ and $L^{p, \infty}$ is weak-$L^p$.
\item The dual of $L^{p, q}$ is $L^{p', q'}$, where $1/p + 1/p' = 1$, $1/q + 1/q' = 1$, just as for Lebesgue spaces.
\item If $q_1 \leq q_2$, then $L^{p, q_1} \subset L^{p, q_2}$.
\item Except when $q=\infty$, the set of bounded compactly supported functions is dense in $L^{p, q}$.
\end{enumerate}\end{lemma}
For a more complete enumeration, see \cite{bergh}.

In this somewhat abstract setting, it is important to note that $f \in L^{p, q}$ if and only if $|f| \in L^{p, q}$, so the study of the general case can be reduced to that of real-valued functions.

Other spaces of functions are introduced as needed in the course of the proof. Especially important are the space of Borel measures of finite mass, denoted by $M$, and the discrete Lebesgue spaces $\dot \ell^q = L^q(\set Z)$, $\ell^q = L^q(\set N)$ with respect to the counting measure.

We conclude by proving a lemma concerning the atomic decomposition of Lorentz spaces. In preparation for that, we give the following definition:
\begin{definition}
The function $a$ is an $L^p$ atom, $1 \leq p < \infty$, if $a$ is measurable, bounded, its support has finite measure, and $a$ is $L^p$ normalized:
\be
\esssup_x |a(x)| < \infty,\ \mu(\supp a)< \infty,\ (\esssup_x |a(x)|)^p \cdot \mu(\supp a) = 1.
\ee
\end{definition}
Again, note that $a$ is an atom if and only if $|a|$ is one.
\begin{lemma}[Atomic decomposition of $L^{p, q}$]\lb{lemma_30} Consider $(\set R^n, \mu)$, with $\mu$ the Lebesgue measure. A function $f$ belongs to $L^{p, q}(\set R^n)$, $1 \leq p$, $q < \infty$, if and only if it can be written as a linear combination of $L^p$ atoms
\be
f = \sum_{k \in \set Z} \alpha_k a_k,
\ee
where the atoms $a_k$ have disjoint supports of size $2^k$ and $(\alpha_k)_k \in \ell^q$. Furthermore, $\|f\|_{L^{p, q}} \sim \|\alpha_k\|_{\ell^q}$ and the sum converges unconditionally (regardless of the order of summation) in the $L^{p, q}$ norm.
\end{lemma}
This is just a discretized version of the definition of Lorentz spaces (\ref{lorentz}).
\begin{proof}
In one direction, assume that $f \in L^{p, q}$ and let $f_k=f^*(2^k)$, for $f^*$ as in (\ref{A6}). By the definition of the distribution function, there exist disjoint sets $A_k$, of measure $2^k$, on which $f_k \geq |f(x)| \geq f_{k+1}$. Furthermore, since the distribution function is decreasing, by (\ref{lorentz}) one has that
\be
\bigg(\sum_k (2^{(k+1)/p} f_k)^q \bigg)^{1/q} \geq \|f\|_{L^{p, q}} \geq \bigg(\sum_k (2^{k/p} f_{k+1})^q \bigg)^{1/q}
\ee
or, equivalently,
\be
2^{1/p} \bigg(\sum_k (2^{k/p} f_k)^q \bigg)^{1/q} \geq \|f\|_{L^{p, q}} \geq 2^{-1/p} \bigg(\sum_k (2^{k/p} f_k)^q \bigg)^{1/q}.
\ee
Then let
\be\begin{aligned}
\alpha_k &= 2^{k/p} \esssup\{|f(x)| \mid x \in A_k\}, \\
a_k &= (\chi_{A_k} f)/\alpha_k;
\end{aligned}\ee
when $\alpha_k = 0$, we pick $a_k$ in an arbitrary fashion, say $a_k = 2^{-k/p} \chi_{A_k}$. Note that $a_k$ is an atom for each $k$ and
\be
2^{k/p} f_k \geq \alpha_k \geq 2^{k/p} f_{k+1}.
\ee
Therefore
\be
2^{2/p} (\sum_k |\alpha_k|^q)^{1/q} \geq \|f\|_{L^{p, q}} \geq 2^{-1/p} (\sum_k |\alpha_k|^q)^{1/q},
\ee
so $\|f\|_{L^{p, q}} \sim \|\alpha_k\|_{\ell^q}$. The unconditional convergence of $\sum_k |\alpha_k|^q$, then, implies the unconditional convergence of $\sum \alpha_k a_k$.

In order to establish the converse, consider any sum
\be
f = \sum_{k \in \set Z} \alpha_k a_k,
\ee
where $a_k$ are $L^p$ atoms of disjoint supports of size $2^k$ and only finitely many of the coefficients $\alpha_k$ are nonzero. One needs to show that
\be\lb{A.17}
\|f\|_{L^{p, q}} \leq C \|\alpha_k\|_{\ell^q},
\ee
with a constant that does not depend on the number of terms.

Observe that each atom $a_k$ has $\esssup |a_k(x)| = 2^{-k/p}$. Since the measures of supports of atoms $a_k$, for $k < k_0$, add up to at most $2^{k_0}$, it follows by the definition of the distribution function that
\be
f^*(2^{k_0}) \leq \sup_{k \geq k_0} |\alpha_k a_k(x)| = \sup_{k \geq k_0}2^{-k/p} |\alpha_k|.
\ee
Then, the integral that appears in the definition (\ref{lorentz}) can be bounded by
\be\begin{aligned}
\int_0^{\infty} (t^{1/p} f^*(t))^q \frac {\dd t} t &\leq \sum_{k_0 \in \set Z} \big(2^{k_0/p} (\sup_{k \geq k_0}2^{-k/p} |\alpha_k|)\big)^q \\
&\leq \sum_{k \in \set Z} \big(\sum_{k_0 \leq k} 2^{k_0 q/p}\big) 2^{-kq/p} |\alpha_k|^q \\
&= \frac {2^{q/p}}{2^{q/p}-1} \sum_{k \in \set Z} |\alpha_k|^q.
\end{aligned}\ee
It follows that $\|f\|_{L^{p, q}}$ is indeed finite and fulfills (\ref{A.17}).
\end{proof}

In order to avoid needless complications, we proved the atomic decomposition lemma only for $\set R^n$ endowed with the usual Lebesgue measure. The exact same conclusion holds for any infinitely divisible measure space.

Instead of having atoms supported on sets of size $2^k$, in a more general setting the correct condition is that the support sizes should be separated:
\be
\mu(\supp a_k) \in [2^k, 2^{k+1}).
\ee
With this minor modification, the lemma holds for any measure space.

\chapter{Real interpolation}

Our presentation of interpolation faithfully follows Bergh and L\"{o}f\-str\"{o}m, \cite{bergh}, and is included only for the sake of completeness. The reader is advised to consult this reference work for a much more detailed exposition.

For any couple of Banach spaces $(A_0, A_1)$, contained within a wider space $X$, their intersection, $A_0 \cap A_1$, and their sum
\be
A_0+A_1 = \{x \in X \mid x = a_0+a_1,\ a_0 \in A_0,\ a_1 \in A_1\}
\ee
give rise to two potentially new Banach spaces.

Given a couple of Banach spaces $(A_0, A_1)$ as above, define the so-called $K$ functional on $A_0+A_1$ by
\be
K(t, a) = \inf_{a = a_0 + a_1} (\|a_0\|_{A_0} + t \|a_1\|_{A_1}).
\ee
\begin{definition}
For $0 \leq \theta \leq 1$, $1 \leq q \leq \infty$, the interpolation space $(A_0, A_1)_{\theta, q} = A_{\theta, q}$ is the set of elements $f \in A_0 + A_1$ whose norm
\be
\|f\|_{A_{\theta, q}} = \bigg(\int_0^{\infty} (t^{-\theta} K(t, f))^q \frac {dt} t \bigg)^{1/q}
\ee
is finite.
\end{definition}

$(A_0, A_1)_{\theta, q}$ is an exact interpolation space of exponent $\theta$ between $A_0$ and $A_1$,  meaning that it satisfies the following defining property:
\begin{theorem} Let $T$ be a bounded linear mapping between two pairs of Banach spaces $(A_0, A_1)$ and $(B_0, B_1)$, i.e.
\be
\|T f\|_{B_j} \leq M_j \|f\|_{A_j},\ j = 0, 1.
\ee
Then
\be
\|T f\|_{(B_0, B_1)_{\theta, q}} \leq M_0^{1-\theta} M_1^{\theta} \|f\|_{(A_0, A_1)_{\theta, q}}.
\ee
\end{theorem}

For two couples of Banach spaces, $(A_0^{(1)}, A_1^{(1)})$ and $(A_0^{(2)}, A_1^{(2)})$,
\be
(A_0^{(1)} \times A_0^{(2)}, A_1^{(1)} \times A_1^{(2)})_{\theta, q} = (A_0^{(1)}, A_1^{(1)})_{\theta, q} \times (A_0^{(2)}, A_1^{(2)})_{\theta, q}.
\ee

The following multilinear interpolation theorem is due to Lions--Peetre:



\begin{theorem}\lb{thm_32} Assume that $T$ is a bilinear mapping from $(A_0^{(j)} \times A_1^{(j)})$ to $B^j$ for $j = 0, 1$ and
\be
\|T(a^{(1)}, a^{(2)})\|_{B_j} \leq M_j \|a^{(1)}\|_{A_j^{(1)}} \|a^{(2)}\|_{A_j^{(2)}}.
\ee
Then
\be
\|T(a^{(1)}, a^{(2)})\|_{(B_0, B_1)_{\theta, q}} \leq C \|a^{(1)}\|_{(A_0^{(1)}, A_1^{(1)})_{\theta, q_1}} \|a^{(2)}_j\|_{(A_0^{(2)}, A_1^{(2)})_{\theta, q_2}},
\ee
if $0<\theta<1$, $1/q-1=(1/q_1-1) + (1/q_2-1)$, $1 \leq q \leq \infty$.
\end{theorem}

Finally, below we list the results of real interpolation in some standard situations:
\begin{proposition}\label{prop_33} Let $L^p$ be the Lebesgue spaces defined over a measure space $(X, \mu)$ and $L^{p, q}$ be Lorentz spaces over the same. Then
\begin{enumerate}
\item $(L^{p_0, q_0}, L^{p_1, q_1})_{\theta, q} = L^{p, q}$, for $p_0$, $p_1$, $q_0$, $q_1 \in (0, \infty]$, $p_0 \ne p_1$, $1/p = (1-\theta)/p_0 + \theta/p_1$, $0<\theta<1$.
\item For a Banach space $A$, define the weighted spaces of sequences (homogenous and inhomogenous, respectively)
\be\begin{aligned}
\dot \ell^q_s(A) &= \bigg\{(a_n)_{n \in \set Z} \mid \|(a_n)\|_{\ell^q_s(A)} = \bigg(\sum_{n=-\infty}^{\infty} (2^{ns} \|a_n\|_A)^q\bigg)^{1/q} < \infty\bigg\}\\
\ell^q_s(A) &= \bigg\{(a_n)_{n \geq 0} \mid \|(a_n)\|_{\ell^q_s(A)} = \bigg(\sum_{n=0}^{\infty} (2^{ns} \|a_n\|_A)^q\bigg)^{1/q} < \infty\bigg\}
\end{aligned}\ee
Then
\be\begin{aligned}
(\dot \ell_{s_0}^{q_0}(A_0), \dot \ell_{s_1}^{q_1}(A_1))_{\theta, q} &= \dot \ell_s^q((A_0, A_1)_{\theta, q}), \\
(\ell_{s_0}^{q_0}(A_0), \ell_{s_1}^{q_1}(A_1))_{\theta, q} &= \ell_s^q((A_0, A_1)_{\theta, q}),
\end{aligned}\ee
where $0<q_0$, $q_1<\infty$, $s = (1-\theta) s_0 + \theta s_1$, $1/q = (1-\theta)/q_0 + \theta/q_1$, $0<\theta<1$.
\item $(L^{p_0}(A_0), L^{p_1}(A_1))_{\theta, p} = L^p((A_0, A_1)_{\theta, p})$, for $1/p = (1-\theta)/p_0 + \theta/p_1$.
\end{enumerate}\end{proposition}
Again, the reader is referred to \cite{bergh} for the proofs and for more details.


\end{document}